\documentclass[11pt]{article}

\usepackage[margin=1in]{geometry}
\usepackage{amsmath,amssymb,amsthm,mathtools,bm,mathrsfs}
\usepackage{enumitem}
\usepackage{booktabs}
\usepackage{array}
\usepackage{longtable}
\usepackage{capt-of}
\usepackage{microtype}
\usepackage{graphicx}
\usepackage{placeins}
\usepackage{xcolor}
\definecolor{porPurple}{HTML}{482878}
\definecolor{porIndigo}{HTML}{414487}
\definecolor{porBlue}{HTML}{355F8D}
\definecolor{porTeal}{HTML}{2A788E}
\definecolor{porGreen}{HTML}{5AAE86}
\definecolor{porLime}{HTML}{86B94A}
\definecolor{porYellow}{HTML}{C8D93D}
\definecolor{porMagenta}{HTML}{CC4778}
\definecolor{porOrange}{HTML}{C76C2B}
\usepackage{tikz}
\usetikzlibrary{arrows.meta,positioning,calc,fit,backgrounds}
\usepackage{pgfplots}
\pgfplotsset{compat=1.18}
\usepackage{url}
\usepackage[authoryear,round]{natbib}
\bibpunct{(}{)}{;}{a}{}{,}
\providecommand{\BIBand}{and}
\usepackage[unicode,hidelinks,hypertexnames=false]{hyperref}
\hypersetup{pdftitle={The Price of Relearning: Ambiguity Provenance in Dynamic Decisions},pdfauthor={Han Yan\c{c}}}
\newcommand{\ECref}[1]{\hyperref[sec:EC#1]{EC.#1}}
\usepackage{aliascnt}
\allowdisplaybreaks
\newtheorem{theorem}{Theorem}[section]
\newaliascnt{proposition}{theorem}
\newtheorem{proposition}[proposition]{Proposition}
\aliascntresetthe{proposition}
\newaliascnt{lemma}{theorem}
\newtheorem{lemma}[lemma]{Lemma}
\aliascntresetthe{lemma}
\newaliascnt{corollary}{theorem}
\newtheorem{corollary}[corollary]{Corollary}
\aliascntresetthe{corollary}
\theoremstyle{definition}
\newaliascnt{definition}{theorem}
\newtheorem{definition}[definition]{Definition}
\aliascntresetthe{definition}
\newaliascnt{example}{theorem}

\aliascntresetthe{example}
\newaliascnt{assumption}{theorem}
\newtheorem{assumption}[assumption]{Assumption}
\aliascntresetthe{assumption}
\theoremstyle{remark}
\newaliascnt{remark}{theorem}

\aliascntresetthe{remark}
\usepackage[nameinlink,capitalize,noabbrev]{cleveref}
\crefname{theorem}{Theorem}{Theorems}
\crefname{proposition}{Proposition}{Propositions}
\crefname{lemma}{Lemma}{Lemmas}
\crefname{corollary}{Corollary}{Corollaries}
\crefname{definition}{Definition}{Definitions}
\crefname{example}{Example}{Examples}
\crefname{assumption}{Assumption}{Assumptions}
\crefname{remark}{Remark}{Remarks}

\newcommand{\R}{\mathbb{R}}

\newcommand{\Pp}{\mathbb{P}}
\newcommand{\E}{\mathbb{E}}

\newcommand{\A}{\mathcal{A}}

\newcommand{\F}{\mathcal{F}}

\newcommand{\Hh}{\mathcal{H}}
\newcommand{\ClrH}{\mathsf{H}^{\circ}}
\newcommand{\K}{\mathcal{K}}

\newcommand{\M}{\mathcal{M}}
\newcommand{\Q}{\mathcal{Q}}

\newcommand{\V}{\mathcal{V}}
\newcommand{\X}{\mathcal{X}}
\newcommand{\Y}{\mathcal{Y}}

\newcommand{\clco}{\overline{\operatorname{co}}}
\newcommand{\argmax}{\operatorname*{arg\,max}}

\newcommand{\dist}{\operatorname{dist}}

\newcommand{\Rect}{\operatorname{Rect}}
\newcommand{\supp}{\operatorname{supp}}
\newcommand{\clr}{\operatorname{clr}}

\newcommand{\TV}{\operatorname{TV}}
\newcommand{\PoR}{\operatorname{PoR}}

\newcommand{\Dop}{D_{\mathrm{op}}}
\newcommand{\Dpred}{D_{\mathrm{pred}}}
\newcommand{\Dctrl}{D_{\mathrm{ctrl}}}
\newcommand{\Deltao}{\Delta_m^\circ}
\newcommand{\normalize}[1]{\mathsf{N}\!\left(#1\right)}
\newcommand{\fresh}{\mathrm{F}}
\newcommand{\inh}{\mathrm{I}}

\newcommand{\spe}{\mathrm{SPE}}

\newcommand{\softmax}{\operatorname{softmax}}

\title{\textbf{The Price of Relearning:}\\
Ambiguity Provenance in Dynamic Decisions}
\author{Han Yan\c{c}\\
School of Physical and Mathematical Sciences\\
Nanyang Technological University, Singapore}
\date{}

\begin{document}

\maketitle
\vspace{-3em}

\begin{abstract}
Dynamic robust systems often rebuild ambiguity sets as data arrive. Relearning can then replace the evaluator rather than merely update its beliefs; we call this dependence on an uncertainty set's date of origin ambiguity provenance. Under weak-evidence richness, a likelihood quotient exactly characterizes continuous compact-valued reconstruction rules that preserve Bayesian provenance. When compatibility fails and the discrepancy is decision-visible, sophisticated behavior admits an exact triangular evaluator-vintage representation before operationally redundant vintages are quotiented out. The directed welfare loss from later reconstruction is the Price of Relearning. We characterize when vintage state can be compressed, the divide between polynomial evaluation of specified finite models and NP-hard universal certification, and a compatible reconstruction that repairs the protocol. A Gaussian pricing benchmark closes the reconstruction--action--information loop; scanner data calibrate its scale without a causal protocol claim.
\end{abstract}

\noindent\textbf{Keywords:} ambiguity provenance; Price of Relearning; robust sequential decisions.

\section{Introduction}\label{sec:intro}

An organization can update every probability correctly and still change the criterion by which its earlier plan is judged. A retailer revises demand, a platform revises response, or a service system revises arrivals and rebuilds an ambiguity set around the new estimate. The update may be statistically coherent at every date, yet the reconstructed set need not be the Bayesian continuation of the models that justified the earlier action. Relearning can replace the evaluator rather than merely update its beliefs.

Prior-by-prior continuation is the standard multiple-prior Bayesian benchmark \citep{Pires2002,Kovach2024}; we call it \emph{Bayesian inheritance}, and call reapplication of a statistical constructor \emph{fresh reconstruction}. The distinction concerns the \emph{provenance} of the uncertainty object, not a failure of Bayes' rule or a change in preferences. The two-state example in \cref{sec:model} already produces opposite robust choices after the same signal.

\paragraph{Ambiguity Provenance Principle.}
\emph{A statistically coherent fresh reconstruction can replace the inherited evaluator. When that replacement is decision-visible, evaluator vintage becomes a dynamic state variable; its directed welfare consequence is the Price of Relearning.} The paper asks when provenance is preserved, what state is needed when it is not, whether that state can be discarded, and how reconstruction can be repaired.

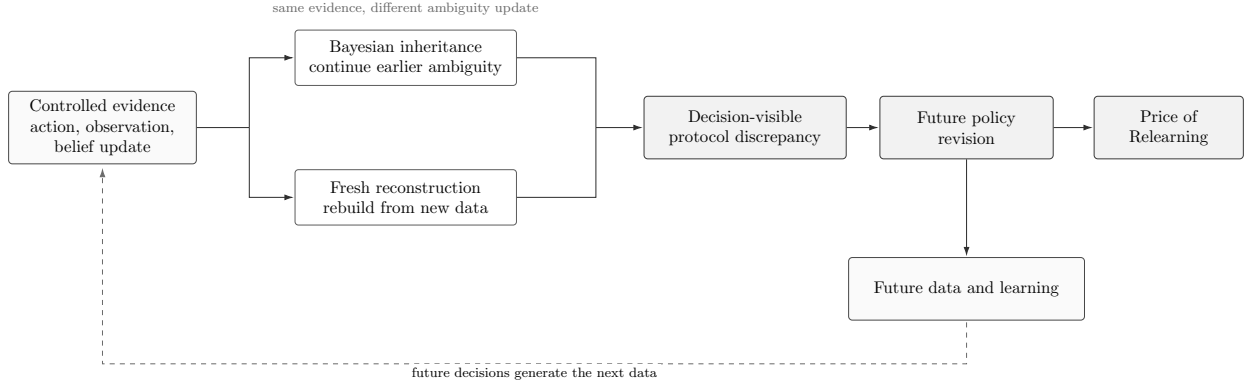
\begin{figure}[t]
\centering
\resizebox{0.99\textwidth}{!}{%
\begin{tikzpicture}[
  >=Latex,
  every node/.style={font=\small},
  box/.style={draw=black!72, rounded corners=2.2pt, align=center,
             minimum height=13mm, inner xsep=6pt, inner ysep=5pt,
             fill=black!2, line width=0.45pt},
  protocol/.style={draw=black!64, rounded corners=2pt, align=center,
                  minimum width=46mm, minimum height=11.5mm,
                  inner xsep=5pt, inner ysep=4.5pt, fill=white, line width=0.45pt},
  outcome/.style={draw=black!78, rounded corners=2.2pt, align=center,
                 minimum height=13mm, inner xsep=6pt, inner ysep=5pt,
                 fill=black!5, line width=0.5pt},
  flow/.style={semithick, draw=black!82},
  arr/.style={-{Latex[length=2.0mm,width=1.25mm]}, semithick, draw=black!82,
             shorten <=0pt, shorten >=0.35pt},
  fb/.style={-{Latex[length=2.0mm,width=1.25mm]}, semithick, dashed, draw=black!58,
            shorten <=1.1pt, shorten >=2.0pt}
]
\node[box, minimum width=39mm] (evidence) at (0,0)
  {Controlled evidence\\action, observation,\\belief update};

\coordinate (split) at (3.05,0);
\node[protocol] (inherit) at (6.30,1.45)
  {Bayesian inheritance\\continue earlier ambiguity};
\node[protocol] (fresh) at (6.30,-1.45)
  {Fresh reconstruction\\rebuild from new data};

\coordinate (merge) at (10.25,0);
\node[outcome, minimum width=42mm] (visible) at (13.35,0)
  {Decision-visible\\protocol discrepancy};
\node[outcome, minimum width=36mm] (revision) at (17.95,0)
  {Future policy\\revision};
\node[outcome, minimum width=31mm] (por) at (22.15,0)
  {Price of\\Relearning};

\node[box, minimum width=49mm] (learning) at (17.95,-3.35)
  {Future data and learning};

\draw[flow] (evidence.east) -- (split);
\draw[arr] (split) |- (inherit.west);
\draw[arr] (split) |- (fresh.west);

\draw[flow] (inherit.east) -| (merge);
\draw[flow] (fresh.east) -| (merge);
\draw[arr] (merge) -- (visible.west);

\draw[arr] (visible.east) -- (revision.west);
\draw[arr] (revision.east) -- (por.west);
\draw[arr] (revision.south) -- (learning.north);

\coordinate (fbRight) at ($(learning.south)+(0,-9mm)$);
\coordinate (fbLeft) at (evidence.south |- fbRight);
\draw[fb] (learning.south) -- (fbRight)
  -- node[midway,below,font=\scriptsize,fill=white,inner sep=1.8pt]
  {future decisions generate the next data}
  (fbLeft) -- (evidence.south);

\node[font=\scriptsize, text=black!60, align=center]
  at (6.30,2.48) {same evidence, different ambiguity update};
\end{tikzpicture}%
}
\caption{Ambiguity provenance mechanism.  The same evidence generates Bayesian inheritance or fresh reconstruction; only a decision-visible discrepancy revises policy, creates the Price of Relearning, and feeds back into future data.}
\label{fig:conceptual-mechanism}
\end{figure}

The four questions form one hierarchy. \emph{When is provenance preserved?} \Cref{thm:likelihood-quotient-characterization} gives the likelihood-quotient characterization. \emph{What state is required otherwise, and what does replacement cost?} \Cref{thm:triangular-v2,thm:exact-dynamic-por-representation} give the evaluator-vintage recursion and exact Price-of-Relearning representation. \emph{Can that state be discarded?} Vintage quotients give positive reductions; endpoint and memory results rule out the stated low-complexity representations uniformly over the unrestricted class; explicitly listed finite models remain polynomially evaluable, whereas universal worst-reward certification is NP-hard. \emph{What should be done?} Compatible design restricts reconstruction to provenance-preserving protocols while accounting for the induced sophisticated policy.

A conditional myopic Gaussian benchmark closes the chain: fresh reconstruction can lower price, reduce Fisher information, and preserve the uncertainty that triggers later reconstruction. Scanner data calibrate the channel's scale; the identification result states why observational histories alone do not identify a causal protocol effect.

\paragraph{Relation to the literature.}
Bayesian updating, multiple-prior preferences, rectangularity, sophisticated behavior, and robust dynamic programming are established \citep{GilboaSchmeidler1989,Pires2002,EpsteinSchneider2003,Strotz1955,Iyengar2005}. Dynamic risk measures and robust control study consistency or learning within specified intertemporal evaluation architectures \citep{MorescoMailhotPesenti2024,EpsteinJi2020,LinRenZhou2022,ShapiroZhouLinWang2025,MaChenXu2026BCR}; Bayesian-DRO and learned uncertainty sets reconstruct a current robust criterion from data \citep{ShapiroZhouLin2023,DellaportaOHaraDamoulas2025,WangVanParysStellato2023}. Our question is their interface: when does repeated reconstruction reproduce prior-by-prior Bayesian inheritance, and what follows when it does not? Adjacent 2026 work studies posterior-updated ambiguity in control and portfolio models \citep{MaChenXuZhou2026,LiangLiuMa2026}, ambiguity learning and ex-ante/interim statistical consistency \citep{Lim2026,LimPark2026}, learned or stationary ambiguity \citep{Guo2026,MuellerAkkariWoodGonon2026}, preference-augmented state reduction \citep{Haskell2026}, and risk measures compatible with dynamic programming under epistemic uncertainty \citep{BenyamineGrandClementPetrikJordanDurmus2026}. \citet{ParkPark2026} recover latent uncertainty from dynamic sublinear valuations. Our object is different: for the same evidence path, we compare fresh reconstruction with the Bayesian image of an earlier ambiguity set and trace any mismatch through evaluator vintage, equilibrium, welfare, representation, certification, and compatible repair. Standard state abstraction remains complementary \citep{GivanDeanGreig2003,LiWalshLittman2006}.

\section{Controlled Learning and Ambiguity Vintages}\label{sec:model}

We formalize the one distinction used throughout: an earlier ambiguity set may be continued prior by prior, or a constructor may be reapplied to the same updated statistical state. The first preserves provenance by construction; the second need not.

\subsection{Finite parameter model and Bayesian update}

Let \(\Theta=\{1,\ldots,m\}\), \(m\ge2\), and define
\[
\Delta_m:=\left\{p\in\R_+^m:\sum_{i=1}^mp_i=1\right\},
\qquad
\Deltao:=\{p\in\Delta_m:p_i>0\ \forall i\}.
\]
For \(x\in\R_+^m\setminus\{0\}\), write
\[
\normalize{x}:=\frac{x}{\sum_i x_i}.
\]
For \(p\in\Delta_m\) and a nonnegative likelihood vector \(\lambda\in\R_+^m\) with \(p^\top\lambda>0\), define the posterior
\begin{equation}\label{eq:bayes}
B_\lambda(p):=\normalize{(p_i\lambda_i)_{i=1}^m}.
\end{equation}
Thus strictly positive likelihoods define \(B_\lambda\) on all of \(\Delta_m\), and map \(\Deltao\) into \(\Deltao\). For any map \(F\) and any set \(A\) contained in its domain, write
\[
F^{\#}[A]:=\{F(q):q\in A\}
\]
for its set image.
A fresh ambiguity constructor at date \(t\) is a nonempty-valued map
\[
\Phi_t:\Deltao\rightrightarrows\Deltao.
\]
Regularity is result specific: whenever Hausdorff geometry is used, the relevant constructor images are assumed nonempty and compact, and Hausdorff continuity is imposed only where stated.
Fix a nominal prior \(\mu_0\in\Deltao\).  Along a realized path with likelihood vectors \(\lambda_1,\lambda_2,\ldots\), define the nominal posterior recursively by
\[
\mu_{u+1}:=B_{\lambda_{u+1}}(\mu_u)
\]
whenever the Bayes denominator is positive.  We require \(\mu_t\in\Deltao\) whenever fresh reconstruction \(\Phi_t(\mu_t)\) is invoked; strictly positive likelihood vectors along every feasible history are sufficient.

Recent preprints propose posterior-centered or data-adaptive ambiguity constructions in Bayesian and decision-focused distributionally robust optimization \citep{MaChenXuZhou2026,Guo2026}. If date \(s\) starts from \(\M_s^{\fresh}=\Phi_s(\mu_s)\), let \(B_{s:t}\) denote the composition of the realized Bayes maps on its natural domain
\[
\mathsf D_{s:t}:=\operatorname{dom}(B_{s:t})
=\{p\in\Delta_m:\text{every Bayes update along the realized path is well defined}\}.
\]
At every inherited evaluation in this parameter-generated protocol we require the survivor set \(\M_s^{\fresh}\cap\mathsf D_{s:t}\) to be nonempty, and define the inherited ambiguity by the prior-by-prior Bayesian image
\begin{equation}\label{eq:survivor-bayes-image}
\M_{s,t}^{\inh}:=B_{s:t}^{\#}\!\left[\M_s^{\fresh}\cap\mathsf D_{s:t}\right].
\end{equation}
Under strictly positive likelihoods, \(\mathsf D_{s:t}=\Delta_m\). For the interior fresh-constructor geometry below, survival is therefore automatic along every admissible path. The partial-domain notation is retained only for boundary/static-parameter continuations and null-history scope. Prior-by-prior Bayesian updating is classical \citep{Pires2002,Kovach2024}; data-revised sets distinct from this benchmark are studied by \citet{Cheng2026}. The fresh date-\(t\) set is
\[
\M_t^{\fresh}:=\Phi_t(\mu_t).
\]
The cross-vintage object is the pair \((\M_{s,t}^{\inh},\M_t^{\fresh})\).  These domain conditions belong only to the finite-parameter reconstruction device in this section; the later vintage-kernel theory takes its nonempty kernel correspondences as primitives and does not impose full support.

\subsection{Finite controlled event tree}

Time is \(t=0,\ldots,T\). The set of histories at date \(t\) is finite and denoted \(\Hh_t\). At \(h_t\), the feasible action set \(\A_t(h_t)\) is finite and nonempty. The next observation lies in a finite set \(\Y_{t+1}(h_t,a)\). Under parameter \(i\), its conditional law is
\[
P_i(\cdot\mid h_t,a)\in\Delta(\Y_{t+1}(h_t,a)).
\]
For the controlled tree, the sufficient full-support condition above is \(P_i(y\mid h,a)>0\) for every parameter \(i\), feasible \((h,a)\), and \(y\in\Y_{t+1}(h,a)\).  Without it, \eqref{eq:survivor-bayes-image} restricts inheritance to priors assigning the realized path positive probability, as formalized again in \cref{prop:static-parameter-rect}.  Off-path equilibrium kernels at histories not reached by a candidate path law are specified directly by the rectangular correspondences introduced in \cref{sec:equilibrium}, rather than inferred from a null-event conditional.
The stage reward is \(r_t(h_t,a)\), the terminal reward is \(g(h_T)\), and \(\beta\in(0,1]\). We assume
\begin{equation}\label{eq:reward-bounds}
|r_t(h,a)|\le R_t,
\qquad
|g(h_T)|\le G.
\end{equation}
Define continuation bounds recursively by
\begin{equation}\label{eq:B-bounds}
B_T:=G,
\qquad
B_t:=R_t+\beta B_{t+1}.
\end{equation}

For \((h_t,a)\), stack the parameter-specific predictive laws into the stochastic matrix
\[
P_{h_t,a}:=
\begin{pmatrix}
P_1(\cdot\mid h_t,a)\\
\vdots\\
P_m(\cdot\mid h_t,a)
\end{pmatrix}.
\]
A parameter distribution \(\nu\in\Delta_m\) induces the one-step predictive law \(\nu P_{h_t,a}\). Hence a parameter ambiguity set \(M\subseteq\Delta_m\) induces
\begin{equation}\label{eq:predictive-push}
\mathscr P_{h_t,a}(M):=\{\nu P_{h_t,a}:\nu\in M\}.
\end{equation}
At a given history, the inherited and fresh \emph{admissible one-step mixture sets} are
\begin{equation}\label{eq:kernel-from-parameter}
\K^{\mathrm{mix}}_{s,t}(h_t,a):=\mathscr P_{h_t,a}(\M_{s,t}^{\inh}),
\qquad
\K^{\mathrm{mix}}_{t,t}(h_t,a):=\mathscr P_{h_t,a}(\M_t^{\fresh}).
\end{equation}
A crucial semantic distinction is that a fixed but unknown parameter is drawn once and therefore couples future transitions across histories. The path-law family generated by a single parameter mixture is generally nonrectangular; independently selecting kernels from the local sets in \eqref{eq:kernel-from-parameter} produces its rectangularization rather than the original fixed-parameter model. \Cref{sec:dynamic} formalizes this committed-versus-rectangular distinction. The bridge in \eqref{eq:kernel-from-parameter} should therefore be read as a statement about local conditional kernels, not as an assertion that static parameter ambiguity is automatically rectangular.

\subsection{A minimal reconstruction paradox}

The phenomenon does not require nonadditive utility or nonexponential discounting. Let \(\Theta=\{L,H\}\), nominal prior \(p_0=1/2\), and initial ambiguity \([0.3,0.7]\). A signal satisfies
\[
\Pp(Y=1\mid H)=0.8,
\qquad
\Pp(Y=1\mid L)=0.2.
\]
After \(Y=1\), the nominal posterior is \(0.8\). Prior-by-prior updating sends the initial interval to
\[
\left[
\frac{0.8(0.3)}{0.8(0.3)+0.2(0.7)},
\frac{0.8(0.7)}{0.8(0.7)+0.2(0.3)}
\right]
\approx[0.6316,0.9032].
\]
Suppose the next self instead reconstructs \([0.7,0.9]\). For a risky payoff \((1,-1)\) and safe payoff \(0.3\), the inherited evaluator assigns risky value approximately \(0.2632\), whereas the fresh evaluator assigns value \(0.4\). Thus the earlier self plans safe and the future self chooses risky. The remainder of the paper identifies the geometry, observability, equilibrium representation, and welfare cost behind this reversal.

\section{Preserving Ambiguity Provenance}\label{sec:geometry}

We begin with compatibility: when does fresh reconstruction preserve the inherited evaluator? In centered log-ratio coordinates, Bayesian updating is translation, so the problem becomes restricted equivariance.

\subsection{Bayesian updating as translation}

For \(p,q\in\Deltao\), define Aitchison perturbation \citep{Aitchison1986,EgozcueDiazPawlowsky2006}
\begin{equation}\label{eq:oplus}
p\oplus q:=\normalize{(p_iq_i)_{i=1}^m}.
\end{equation}
The neutral element is \(u=(1/m,\ldots,1/m)\). Let
\[
\ClrH:=\left\{x\in\R^m:\mathbf 1^\top x=0\right\}
\]
and define the centered log-ratio map
\begin{equation}\label{eq:clr-v2}
\clr(p)_i:=\log p_i-\frac1m\sum_{j=1}^m\log p_j.
\end{equation}
It is a bijection from \(\Deltao\) to \(\ClrH\), with inverse \(\softmax(x)=\normalize{(e^{x_i})_i}\).

\begin{lemma}[Bayes translation]\label{lem:bayes-translation-v2}
For \(p\in\Deltao\) and \(\lambda\in\R_{++}^m\), let \(q_\lambda=\normalize{\lambda}\). Then
\[
B_\lambda(p)=p\oplus q_\lambda,
\qquad
\clr(p\oplus q)=\clr(p)+\clr(q).
\]
\end{lemma}

\noindent\emph{Proof.} Normalization cancels in the componentwise product, and centering removes the common log normalizer. \hfill\(\square\)

The operation is classical in compositional geometry and Bayes Hilbert spaces \citep{Aitchison1986,EgozcueDiazPawlowsky2006,VanDenBoogaartEgozcuePawlowsky2014}; we use it only to expose the reconstruction symmetry.

\subsection{Full and restricted likelihood equivariance}

For \(A\subseteq\Deltao\), write \(A\oplus q=\{a\oplus q:a\in A\}\).

\begin{definition}[Full Bayes equivariance]\label{def:full-equiv}
A nonempty-valued constructor \(\Phi:\Deltao\rightrightarrows\Deltao\) is fully Bayes-equivariant if
\begin{equation}\label{eq:full-equiv}
\Phi(p\oplus q)=\Phi(p)\oplus q
\qquad\forall p,q\in\Deltao.
\end{equation}
\end{definition}

\begin{proposition}[Full-orbit representation]\label{thm:full-representation-v2}
A constructor \(\Phi\) is fully Bayes-equivariant if and only if there exists a unique nonempty set \(K\subseteq\Deltao\) such that
\begin{equation}\label{eq:fixed-shape}
\Phi(p)=K\oplus p
\qquad\forall p\in\Deltao.
\end{equation}
The shape is \(K=\Phi(u)\). Equivalently, with \(C=\clr(K)\),
\[
\clr(\Phi(p))=\clr(p)+C.
\]
\end{proposition}

\noindent\emph{Proof.} Set \(K=\Phi(u)\); equivariance gives \(\Phi(p)=K\oplus p\), uniqueness follows at \(u\), and the converse follows by associativity. \hfill\(\square\)

The transitive case is elementary.  A statistical model may, however, generate only a restricted set $L$ of centered log-likelihood increments.  EC.1 proves the corresponding orbit theorem: exact equivariance fixes the translated ambiguity shape along every coset of the additive subgroup generated by $L$.  If that subgroup is dense, Hausdorff continuity restores the fixed-shape conclusion; when the generated closed subgroup is proper---equivalently, when the likelihood quotient is nontrivial---shape variation can survive only in likelihood-unidentified quotient degrees of freedom.  This extension clarifies scope without interrupting the cross-time argument used below.

\subsection{Cross-time compatibility and the no-free-shrinkage boundary}

\begin{definition}[Full cross-time compatibility]\label{def:cross-time-v2}
A family \(\{\Phi_t\}_{t=0}^T\) is fully cross-time compatible if
\begin{equation}\label{eq:cross-time-v2}
\Phi_{t+1}(p\oplus q)=\Phi_t(p)\oplus q
\qquad\forall t<T,\ p,q\in\Deltao.
\end{equation}
\end{definition}

\begin{proposition}[No-free-shrinkage under full compatibility]\label{thm:no-free-shrinkage-v2}
A nonempty-valued family is fully cross-time compatible if and only if there exists one nonempty set \(K\subseteq\Deltao\) such that
\[
\Phi_t(p)=K\oplus p
\qquad\forall t\in\{0,\ldots,T\},\ p\in\Deltao.
\]
Consequently, for any log-ratio norm \(\|\cdot\|_*\), a ball family
\[
\Phi_t(p)=\{q:\|\clr(q)-\clr(p)\|_*\le r_t\}
\]
is fully cross-time compatible if and only if \(r_t\) is constant.
\end{proposition}

\noindent\emph{Proof.} The identity update makes the constructor time independent; \cref{thm:full-representation-v2} then applies. The converse and the ball statement are immediate. \hfill\(\square\)

Full compatibility is stronger than model-generated evidence. The next result gives the exact restricted-evidence structure and identifies when concentration forces degeneracy.

Let $L\subseteq\ClrH$ be the admissible centered log-likelihood increments, let $\Gamma(L)$ be the additive subgroup generated by $L$ (with $\Gamma(\varnothing)=\{0\}$), and set $G_L:=\overline{\Gamma(L)}$.  Write $\pi_L:\ClrH\to\ClrH/G_L$ for the quotient map and
$\Psi_t(x):=\clr(\Phi_t(\softmax x))$.  We equip $\ClrH/G_L$ with its quotient topology; equivalently, for any norm $\|\cdot\|_*$ on $\ClrH$, one may use the compatible metric
\[
d_{Q,*}(\pi_Lx,\pi_Ly):=\inf_{g\in G_L}\|x-y-g\|_*.
\]
Nonempty compact subsets of $\ClrH$ are endowed with the Hausdorff metric induced by the same norm.

For a centered log-likelihood increment $\ell\in\ClrH$, define the associated Bayes map
\[
\mathsf B_\ell:=B_{\exp(\ell)}:\Delta_m\to\Delta_m.
\]
On $\Deltao$, $\clr(\mathsf B_\ell p)=\clr(p)+\ell$ by \cref{lem:bayes-translation-v2}.

\begin{definition}[Restricted cross-time evidence compatibility]\label{def:restricted-cross-time}
The family $\{\Phi_t\}_{t\ge0}$ is $L$-compatible if
\begin{equation}\label{eq:restricted-cross-time}
\Psi_{t+1}(x+\ell)=\Psi_t(x)+\ell
\qquad\forall t\ge0,\ x\in\ClrH,\ \ell\in L.
\end{equation}
The condition compares reconstruction only along likelihood increments admitted by the statistical experiment.
\end{definition}
We call the condition $0\in\overline L$ \emph{weak-evidence richness}; equivalently, admissible centered log-likelihood increments can approach zero.

\begin{theorem}[Likelihood-quotient characterization]\label{thm:likelihood-quotient-characterization}
Assume $0\in\overline L$---equivalently, there exists a sequence $\ell_n\in L$ with $\ell_n\to0$, including the case $0\in L$---and that every $\Psi_t$ is nonempty compact-valued and Hausdorff-continuous.  The family is $L$-compatible if and only if there is a unique Hausdorff-continuous compact-valued map
\[
\mathcal C:\ClrH/G_L\rightrightarrows\ClrH
\]
such that
\begin{equation}\label{eq:likelihood-quotient-representation}
\boxed{\Psi_t(x)=x+\mathcal C(\pi_Lx)\qquad\forall t\ge0,\ x\in\ClrH.}
\end{equation}
Thus time variation disappears when arbitrarily weak evidence is admissible, while shape variation can survive exactly in likelihood-unidentified quotient degrees of freedom.

Let $\mathcal A\subseteq\ClrH$ be a set of consistency anchors: for each $x^\star\in\mathcal A$ there is a sequence $x_t\to x^\star$ such that
\[
d_{H,*}(\Psi_t(x_t),\{x^\star\})\to0.
\]
Then $\mathcal C(\pi_Lx^\star)=\{0\}$ for every anchor.  Moreover, the anchor condition forces \emph{every} continuous $L$-compatible family to be the singleton constructor if and only if
\begin{equation}\label{eq:anchor-density}
\overline{\pi_L(\mathcal A)}=\ClrH/G_L.
\end{equation}
\end{theorem}

\noindent\emph{Proof idea.} Let $S_t(x)=\Psi_t(x)-x$.  Taking admissible increments converging to zero in \eqref{eq:restricted-cross-time} gives $S_{t+1}=S_t$.  The common residual is invariant under $\Gamma(L)$ and, by continuity, under $G_L$, so it factors uniquely through the quotient, proving \eqref{eq:likelihood-quotient-representation}; the converse is immediate.  Concentration at an anchor forces its quotient shape to be $\{0\}$.  Density and continuity then force every quotient shape to vanish.  Conversely, if the projected anchors are not dense, a nondegenerate compatible constructor survives: when there are no anchors take a fixed nonzero compact segment, and otherwise multiply that segment by the distance to $\overline{\pi_L(\mathcal A)}$.  EC.1 gives the complete factorization and converse construction.

The condition $0\in\overline L$ means admissible evidence can approach an uninformative likelihood ratio; it can fail in a fixed discrete experiment whose nonzero increments are bounded away from zero. Thus the rigidity result is conditional on weak-evidence richness.

This characterization is distinct from a recent preprint on \emph{stationary ambiguity}, which prevents latent uncertainty from systematically vanishing in robust-control simulators \citep{MuellerAkkariWoodGonon2026}.  Here the quotient records precisely which statistical likelihood directions can discipline intertemporal reconstruction.

\subsection{Why statistical concentration cannot be dynamically free}

The preceding rigidity result becomes economically sharper once a reconstruction rule is required to learn.  For this subsection only, allow an indefinitely extendable sequence of constructors; any finite-horizon model is a truncation of such a sequence.

\begin{definition}[Log-ratio point consistency]\label{def:logratio-consistency}
A compact-valued constructor sequence $\{\Phi_t\}_{t\ge0}$ is \emph{log-ratio point-consistent along} $p_t\to p^\star\in\Deltao$ if
\[
d_{H,*}\!\left(\clr(\Phi_t(p_t)),\{\clr(p^\star)\}\right)\longrightarrow0
\]
for some norm $\|\cdot\|_*$ on $\ClrH$.  This is Hausdorff concentration of the entire ambiguity set at an interior limiting model, not merely consistency of its center.
\end{definition}

\begin{corollary}[Compatibility--concentration trilemma]\label{thm:compatibility-concentration}
Suppose the admissible likelihood directions are topologically generating, $G_L=\ClrH$, and the family satisfies the assumptions of \cref{thm:likelihood-quotient-characterization}.  If it is log-ratio point-consistent along one sequence $p_t\to p^\star\in\Deltao$, then
\[
\boxed{\Phi_t(p)=\{p\}\qquad\forall t\ge0,\ p\in\Deltao.}
\]
Thus nondegenerate ambiguity, exact compatibility along a topologically generating evidence class, and whole-set concentration at an interior truth cannot coexist.  Full cross-time compatibility is the special case in which every centered likelihood increment is admissible.
\end{corollary}

\noindent\emph{Proof.} Here $\ClrH/G_L$ is a singleton, so one anchor is dense and \cref{thm:likelihood-quotient-characterization} forces $\mathcal C=\{0\}$. \hfill\(\square\)

Concentration is not undesirable; the result says only that contraction must register as a cross-vintage protocol change unless ambiguity was absent. Statistically motivated adaptive sets can therefore fail exact inheritance compatibility \citep{MaChenXuZhou2026,LiangLiuMa2026,Guo2026,WangVanParysStellato2023}. By the triangle inequality, any concentrating shape sequence accumulates at least $d_{H,*}(C_s,\{0\})$ total Hausdorff movement from date $s$ onward, with equality for monotone concentric balls.

Whether this geometric movement reaches welfare depends on prediction and control. The following identity supplies the first link.  Let \(C_t\subseteq\ClrH\) be nonempty compact sets and define
\begin{equation}\label{eq:shape-family}
\Phi_t(p):=\softmax(\clr(p)+C_t)
:=\{\softmax(\clr(p)+c):c\in C_t\}.
\end{equation}

\begin{proposition}[Exact geometric and operational reconstruction defects]\label{thm:shape-defect}
Fix a norm \(\|\cdot\|_*\) on \(\ClrH\) and let \(d_{H,*}\) be its Hausdorff distance. After any Bayes increment \(\ell\in\ClrH\), the inherited and fresh shapes in \(\clr\)-coordinates satisfy
\begin{equation}\label{eq:exact-shape-defect}
d_{H,*}\bigl(\clr(\mathsf B_\ell^{\#}[\Phi_t(p)]),\clr(\Phi_{t+1}(\mathsf B_\ell p))\bigr)
=d_{H,*}(C_t,C_{t+1}).
\end{equation}
Let \(c_*<\infty\) satisfy \(\|x\|_\infty\le c_*\|x\|_*\). Then the parameter-space operational defect obeys
\begin{equation}\label{eq:operational-shape-bound}
\Dop\bigl(\mathsf B_\ell^{\#}[\Phi_t(p)],\Phi_{t+1}(\mathsf B_\ell p)\bigr)
\le 2c_*\,d_{H,*}(C_t,C_{t+1}).
\end{equation}
For concentric norm balls \(C_t=r_t\mathbb B_*\), the geometric defect equals \(|r_{t+1}-r_t|\).
\end{proposition}

\noindent\emph{Proof idea.} Bayesian updating translates both shapes by the same log-likelihood increment, so translation invariance gives the exact Hausdorff identity.  The probability-coordinate bound follows from the global estimate $\|\softmax(x)-\softmax(y)\|_1\le 2c_*\|x-y\|_*$, followed by Hausdorff contraction, convexification, and the operational duality in \cref{thm:operational-dual-v2}.  EC.1 supplies the Jacobian calculation and all set-valued details.

\section{When Provenance Reaches Decisions}\label{sec:bridge}

A geometric reconstruction defect matters only if it survives the predictive channel and reaches a continuation payoff that the control problem can realize. This section turns the provenance defect into an operational, decision-visible quantity.

\subsection{A decision-operational metric}

Let \(A\subseteq\Delta_m\) be nonempty and compact. For a parameter payoff vector \(z\in\R^m\), define the lower-envelope/maxmin evaluation
\[
\rho_A(z):=\inf_{p\in A}p^\top z.
\]
This is the finite-state lower-expectation object underlying maxmin expected utility \citep{GilboaSchmeidler1989}.  Since the objective is linear, \(\rho_A=\rho_{\clco A}\).

\begin{definition}[Parameter operational defect]\label{def:Dop-v2}
For nonempty compact \(A,B\subseteq\Delta_m\), define
\begin{equation}\label{eq:Dop-v2}
\Dop(A,B):=\sup_{\|z\|_\infty\le1}|\rho_A(z)-\rho_B(z)|.
\end{equation}
\end{definition}

\begin{proposition}[Exact dual representation]\label{thm:operational-dual-v2}
For nonempty compact \(A,B\subseteq\Delta_m\),
\begin{equation}\label{eq:dual-v2}
\Dop(A,B)=d_H^{\ell_1}(\clco A,\clco B).
\end{equation}
Equivalently, \(\Dop(A,B)=2d_H^{\TV}(\clco A,\clco B)\) under \(d_{\TV}(p,q)=\frac12\|p-q\|_1\).
\end{proposition}

\noindent\emph{Proof idea.} Linear lower expectations depend only on closed convex hulls and satisfy $\rho_A(z)=-h_{\clco A}(-z)$.  The dual formula for distance to a compact convex set identifies each directed Hausdorff excess with a one-sided support-function difference over the $\ell_\infty$ unit ball.  Taking the larger directed excess gives \eqref{eq:dual-v2}; see \citet{RockafellarWets1998}.  EC.2 gives the complete duality calculation.

\begin{corollary}[Decision equivalence and behavioral separation]\label{cor:decision-equivalence-v2}
The following are equivalent: $\Dop(A,B)=0$; the two lower expectations agree on every payoff; and $\clco A=\clco B$.  If $\Dop(A,B)>0$, a bounded risky payoff and a constant safe payoff can be chosen so that the two evaluators rank them strictly in opposite orders.
\end{corollary}

\noindent\emph{Proof.} The equivalences follow from \cref{thm:operational-dual-v2}; a constant strictly between two unequal lower expectations gives separation. \hfill\(\square\)

\subsection{The parameter-to-predictive bridge}

Fix a stochastic matrix \(P\in\R^{m\times n}\), whose \(i\)th row is the next-observation law under parameter \(i\). Define the linear mixture map
\[
T_P(p):=pP,
\qquad
T_P(A):=\{pP:p\in A\}.
\]

\begin{definition}[Predictive operational defect]\label{def:Dpred}
For compact parameter ambiguity sets \(A,B\subseteq\Delta_m\), define
\begin{equation}\label{eq:Dpred}
\Dpred^P(A,B)
:=\sup_{\|v\|_\infty\le1}
\left|
\inf_{p\in A}pPv-
\inf_{q\in B}qPv
\right|.
\end{equation}
\end{definition}

\begin{proposition}[Parameter-to-predictive bridge]\label{thm:parameter-predictive-bridge}
For nonempty compact \(A,B\subseteq\Delta_m\),
\begin{equation}\label{eq:Dpred-exact}
\Dpred^P(A,B)
=\Dop(T_P(A),T_P(B))
=d_H^{\ell_1}\bigl(T_P(\clco A),T_P(\clco B)\bigr).
\end{equation}
Moreover,
\begin{equation}\label{eq:bridge-contraction}
\Dpred^P(A,B)\le\Dop(A,B).
\end{equation}
Define the predictive identification modulus
\begin{equation}\label{eq:alphaP}
\alpha(P):=
\inf_{\substack{x\in\ClrH\\\|x\|_1=1}}
\|xP\|_1.
\end{equation}
Then
\begin{equation}\label{eq:bridge-bilipschitz}
\alpha(P)\Dop(A,B)
\le\Dpred^P(A,B)
\le\Dop(A,B).
\end{equation}
Finally, \(\alpha(P)>0\) if and only if \(xP=0\) with \(x\in\ClrH\) implies \(x=0\), equivalently the rows of \(P\) are affinely independent.
\end{proposition}

\noindent\emph{Proof idea.} Pushing a parameter mixture through $P$ converts every parameter payoff $Pv$ into a predictive payoff, so the exact identity follows from \cref{thm:operational-dual-v2}.  Stochasticity makes $x\mapsto xP$ an $\ell_1$ contraction.  On the zero-sum subspace it is also bounded below by $\alpha(P)$; applying these two inequalities to both directed Hausdorff excesses gives \eqref{eq:bridge-bilipschitz}.  Compactness of the unit sphere shows that $\alpha(P)>0$ exactly when the rows of $P$ are affinely independent.  EC.2 records the full argument.

If $\alpha(P)=0$, the lower inequality can collapse completely: choose a nonzero $x\in\ClrH$ with $xP=0$ and nearby beliefs $p$ and $p+\varepsilon x$.  Their parameter defect is positive but their predictive laws coincide.

The raw parameter defect asks whether arbitrary parameter payoffs can separate two ambiguity sets; the predictive defect asks whether a payoff generated through the controlled observation channel can do so.  The identification modulus quantifies how much of the parameter geometry survives that channel, which is the bridge from ambiguity construction to stochastic control.

\subsection{Control-reachable defects}

In a fixed control problem, not every bounded function of the next observation must be attainable as a continuation value. Let \(\V_{t+1}(h,a)\subseteq\R^{|\Y_{t+1}|}\) be a normalized class of reachable continuation values.

\begin{definition}[Control-reachable defect]\label{def:Dctrl}
Define
\[
\Dctrl^{P,\V}(A,B)
:=\sup_{v\in\V}
\left|
\inf_{p\in A}pPv-
\inf_{q\in B}qPv
\right|.
\]
\end{definition}

For a nonempty reachable class \(\varnothing\ne\V\subseteq\{v:\|v\|_\infty\le1\}\), define its \emph{richness radius} and \emph{covering error} by
\begin{align}
\chi(\V)
&:=\sup\Bigl(\{0\}\cup
\bigl\{c\in(0,1]:c\mathbb B_\infty\subseteq\overline{\V}\bigr\}\Bigr),
\label{eq:reachability-richness}\\
\epsilon(\V)
&:=\sup_{\|v\|_\infty\le1}\inf_{w\in\overline{\V}}\|v-w\|_\infty,
\label{eq:reachability-covering}
\end{align}
where \(\mathbb B_\infty=\{v:\|v\|_\infty\le1\}\).  The adjoined value $0$ makes $\chi(\V)=0$ when $0\notin\overline{\V}$; no reachability of the zero payoff is otherwise assumed.

\begin{proposition}[Control-observability equivalence and approximation]\label{thm:control-observability}
For nonempty compact \(A,B\subseteq\Delta_m\) and any nonempty normalized reachable class \(\varnothing\ne\V\subseteq\mathbb B_\infty\),
\begin{equation}\label{eq:control-richness-bound}
\chi(\V)\,\Dpred^P(A,B)
\le
\Dctrl^{P,\V}(A,B)
\le
\Dpred^P(A,B),
\end{equation}
and
\begin{equation}\label{eq:control-covering-bound}
\Dctrl^{P,\V}(A,B)
\ge
\max\bigl\{0,\Dpred^P(A,B)-2\epsilon(\V)\bigr\}.
\end{equation}
Consequently, \(\chi(\V)=1\) or \(\epsilon(\V)=0\) implies
\[
\Dctrl^{P,\V}(A,B)=\Dpred^P(A,B).
\]
Thus the predictive defect is exactly decision-visible whenever the continuation-value class is sup-norm dense in the normalized payoff ball, and it remains quantitatively visible whenever the class has positive richness radius.
\end{proposition}

\noindent\emph{Proof idea.} The upper bound is immediate from the inclusion of payoff classes.  For the richness bound, embed a scaled near-separating predictive payoff into $\overline{\V}$ and use homogeneity and continuity of lower expectations.  For the covering bound, approximate a near-separating unit payoff by a reachable payoff; each lower expectation is one-Lipschitz in the sup norm, so the two-evaluator difference loses at most twice the approximation error.  EC.2 gives the complete $\eta$-argument.

If the continuation problem implements every normalized terminal transfer in $c\mathbb B_\infty$, then $\chi(\V)\ge c$ and \eqref{eq:control-richness-bound} gives
\[
c\,\Dpred^P(A,B)\le\Dctrl^{P,\V}(A,B)\le\Dpred^P(A,B).
\]
In particular, implementation of the full unit ball makes predictive and control-reachable defects identical.

\section{Why Rectangularity Does Not Repair Relearning}\label{sec:dynamic}

Before treating provenance as a new state variable, we separate it from classical nonrectangularity. Rectangularity resolves coherence within a fixed evaluator vintage; it does not force a later reconstruction to preserve that evaluator.

\subsection{Continuation laws and rectangular hulls}

Fix a finite continuation subtree rooted at a history-action pair \((h_t,a)\), with terminal path set \(\Omega(h_t,a)\). A nonempty compact set \(\Q\subseteq\Delta(\Omega(h_t,a))\) represents a committed set of continuation laws. Its rectangular hull \(\Rect(\Q)\) is the smallest closed set of path laws generated by history-wise pasting of conditional kernels admitted by \(\Q\). This is the standard rectangularization/pasting operation from recursive multiple-prior and robust-MDP theory \citep{EpsteinSchneider2003,Iyengar2005}. On a finite tree, \(\Q\subseteq\Rect(\Q)\).

For a bounded continuation payoff \(Z:\Omega(h_t,a)\to\R\), define
\[
\rho_\Q(Z):=\inf_{Q\in\Q}E_Q[Z].
\]
For compact law sets \(\Q_1,\Q_2\), define the normalized operational distance
\begin{equation}\label{eq:path-operational}
D_{\infty}(\Q_1,\Q_2)
:=\sup_{\|Z\|_\infty\le1}
|\rho_{\Q_1}(Z)-\rho_{\Q_2}(Z)|.
\end{equation}
The finite-dimensional duality in \cref{thm:operational-dual-v2} applies directly on the path simplex.

At node \((h_t,a)\), let
\[
\Q_{s,t}^{\mathrm C}(h_t,a)
\]
denote the continuation-law set inherited from vintage \(s\) under a committed-nature interpretation. Let
\[
\Q_{s,t}^{\mathrm R}(h_t,a):=\Rect(\Q_{s,t}^{\mathrm C}(h_t,a))
\]
be its rectangular counterpart, and let
\[
\Q_t^{\fresh}(h_t,a)
\]
be the continuation-law set used by the fresh date-\(t\) self.

\begin{proposition}[Static-parameter commitment versus rectangularization]\label{prop:static-parameter-rect}
Fix a continuation policy \(\pi\), a current history \(h_t\), and a compact inherited parameter ambiguity set \(M\subseteq\Delta_m\). Let \(P_i^\pi(\cdot\mid h_t)\) be the future observation-path law when parameter \(i\) is fixed for the entire continuation. The committed fixed-parameter family is
\begin{equation}\label{eq:committed-parameter-path}
\Q_{M}^{\mathrm C,\pi}(h_t)
:=\left\{\sum_{i=1}^m p_iP_i^\pi(\cdot\mid h_t):p\in M\right\}.
\end{equation}
For a descendant history \(h_j\), define
\[
M(h_j):=\left\{p\in M:
\Pp_p^{\pi}(h_j\mid h_t)>0\right\}.
\]
For every \(p\in M(h_j)\), the conditional one-step kernel is obtained by Bayes-updating the mixture weights, and the set of well-defined conditional kernels admitted at \(h_j\) is exactly
\[
\mathscr P_{h_j,\pi_j(h_j)}\bigl(B_{t:j}^{\#}[M(h_j)]\bigr).
\]
In particular, if every \(p\in M\) assigns \(h_j\) positive probability---for example under strictly positive primitive likelihoods---this reduces to \(\mathscr P_{h_j,\pi_j(h_j)}(B_{t:j}^{\#}[M])\).  If some mixture assigns \(h_j\) probability zero, no unique conditional law is implied by that path law on the null history; the displayed formula describes exactly the conditionals generated by the positive-probability mixtures.  Thus the Bayes-image description of every off-path node is version-free under full support.  Without full support, rectangular equilibrium kernels at null histories must be specified separately, as they are in \cref{sec:equilibrium}. In general
\[
\Q_M^{\mathrm C,\pi}(h_t)
\subseteq
\Rect\!\left(\Q_M^{\mathrm C,\pi}(h_t)\right).
\]
The inclusion can be strict, and equality holds exactly when the committed family is stable under the corresponding history-wise pasting operation.
\end{proposition}

\noindent\emph{Proof idea.} Conditioning a positive-probability mixture path law at $h_j$ updates its parameter weights by Bayes' rule, giving exactly the displayed image of $M(h_j)$.  Under full support this specifies every descendant kernel without a version choice; rectangularization then closes the family under independent history-wise pasting.  EC.2 gives the complete conditioning and null-history argument.

In the remainder of the equilibrium analysis, \emph{within-vintage ambiguity is deliberately rectangularized} in order to isolate reconstruction from classical nonrectangularity. Thus fresh and inherited recursive evaluators use the appropriate rectangular hulls. The original fixed-parameter model remains represented by the committed family \(\Q^{\mathrm C}\), and the within-vintage defect measures the consequence of replacing it by \(\Q^{\mathrm R}\).

\begin{definition}[Within- and cross-vintage defects]\label{def:two-defects-v2}
Define
\begin{align}
\Delta_{s,t}^{\mathrm{rect}}(h_t,a)
&:=D_\infty\bigl(\Q_{s,t}^{\mathrm C}(h_t,a),
                  \Q_{s,t}^{\mathrm R}(h_t,a)\bigr),
\label{eq:rect-defect-v2}\\
\Delta_{s,t}^{\mathrm{cross}}(h_t,a)
&:=D_\infty\bigl(\Q_{s,t}^{\mathrm R}(h_t,a),
                  \Q_t^{\fresh}(h_t,a)\bigr).
\label{eq:cross-defect-v2}
\end{align}
The first measures classical within-vintage nonrectangularity. The second measures disagreement between the recursively inherited evaluator and the freshly reconstructed evaluator.
\end{definition}

\subsection{Three nature protocols}

The same primitive ambiguity therefore supports three protocols: the standard committed/rectangular alternatives and the fresh-self intrapersonal-game interpretation \citep{EpsteinSchneider2003,Iyengar2005,AusterCheMierendorff2024}.
\begin{enumerate}[label=(\roman*)]
\item \emph{Committed nature}: one continuation law is selected from \(\Q^{\mathrm C}\) and persists over the subtree.
\item \emph{Rectangular nature}: conditional kernels may be selected history by history from the rectangular hull \(\Q^{\mathrm R}\).
\item \emph{Fresh-self evaluation}: there is no persistent adversary; the self moving at date \(t\) evaluates through \(\Q_t^{\fresh}\), anticipating future selves' actions.
\end{enumerate}

\begin{proposition}[Protocol ordering and separation]\label{prop:protocol-ordering}
For every bounded \(Z\),
\[
\rho_{\Q^{\mathrm R}}(Z)\le\rho_{\Q^{\mathrm C}}(Z),
\]
while no universal ordering exists between \(\rho_{\Q^{\mathrm R}}\) and \(\rho_{\Q^{\fresh}}\). Equality between committed and rectangular values for every bounded \(Z\) holds if and only if \(\Delta^{\mathrm{rect}}=0\). Equality between rectangular-inherited and fresh values for every bounded \(Z\) holds if and only if \(\Delta^{\mathrm{cross}}=0\).
\end{proposition}

\noindent\emph{Proof.} Since \(\Q^{\mathrm C}\subseteq\Q^{\mathrm R}\), rectangularization cannot increase the lower value.  The two equivalences are exactly the zero-defect definitions.  No fresh/inherited ordering is possible even on a binary one-step node: with \(\Q^{\fresh}=\{\delta_0\}\) and \(\Q^{\mathrm R}=\operatorname{co}\{\delta_0,\delta_1\}\), the payoff \(Z=(1,0)\) gives \(\rho_{\Q^{\mathrm R}}(Z)=0<1=\rho_{\Q^{\fresh}}(Z)\); interchanging the two sets reverses the inequality. \hfill\(\square\)

\subsection{Universal coherence theorem}

A fixed reward specification may conceal a structural defect because an action dominates by a large margin or because the relevant payoff direction is absent. Necessity therefore requires a rich behavioral criterion.

\begin{definition}[Universal local protocol coherence]\label{def:universal-coherence}
A triple \((\Q^{\mathrm C},\Q^{\mathrm R},\Q^{\fresh})\) at a continuation node is \emph{universally protocol coherent} if, for every bounded risky payoff \(Z\) and every constant safe payoff \(c\), the three robust evaluators induce the same weak ranking between safe and risky.
\end{definition}

\begin{proposition}[Universal two-defect characterization]\label{thm:universal-two-defect}
At a finite continuation node, the following are equivalent:
\begin{enumerate}[label=(\roman*)]
\item \(\Delta^{\mathrm{rect}}=\Delta^{\mathrm{cross}}=0\);
\item committed, rectangular-inherited, and fresh lower expectations agree for every bounded continuation payoff;
\item the three protocols are universally locally coherent;
\item no finite action menu with bounded continuation payoffs can produce different robust choice correspondences under the three protocols.
\end{enumerate}
If either defect is positive, a two-action safe-versus-risky problem strictly separates at least two protocols.
\end{proposition}

\noindent\emph{Proof idea.} Zero defects identify the three lower-expectation functionals and therefore all bounded-payoff rankings.  Conversely, any positive operational defect supplies a bounded payoff on which two functionals differ; a constant strictly between their values creates a two-action safe-versus-risky separation.  EC.2 gives the complete implication chain.

\begin{definition}[Universal dynamic coherence]\label{def:global-coherence}
A finite ambiguity-vintage system is \emph{universally dynamically coherent} if the condition of \cref{thm:universal-two-defect} holds at every subgame, every feasible action, and every pair of relevant vintages.
\end{definition}

\begin{proposition}[Precommitment--equilibrium characterization]\label{thm:global-dynamic-iff}
On a finite event tree, suppose every subgame is evaluated and optimized, including off-path histories. The following are equivalent:
\begin{enumerate}[label=(\roman*)]
\item the system is universally dynamically coherent;
\item for every bounded reward specification and every finite augmentation by local safe-versus-risky actions, committed precommitment, rectangular robust dynamic programming, and fresh-self subgame-perfect optimization have identical continuation preference and optimal-policy correspondences at every subgame;
\item all within-vintage and cross-vintage operational defects in \eqref{eq:rect-defect-v2}--\eqref{eq:cross-defect-v2} vanish.
\end{enumerate}
Thus universal dynamic agreement requires both within-vintage rectangular coherence and cross-vintage reconstruction coherence; neither condition subsumes the other.
\end{proposition}

\noindent\emph{Proof idea.} Apply \cref{thm:universal-two-defect} at every subgame.  If all local functionals agree, backward induction makes continuation preferences and policy correspondences coincide.  If a local defect is positive, the local two-action separation can be inserted at that subgame and violates global agreement.  EC.2 gives the full off-path argument.

The characterization is deliberately operational: it uses equality of lower-expectation functionals rather than literal equality of ambiguity sets.  Maxmin evaluation identifies only closed convex hulls, and the functional formulation prevents dynamically irrelevant geometric differences from being mislabeled as inconsistency.

\section{Ambiguity Provenance as Dynamic State}\label{sec:equilibrium}

Once the evaluator can change, the dynamic problem is representational: what state carries the relevant provenance? Rectangularize each fixed vintage as in \cref{prop:static-parameter-rect}. For \(0\le s\le t<T\), history \(h_t\), and action \(a\), let
\[
\K_{s,t}(h_t,a)\subseteq\Delta(\Y_{t+1}(h_t,a))
\]
be its nonempty compact one-step inherited correspondence, with \(\K_{t,t}\) the fresh date-\(t\) counterpart. These correspondences may also be taken as primitives.

\subsection{Subgame-perfect reconstructed-ambiguity policies}

Sophisticated/subgame-perfect behavior is the classical intrapersonal-game treatment of time inconsistency, not a new equilibrium concept \citep{Strotz1955,AusterCheMierendorff2024}; \citet{LimPark2026} study dynamic consistency of statistical rules in a different framework.

\begin{definition}[Subgame-perfect policy]\label{def:spe-v2}
A policy \(\pi^*=(\pi_t^*)_{t=0}^{T-1}\) is a reconstructed-ambiguity subgame-perfect equilibrium (SPE) if, at every history \(h_t\), \(\pi_t^*(h_t)\) maximizes date-\(t\) self's fresh robust evaluation among all one-period deviations, taking the future policy \(\pi_{t+1:T-1}^*\) as given.
\end{definition}

\begin{proposition}[Finite-tree pure SPE existence]\label{thm:spe-existence-v2}
If histories and action sets are finite, rewards are finite, and every fresh continuation ambiguity set is nonempty and compact, then a pure SPE exists.
\end{proposition}

\noindent\emph{Proof.} Backward induction selects a fresh-evaluation maximizer at every finite subgame, yielding a best response after every history. \hfill\(\square\)

The provenance-specific content is the representation: evaluator origin survives as a state index even when the physical state is unchanged.

\subsection{Triangular equilibrium recursion}

For a candidate policy \(\pi\), define, for \(0\le s\le t\le T\),
\begin{equation}\label{eq:vintage-terminal-v2}
V_{s,T}^{\pi}(h_T):=g(h_T),
\end{equation}
and for \(t<T\),
\begin{equation}\label{eq:vintage-recursion-v2}
\begin{split}
V_{s,t}^{\pi}(h_t)
:=\;&r_t(h_t,\pi_t(h_t))\\
&+\beta\inf_{q\in\K_{s,t}(h_t,\pi_t(h_t))}
\sum_{y}q(y)
V_{s,t+1}^{\pi}(h_t,\pi_t(h_t),y).
\end{split}
\end{equation}
For a date-\(t\) deviation \(a\), followed by \(\pi\), define
\begin{equation}\label{eq:fresh-Q-v2}
\begin{split}
Q_t^{s,\pi}(h_t,a)
:=\;&r_t(h_t,a)\\
&+\beta\inf_{q\in\K_{s,t}(h_t,a)}
\sum_y q(y)V_{s,t+1}^{\pi}(h_t,a,y).
\end{split}
\end{equation}
The fresh deviation value is \(Q_t^{t,\pi}\).

\begin{theorem}[Triangular SPE characterization]\label{thm:triangular-v2}
A policy \(\pi^*\) is an SPE if and only if the family \(\{V_{s,t}^{\pi^*}:0\le s\le t\le T\}\) satisfies \eqref{eq:vintage-terminal-v2}--\eqref{eq:vintage-recursion-v2} and
\begin{equation}\label{eq:fresh-argmax-v2}
\pi_t^*(h_t)
\in\argmax_{a\in\A_t(h_t)}Q_t^{t,\pi^*}(h_t,a)
\end{equation}
for every date and history.
\end{theorem}

\begin{proof}
Rectangularity within vintage \(s\) gives the recursive policy evaluation \eqref{eq:vintage-recursion-v2}. If \(\pi^*\) is an SPE, date \(t\) evaluates a deviation using \(\K_{t,t}\) while anticipating that future selves follow \(\pi^*\); its deviation value is exactly \(Q_t^{t,\pi^*}\), so \eqref{eq:fresh-argmax-v2} holds. Conversely, if the argmax condition holds at every subgame, no self benefits from a one-period deviation, which is the SPE condition.
\end{proof}

The index set $\{(s,t):s\le t\}$ is triangular: the first coordinate records \emph{which evaluator} is judging the continuation, while the second records \emph{when the action is implemented}. This is the dynamic-state form of ambiguity provenance.

For a fixed reward system, the triangle does not require enumeration of physical histories when states recombine.

\begin{proposition}[Vintage-backward evaluation on a recombining state graph]\label{thm:vintage-backward-evaluation}
Let $\X_t$ be finite state sets.  Suppose rewards, feasible actions, controlled successor maps, and all vintage kernels depend on history only through $(t,x_t,a_t)$.  Fix deterministic tie breaking at each date--state pair.  Backward induction that computes, for every $(t,x)$, every action value $Q_{s,t}(x,a)$ with $s\le t$, selects a maximizer of $Q_{t,t}(x,\cdot)$, and records the selected action's value for all active vintages returns the tie-broken pure Markov SPE.  A second initial-vintage Bellman pass returns the rectangular precommitment value and the selected-equilibrium Price of Relearning.

If each one-step ambiguity set has at most $K$ listed laws over at most $Y$ outcomes and there are at most $A$ actions, the work is
\begin{equation}\label{eq:vintage-backward-complexity}
O\!\left(AKY\sum_{t=0}^{T-1}(t+1)|\X_t|\right),
\end{equation}
and storage is $O(\sum_t(t+1)|\X_t|)$.  In particular, when $|\X_t|\le S$, work is $O(T^2SAKY)$.  This evaluates a fixed model and reward system; it does not solve the universal worst-reward maximization defining the selection-free divergence.
\end{proposition}

\noindent\emph{Proof idea.} Backward induction computes every vintage action value, selects the current vintage's maximizer, and records all older valuations of that action.  A separate initial-vintage Bellman pass gives precommitment.  Counting vintage--state--action lower expectations yields \eqref{eq:vintage-backward-complexity}; EC.3 gives the full induction and validation.

\subsection{Behavioral vintage quotients}

State abstraction and stochastic bisimulation study when decision states can be merged without changing values or optimal policies \citep{GivanDeanGreig2003,LiWalshLittman2006}.  A recent preprint develops a canonical preference-augmented state for general history-dependent recursive preferences under its stated behavioral assumptions \citep{Haskell2026}.  Our question is narrower: holding the physical history fixed, which evaluator vintages can be identified without losing any robust continuation valuation?

\begin{definition}[Future operational equivalence of vintages]\label{def:vintage-equivalence}
Fix a subtree rooted at $h_t$ and vintages $s,s'\le t$.  Write
$s\sim_{t,h_t}s'$ if, at every descendant history $h_j\succeq h_t$, every date $j\ge t$, and every feasible action $a$,
\begin{equation}\label{eq:vintage-equivalence}
\Dop\!\left(\K_{s,j}(h_j,a),\K_{s',j}(h_j,a)\right)=0.
\end{equation}
Thus equivalent vintages induce the same lower-expectation functional at every future decision node.  Let $\mathfrak m_t(h_t)$ denote the number of equivalence classes among $\{0,\ldots,t\}$.
\end{definition}

For the converse direction below, call the bounded reward class \emph{locally continuation-rich} if, at every descendant state--action pair, every sufficiently small bounded vector over the next observations can be realized as a common next-date continuation value by independently assigning rewards on the successor subtrees.  The unrestricted branchwise reward class used in the exact representation has this property whenever a positive continuation reward budget is available.

\begin{proposition}[Exact vintage quotient]\label{thm:vintage-quotient}
Fix $h_t$ and $s,s'\le t$.
\begin{enumerate}[label=(\roman*)]
\item If $s\sim_{t,h_t}s'$, then
\[
V_{s,j}^{\pi}(h_j)=V_{s',j}^{\pi}(h_j)
\]
for every descendant $h_j\succeq h_t$, every bounded reward system, and every continuation policy $\pi$.
\item If $s\not\sim_{t,h_t}s'$ and the reward class is locally continuation-rich, there is a bounded reward system and a continuation policy for which the two vintage values differ at some descendant node in the subtree.
\item Consequently, an exact triangular recursion may retain one valuation coordinate per class of $\sim_{t,h_t}$.  Under local continuation richness, no two distinct classes can be merged by a universal coordinate-identification rule.
\end{enumerate}
\end{proposition}

\noindent\emph{Proof idea.} For (i), terminal values coincide.  If the two continuation values coincide at date $j+1$, operational equivalence makes the two date-$j$ lower expectations equal, so backward induction preserves equality.  For (ii), failure of \eqref{eq:vintage-equivalence} yields a bounded continuation vector separated by the two lower-expectation functionals.  Local continuation richness embeds that vector on the corresponding successor branches, producing unequal vintage values at that descendant node.  Part (iii) follows by merging exactly the coordinates that are equal for every continuation problem.  EC.3 gives the full induction and separation construction.

The quotient can be strict even when the physical state is unchanged: what matters is the future lower-envelope functional carried by the evaluator vintage.  At the opposite extreme, one equivalence class recovers ordinary robust dynamic programming.

Exact quotients describe zero-error compression.  For algorithms and calibrated models, the relevant object is the operational metric entropy of the vintage family.

\begin{definition}[Operational vintage cover]\label{def:operational-vintage-cover}
Fix a subtree rooted at $h_t$.  For $s,s'\le t$ define
\begin{equation}\label{eq:operational-vintage-pseudometric}
d_{t,h_t}^{+}(s,s')
:=
\sum_{j=t}^{T-1}\beta^{j-t+1}B_{j+1}
\sup_{\substack{h_j\succeq h_t\\a\in\A_j(h_j)}}
\Dop\!\left(\K_{s,j}(h_j,a),\K_{s',j}(h_j,a)\right).
\end{equation}
A set $R\subseteq\{0,\ldots,t\}$ is an $\varepsilon$-cover if every vintage is within $\varepsilon$ of some $r\in R$ under $d_{t,h_t}^{+}$.  Its minimum cardinality is $N_{t,h_t}^{\rm vit}(\varepsilon)$.
\end{definition}

\begin{theorem}[Approximate vintage quotient]\label{thm:approximate-vintage-quotient}
For every bounded reward system and fixed continuation policy $\pi$,
\begin{equation}\label{eq:operational-vintage-cover-value}
\left|V_{s,t}^{\pi}(h_t)-V_{s',t}^{\pi}(h_t)\right|
\le d_{t,h_t}^{+}(s,s').
\end{equation}
Consequently an $\varepsilon$-cover recovers the root-date coordinates
$\{V_{s,t}^{\pi}(h_t):s\le t\}$ from $N_{t,h_t}^{\rm vit}(\varepsilon)$ representatives with error at most $\varepsilon$.  Uniform recovery of these root-active vintages $s\le t$ at every descendant root is obtained by covering under
\[
\bar d_{t,h_t}^{+}(s,s')
:=\sup_{\substack{j\ge t\\h_j\succeq h_t}}d_{j,h_j}^{+}(s,s'),
\qquad s,s'\le t.
\]

More generally, suppose a compressed backward recursion has current-self action values satisfying
\begin{equation}\label{eq:compressed-action-error}
\sup_{h_u,a}\left|\widetilde Q_u(h_u,a)-Q_u^{u,\pi}(h_u,a)\right|\le\varepsilon_u.
\end{equation}
Every exact SPE is a $2\varepsilon_u$-SPE of the compressed recursion at date $u$.  A strict exact action gap greater than $2\varepsilon_u$ at every subgame preserves the pure SPE.  If vintage $0$ is retained exactly, its precommitment value and evaluation of the preserved policy can then be recomputed exactly, so the root Price of Relearning is unchanged.
\end{theorem}

\noindent\emph{Proof idea.} A one-step lower-expectation comparison gives a backward recursion for the fixed-policy value difference; iteration yields \eqref{eq:operational-vintage-cover-value}.  Applying the same argument at each descendant root proves the $\bar d^{+}$ statement.  Action-gap perturbations cost at most $2\varepsilon_u$.  EC.3 gives the complete induction and the dynamically coherent compressed recursion.

\begin{corollary}[Logarithmic memory under coherent age forgetting]\label{cor:logarithmic-vintage-memory}
Suppose $\beta<1$ and $B_j\le\bar B$.  For every $L$, let $R_j^L\subseteq\{0,\ldots,j\}$ and let $\phi_j^L$ be a retraction onto $R_j^L$ such that $0,j\in R_j^L$, $|R_j^L|\le L+2$, and
\[
\phi_{j+1}^L\!\circ\phi_j^L=\phi_{j+1}^L
\quad\text{on }\{0,\ldots,j\}.
\]
If, for all omitted coordinates,
\[
\sup_{h,a}\Dop\!\left(\K_{s,j}(h,a),\K_{\phi_j^L(s),j}(h,a)\right)
\le C\rho^L
\]
for some $C<\infty$ and $\rho\in(0,1)$, then the coherent compressed recursion defined in EC.3 satisfies
\begin{equation}\label{eq:age-forgetting-error}
\sup_{\substack{0\le s\le j\le T\\ h\in\Hh_j}}\left|V_{s,j}^{\pi}(h)-\widetilde V_{\phi_j^L(s),j}^{\pi}(h)\right|
\le
\frac{\beta\bar B C}{1-\beta}\rho^L.
\end{equation}
The same order bounds fixed-policy action values.  Hence $L=O(\log(1/\varepsilon))$ representative vintage coordinates suffice for fixed accuracy independently of the horizon.  Under the gap condition in \cref{thm:approximate-vintage-quotient}, compression preserves the pure SPE and, because vintage $0$ is retained, the root Price of Relearning.
\end{corollary}

\subsection{Bellman recovery}

\begin{assumption}[Cross-vintage kernel coherence]\label{ass:kernel-coherence-v2}
For every \(s\le t\), history \(h_t\), and action \(a\), the inherited kernel sets are decision equivalent to a common correspondence \(\K_t(h_t,a)\).
\end{assumption}

\begin{proposition}[Collapse to robust dynamic programming]\label{thm:bellman-recovery-v2}
Under Assumption~\ref{ass:kernel-coherence-v2}, the following statements hold.
\begin{enumerate}[label=(\roman*)]
\item For every fixed policy $\pi$, there are functions $V_t^{\pi}$ and $Q_t^{\pi}$ such that
\[
V_{s,t}^{\pi}=V_t^{\pi},
\qquad
Q_t^{s,\pi}=Q_t^{\pi}
\quad\text{for every }s\le t.
\]
They satisfy the policy-evaluation recursion
\begin{equation}\label{eq:coherent-policy-evaluation-v2}
V_t^{\pi}(h_t)
=
r_t(h_t,\pi_t(h_t))
+\beta\inf_{q\in\K_t(h_t,\pi_t(h_t))}
\sum_yq(y)V_{t+1}^{\pi}(h_t,\pi_t(h_t),y).
\end{equation}
No maximization is asserted for a suboptimal fixed policy.
\item Define the robust Bellman value by
\begin{equation}\label{eq:bellman-v2}
V_t^*(h_t)
=\max_{a\in\A_t(h_t)}
\left\{
r_t(h_t,a)
+\beta\inf_{q\in\K_t(h_t,a)}
\sum_yq(y)V_{t+1}^*(h_t,a,y)
\right\}.
\end{equation}
A policy is an SPE if and only if it selects a maximizer in \eqref{eq:bellman-v2} at every subgame.  For every SPE $\pi^*$,
\[
V_{s,t}^{\pi^*}=V_t^*
\quad\text{for all }s\le t.
\]
Hence the SPE and Bellman subgame-optimal action correspondences coincide, and the rectangularized-vintage precommitment value equals $V_t^*$.
\end{enumerate}
A root-optimal precommitment policy may contain arbitrary off-path actions; policy-correspondence equality therefore refers to policies that are Bellman-optimal at every subgame.  If the vintage originates from a committed static-parameter model, $V_t^*$ also equals the original committed precommitment value whenever all relevant within-vintage rectangularization defects on the continuation subtree are zero.
\end{proposition}

\noindent\emph{Proof idea.} First fix $\pi$.  Backward induction and decision equivalence of the vintage kernels give the vintage-independent policy-evaluation recursion \eqref{eq:coherent-policy-evaluation-v2}; this step contains no action maximization.  A second backward induction applies the SPE one-period-deviation condition to the common continuation value and yields the Bellman equation only for subgame-wise optimal policies.  The converse is immediate from \cref{thm:triangular-v2}.  EC.2 gives the complete induction.

\section{The Price of Relearning}\label{sec:por}

Having identified the evaluator-vintage state, we quantify its welfare consequence and then ask whether that state can be ignored, compressed, or efficiently certified.

\subsection{A sharp one-step bound}

Let \(A,B\subseteq\Delta_m\) be inherited and fresh parameter ambiguity sets. A finite action family has payoff vectors \(z_a\) with \(\|z_a\|_\infty\le1\). Let
\[
a_A\in\argmax_a\rho_A(z_a),
\qquad
a_B\in\argmax_a\rho_B(z_a).
\]
The inherited self's loss from the fresh self's action is
\[
L(A,B):=\rho_A(z_{a_A})-\rho_A(z_{a_B})\ge0.
\]

\begin{proposition}[Sharp one-step price of relearning]\label{thm:one-step-por-v2}
For every such action family,
\begin{equation}\label{eq:one-step-por-v2}
0\le L(A,B)\le2\Dop(A,B).
\end{equation}
The factor \(2\) is sharp.
\end{proposition}

\noindent\emph{Proof idea.} Insert the fresh evaluator at the inherited-optimal and fresh-optimal actions.  Fresh optimality makes the middle term nonpositive and each evaluator-change term is at most $\Dop(A,B)$, giving the factor $2$.  A two-state pair of nearby point beliefs and two opposite payoff vectors approaches equality.  EC.3 gives the sharp construction.

The separation is robust to small payoff perturbations because lower expectations are one-Lipschitz.  The quantitative endpoint theorem below strengthens this observation by tying an open-set loss floor to the omitted evaluator's operational distance and strict action margins.

\subsection{An exact dynamic SPE representation without equilibrium selection}

The one-step bound identifies the local welfare channel. In a sophisticated problem the exact continuation object is triangular across ambiguity vintages: self $t$ evaluates a deviation through $V_{t,t+1}^{\pi}$, not through the successor self's diagonal value $V_{t+1,t+1}^{\pi}$.  Moreover, when several pure SPEs coexist, an exact welfare representation should not depend on an arbitrary tie-breaking convention.  We therefore retain both the active vintage valuations and the equilibrium correspondence.

Fix the full rectangular one-step family $\{\K_{s,u}:0\le s\le u<T\}$ from \cref{sec:equilibrium}.  Let $\mathscr R_t(h_t)$ be the class of reward systems on the subtree rooted at $h_t$ satisfying \eqref{eq:reward-bounds}, with rewards independently specifiable on disjoint action--outcome subtrees.  For $r\in\mathscr R_t(h_t)$, let $\mathrm{SPE}_t(r;h_t)$ denote the nonempty set of pure SPE continuation policies on that subtree, and let
\[
P_t^r(h_t):=V_t^{0,*}(h_t)
\]
be vintage $0$'s rectangular precommitment optimum.  For any $\pi\in\mathrm{SPE}_t(r;h_t)$ define the equilibrium-specific Price of Relearning
\begin{equation}\label{eq:equilibrium-specific-por}
\PoR_t(r,\pi;h_t)
:=P_t^r(h_t)-V_{0,t}^{\pi}(h_t)\ge0.
\end{equation}
The inequality follows because the SPE continuation is feasible for the vintage-$0$ rectangular precommitment problem.
This is a directed commitment benchmark, not a claim that vintage $0$ is normatively superior or better informed: it measures loss relative to the evaluator that justified the original plan.

Define the attainable vintage-vector correspondence recursively.  At a terminal history,
\begin{equation}\label{eq:triple-terminal}
\mathfrak S_T(h_T)
:=
\left\{
(z;z,\ldots,z)\in\R^{T+2}: |z|\le G
\right\},
\end{equation}
where the first coordinate is the precommitment coordinate and the remaining $T+1$ coordinates are the valuations of vintages $0,\ldots,T$.  Suppose $\mathfrak S_{t+1}$ has been defined.  For each action $a\in\A_t(h_t)$ choose $c_a\in[-R_t,R_t]$ and, independently for every successor $y$, choose
\[
(p_{a,y};u_{0,a,y},\ldots,u_{t+1,a,y})
\in
\mathfrak S_{t+1}(h_t,a,y).
\]
Define
\begin{equation}\label{eq:triple-action-values}
\begin{aligned}
P_a
&:=c_a+\beta\inf_{q\in\K_{0,t}(h_t,a)}\sum_yq(y)p_{a,y},\\
U_{s,a}
&:=c_a+\beta\inf_{q\in\K_{s,t}(h_t,a)}\sum_yq(y)u_{s,a,y},
\qquad 0\le s\le t.
\end{aligned}
\end{equation}
For every current action
\[
a_F\in\argmax_{a\in\A_t(h_t)}U_{t,a},
\]
include the vector
\begin{equation}\label{eq:triple-recursion}
\left(\max_aP_a;U_{0,a_F},\ldots,U_{t,a_F}\right)
\end{equation}
in $\mathfrak S_t(h_t)$.  Thus the precommitment coordinate is a value and needs no equilibrium selector, whereas the remaining coordinates retain every current action consistent with the SPE one-period-deviation condition.

Define the worst-equilibrium dynamic relearning divergence
\begin{equation}\label{eq:dynamic-por-divergence}
\mathfrak d_t^{\rm RL,SPE}(h_t)
:=
\sup_{(p;u_0,\ldots,u_t)\in\mathfrak S_t(h_t)}(p-u_0).
\end{equation}
It is directed because every equilibrium continuation is evaluated by vintage $0$ against vintage $0$'s precommitment benchmark.

\begin{theorem}[Selection-free exact SPE Price-of-Relearning representation]\label{thm:exact-dynamic-por-representation}
For every history $h_t$,
\begin{equation}\label{eq:attainable-triplet-equality}
\boxed{
\mathfrak S_t(h_t)
=
\left\{
\left(P_t^r(h_t);
V_{0,t}^{\pi}(h_t),\ldots,V_{t,t}^{\pi}(h_t)\right):
 r\in\mathscr R_t(h_t),\ 
 \pi\in\mathrm{SPE}_t(r;h_t)
\right\}.
}
\end{equation}
Consequently,
\begin{equation}\label{eq:exact-dynamic-por}
\boxed{
\mathfrak d_t^{\rm RL,SPE}(h_t)
=
\sup_{r\in\mathscr R_t(h_t)}
\sup_{\pi\in\mathrm{SPE}_t(r;h_t)}
\left[
P_t^r(h_t)-V_{0,t}^{\pi}(h_t)
\right].
}
\end{equation}
Hence $\mathfrak d_t^{\rm RL,SPE}(h_t)=0$ if and only if every bounded reward system in the class and every pure SPE continuation have zero relearning loss at $h_t$.  If every datewise and terminal reward bound is multiplied by $c\ge0$, then $\mathfrak d_t^{\rm RL,SPE}$ is multiplied by $c$.
\end{theorem}

\noindent\emph{Proof idea.} A two-inclusion backward induction identifies the recursion with the equilibrium correspondence. Restriction of any $(r,\pi)$ to successor subtrees generates admissible vectors and \eqref{eq:triple-action-values}; conversely, branchwise reward richness and branch SPEs paste into a global reward system and off-path policy, with $a_F\in\argmax_aU_{t,a}$ enforcing current optimality. EC.3 gives the construction and homogeneity argument.

The theorem separates application-specific loss $\PoR_t(r,\pi;h_t)$ from the universal structural diagnostic $\mathfrak d_t^{\rm RL,SPE}(h_t)$. A deterministic selector is obtained by restricting \eqref{eq:triple-recursion}; the representation itself remains selection-free.

\subsection{Why the triangular state cannot generally be compressed}

The triangular state is not bookkeeping introduced by the proof. Datewise comparisons between only the initial and current ambiguity sets can omit off-diagonal evaluators that determine future equilibrium selections.

\begin{proposition}[Endpoint ambiguity does not identify dynamic relearning]\label{prop:endpoint-insufficiency}
There exist two finite three-date kernel systems $\mathbb K$ and $\widetilde{\mathbb K}$ with identical endpoint pairs
\[
\K_{0,u}=\widetilde\K_{0,u},
\qquad
\K_{u,u}=\widetilde\K_{u,u}
\qquad\forall u,
\]
but
\[
\mathfrak d_0^{\rm RL,SPE}(\mathbb K)>0,
\qquad
\mathfrak d_0^{\rm RL,SPE}(\widetilde{\mathbb K})=0.
\]
The positive-loss witness in $\mathbb K$ has a unique pure SPE and strict action gaps.  Hence endpoint insufficiency is not an equilibrium-selection artifact, and no functional of $\{(\K_{0,u},\K_{u,u})\}_u$ alone can equal the unrestricted worst-equilibrium Price of Relearning on the full model class.
\end{proposition}

\noindent\emph{Proof idea.} Hold all endpoint kernels fixed and change only the off-diagonal $\K_{1,2}$. With terminal payoffs $(0,2)$ and $(1,0)$, one system yields a strict self-$1$ reversal and loss $1/2$ under the unique SPE; replacing $\K_{1,2}$ by the common endpoint kernel makes all vintages decision equivalent. EC.3 gives the construction and a separate tied example.

The counterexample can be made quantitative.  The next result measures the endpoint error directly by the omitted off-diagonal operational separation and shows that the failure survives simultaneous perturbations of payoffs and ambiguity sets.

\begin{theorem}[Quantitative robustness of endpoint omission]\label{thm:endpoint-robustness}
Let $A$ and $B$ be nonempty compact ambiguity sets with
$\delta:=\Dop(A,B)>0$.  For every $\xi\in(0,\delta)$ there is a bounded two-action three-date system with a unique pure SPE and an endpoint-collapsed counterpart such that, for one common reward system,
\begin{equation}\label{eq:endpoint-robust-base}
\PoR_0\ge g>\frac{\delta-\xi}{2},
\qquad
\PoR_0^{\rm end}=0.
\end{equation}
Both the initial-vintage action gap and the off-diagonal self's action gap equal $g$.

More generally, apply a coupled perturbation to the payoffs and to evaluators $A$ and $B$, and define the perturbed endpoint surrogate by replacing the perturbed off-diagonal evaluator by the perturbed endpoint evaluator.  Suppose each of the two robust action values under either evaluator moves by at most $\omega$.  If $2\omega<g$, the two strict choices persist, the endpoint surrogate remains decision equivalent across vintages, and
\begin{equation}\label{eq:endpoint-robust-perturb}
\boxed{\PoR_0-\PoR_0^{\rm end}\ge g-2\omega>0.}
\end{equation}
In particular, if each payoff vector changes by at most $\varepsilon$ in sup norm and the corresponding ambiguity sets change by at most $\kappa$ in $D_{\rm op}$, then one may take $\omega\le\varepsilon+\kappa$.  Endpoint insufficiency therefore holds on a nonempty open set and has a loss floor controlled by the omitted off-diagonal separation.
\end{theorem}

\noindent\emph{Proof idea.} Choose $z$ with $\rho_A(z)-\rho_B(z)>\delta-\xi$ and place the constant payoff at their midpoint, creating equal strict gaps $g$. A deterministic prefix embeds the disagreement; replacing the off-diagonal evaluator by $B$ removes it. Perturbing each robust action value by at most $\omega$ loses at most $2\omega$, with $\omega\le\varepsilon+\kappa$. EC.3 gives the embedding.

The exact vintage quotient in \cref{thm:vintage-quotient} can reduce the state when future lower-expectation functionals coincide.  Without such equivalence, the dimensional growth is unavoidable in a worst-case topological sense.

\begin{theorem}[Exact and approximate irreducibility of vintage memory]\label{thm:vintage-memory-irreducibility}
For every $t\ge1$, there is a one-decision subtree satisfying the finite-tree and bounded-reward assumptions for which $\mathfrak S_t(h_t)$ contains the open rectangle
\begin{equation}\label{eq:open-vintage-rectangle}
\left(\frac12,\frac34\right)
\times
\left(-\frac14,\frac14\right)^{t+1}
\subset\R^{t+2}
\end{equation}
and also the closed rectangle
$[1/2,3/4]\times[-1/4,1/4]^{t+1}$.  In particular it contains the closed cube
\begin{equation}\label{eq:closed-vintage-cube}
Q_t:=\left(\frac58,0,\ldots,0\right)
+\left[-\frac18,\frac18\right]^{t+2}.
\end{equation}
Let $E:Q_t\to\R^d$ be any continuous state encoder and let
$D:\R^d\to\R^{t+2}$ be any decoder, not necessarily continuous.  If $d<t+2$, then
\begin{equation}\label{eq:approx-memory-lower-bound}
\boxed{
\sup_{x\in Q_t}\|D(E(x))-x\|_\infty\ge\frac18.}
\end{equation}
Consequently, exact recovery and every uniform approximation with error strictly below $1/8$ require $d\ge t+2$.  If all reward bounds are multiplied by $c>0$, the lower bound becomes $c/8$.  Thus $\varepsilon$-sufficient continuous vintage memory has worst-case dimension $\Omega(t)$ even when only approximate value preservation is required.
\end{theorem}

\noindent\emph{Proof idea.} The two-action construction used for the exact theorem realizes every point in \eqref{eq:open-vintage-rectangle}.  For the quantitative claim, identify $Q_t$ with a centered $(t+2)$-cube of half-width $1/8$.  Its boundary is antipodally homeomorphic to the sphere.  If $d<t+2$, the Borsuk--Ulam theorem \citep{Matousek2003} gives antipodal boundary points with the same code.  Their sup-norm distance is $1/4$, so a single decoded vector must be at distance at least $1/8$ from one of them.  Scaling rewards scales the attainable cube.  EC.3 gives the construction and the full antipodal argument; see \citet{Matousek2003} for the topological theorem.

The topological statement is worst-case.  The same attainable cube also yields a statistical lower bound for randomized and discontinuous memories once their information capacity is finite.  For an attainable valuation vector $X$ and a possibly randomized representation $Z$, call $I(X;Z)$ its information budget, measured in bits.  Let
\[
h_2(q):=-q\log_2q-(1-q)\log_2(1-q),
\qquad q\in[0,1/2],
\]
and let $h_2^{-1}$ denote its inverse on this interval.  Mutual information and binary entropy are standard information-theoretic quantities \citep{CoverThomas2006}; the specialization below is to attainable ambiguity-vintage states.

\begin{proposition}[Finite-information vintage-memory lower bound]\label{thm:finite-information-memory}
Set $n=t+2$ and let $X$ be uniform on the $2^n$ vertices of the attainable cube $Q_t$.  Consider any possibly randomized memory $Z$ satisfying $I(X;Z)\le B$ bits and any decoder $\widehat X(Z)$.  Then
\begin{equation}\label{eq:average-memory-lower-bound}
\boxed{
\frac1n\mathbb E\|\widehat X-X\|_1
\ge
\frac18\,h_2^{-1}\!\left(\left[1-\frac Bn\right]_+\right).}
\end{equation}
The same right-hand side without the factor $1/8$ lower-bounds the average error probability across the $n$ coordinate-threshold decisions.  Consequently, if a $d$-coordinate implementation stores at most $b>0$ bits per coordinate, average decision error at most $\eta<1/2$ requires
\begin{equation}\label{eq:finite-bit-memory-dimension}
\boxed{
d\ge \frac{(t+2)\{1-h_2(\eta)\}}{b}.}
\end{equation}
For every encoder--decoder pair, its worst conditional risk over $Q_t$ is at least its risk under this uniform distribution.  Hence the same right-hand side is also a minimax lower bound on this attainable subclass.  Under these continuity and finite-information criteria, linear vintage memory is required both for uniform recovery and for randomized average-case/minimax representation.
\end{proposition}

\noindent\emph{Proof idea.} Write $X=c+(1/8)S$ with independent uniform signs.  If $p_i$ is the Bayes error for recovering sign $i$ from $Z$, entropy subadditivity and binary Fano \citep{CoverThomas2006} inequalities give
$I(S;Z)\ge n\{1-h_2(\bar p)\}$, where $\bar p=n^{-1}\sum_i p_i$.  Hence $\bar p\ge h_2^{-1}([1-B/n]_+)$.  A wrong sign entails absolute coordinate error at least $1/8$.  EC.3 gives the complete argument.  The information inequality is classical; the contribution here is that every sign vector is an actual bounded-reward vintage valuation state \citep{CoverThomas2006}.

The two memory results are complementary: the first gives a uniform error floor for continuous real-valued codes; the second allows randomized/discontinuous encoders but limits mutual information and gives an average-case rate--distortion floor. Neither excludes scalar-objective compression or special subclasses. They isolate information forced by ambiguity vintages rather than physical histories \citep{GivanDeanGreig2003,LiWalshLittman2006,Haskell2026}.

Irreducibility does not preclude exact computation in finite polyhedral models.  EC.3 proves that the attainable correspondence is a finite union of compact polytopes and gives a selection-free disjunctive MILP.  Closed-convex-hull pruning, action-value certificates, and policy cells preserve the exact object.  Shared-state Markov formulations and vintage-backward evaluation address specified state models without independent branchwise successor choices.  The next result shows why the outer maximization over a shared reward system has a different computational status.

\subsection{Computational frontier of universal certification}\label{sec:por-cert-complexity}

\begin{definition}[Universal state-coupled relearning certification]\label{def:rl-certification}
An instance consists of a rational finite state graph, rational polyhedral vintage kernels, a rational box $\Theta\subset\R^p$, an affine state-reward map $\theta\mapsto r^\theta$, and a rational threshold $\gamma$.  Write $\mathcal R_\Theta:=\{r^\theta:\theta\in\Theta\}$.  The same coordinate of $\theta$ may enter rewards at several histories or states, so $\mathcal R_\Theta$ need not satisfy the branchwise independence used in \cref{thm:exact-dynamic-por-representation}.  Its family-specific universal certificate is
\[
\mathfrak d_0^{\rm cert}(\mathcal R_\Theta)
:=
\sup_{r\in\mathcal R_\Theta}\ \sup_{\pi\in\mathrm{SPE}(r)}
\left[P_0^{r}-V_{0,0}^{r,\pi}\right].
\]
We abbreviate this quantity by $\mathfrak d_0^{\rm cert}(\Theta)$ and, when the primitives depend on a graph $G$, by $\mathfrak d_0^{\rm cert}(\Theta;G)$.  The certificate is monotone under reward-family inclusion.  When $\mathcal R_\Theta$ equals the full branchwise-rich bounded class in \cref{thm:exact-dynamic-por-representation}, it coincides with $\mathfrak d_0^{\rm RL,SPE}$; no such identification is made for a state-coupled subfamily.  The reward-interaction graph has one vertex for each reward coordinate and joins two coordinates whenever they enter a common local continuation comparison.  The decision question is whether
\begin{equation}\label{eq:rl-certification-decision}
\mathfrak d_0^{\rm cert}(\Theta)\ge\gamma.
\end{equation}
This is universal Price-of-Relearning certification for a compact state-coupled stress-test family.  Evaluation of one specified $\theta$ is a different problem.
\end{definition}

\begin{theorem}[Hardness and sparse-interaction tractability]\label{thm:por-cert-complexity}
The general problem in \cref{def:rl-certification} is NP-hard.  More precisely, a restricted subclass is NP-complete even with two decision dates, one root action, two actions per post-initial state, zero stage rewards, discount factor one, singleton rectangular kernels, and only the initial and current evaluator vintages.

For every undirected graph $G=(V,E)$ with $|E|\ge1$, one can construct such an instance with shared terminal reward vector $x\in[0,1]^V$ satisfying
\begin{equation}\label{eq:maxcut-por-identity}
\boxed{
\mathfrak d_0^{\rm cert}([0,1]^V;G)
=\frac{\operatorname{MaxCut}(G)}{|E|}.}
\end{equation}
Consequently, fixed horizon and bounded vintage count do not by themselves imply tractability.  On the graph-generated subclass just described, a supplied nice tree decomposition of width $w$ with $N_{\rm bag}$ bags yields the exact certificate in
\begin{equation}\label{eq:treewidth-cert-complexity}
O\!\left(2^{w+1}N_{\rm bag}\right)
\end{equation}
time.  If the supplied or constructed nice decomposition has $N_{\rm bag}=O(|V|+|E|)$, this specializes to
$O(2^{w+1}(|V|+|E|))$.  Standard conversion conventions may instead introduce a factor polynomial in $w$; either form is fixed-parameter tractable in the reward-interaction treewidth.
\end{theorem}

\noindent\emph{Proof idea.} Create one date-one state per edge and let its two actions interchange the singleton terminal laws seen by the initial and current vintages. Then
\[
\PoR_0(x)=|E|^{-1}\sum_{\{u,v\}\in E}|x_u-x_v|.
\]
The convex objective is maximized at a binary cube vertex, hence at a cut; \textsc{Max-Cut} is NP-complete \citep{GareyJohnsonStockmeyer1976}. Tree-decomposition dynamic programming gives the stated FPT bound \citep{CyganEtAl2015}. EC.3 supplies the full reduction and membership argument.

The hardness comes from outer optimization over shared reward coordinates and cross-vintage reversals, not fixed-model evaluation, which is polynomial in the explicit finite representation \citep{PapadimitriouTsitsiklis1987}. Thus specified-model evaluation, sparse-interaction certification, and unrestricted universal certification are distinct tasks.

\subsection{Stagewise benchmark and general bounds}

In the stagewise specialization, all date-$u$ histories are behaviorally equivalent, the action affects only $z_{u,a}(Y_{u+1})$, and admissible menus form $\prod_u\mathscr M_u$.  Let $A_u=\K_{0,u}$, $B_u=\K_{u,u}$, $A_0=B_0$, and
\[
M_u^B(m_u)=\arg\max_a\rho_{B_u}(z_{u,a}),
\qquad
\mathfrak r_u(A_u,B_u)=
\sup_{m_u}\sup_{a_B\in M_u^B(m_u)}
\{\max_a\rho_{A_u}(z_{u,a})-\rho_{A_u}(z_{u,a_B})\}.
\]
Define
\begin{equation}\label{eq:stagewise-spe-correspondence}
\mathrm{SPE}^{\rm stg}(m)=\{\pi:\exists a_u\in M_u^B(m_u),\ \pi_u(h_u)=a_u\ \forall u,h_u\}.
\end{equation}

\begin{proposition}[Stagewise equilibrium embedding]\label{prop:stagewise-spe-embedding}
$\mathrm{SPE}^{\rm stg}(m)$ is a nonempty subset of the full pure-SPE correspondence; the two coincide when every $M_u^B(m_u)$ is a singleton.
\end{proposition}

\begin{corollary}[Exact stagewise product matching]\label{cor:stagewise-matching-por}
\begin{equation}\label{eq:stagewise-exact-sum}
\sup_m\sup_{\pi\in\mathrm{SPE}^{\rm stg}(m)}\PoR_0(m,\pi)
=\sum_{u=1}^{T-1}\beta^{u+1}\mathfrak r_u(A_u,B_u).
\end{equation}
This quantity lower-bounds unrestricted worst-SPE loss and equals it under unique fresh maximizers.  With all finite $C_u$-bounded menus admissible, it lies between
\begin{equation}\label{eq:stagewise-two-sided}
\sum_{u=1}^{T-1}\beta^{u+1}C_u\Dop(A_u,B_u)
\quad\text{and}\quad
2\sum_{u=1}^{T-1}\beta^{u+1}C_u\Dop(A_u,B_u),
\end{equation}
with sharp constants.  In the identifiable binary specialization it is at least
\begin{equation}\label{eq:observable-concentration-cost}
2\lambda_M\sum_{u=1}^{T-1}\beta^{u+1}C_u\alpha(P_u)(r_0-r_u),
\qquad \lambda_M=(2+2\cosh M)^{-1}.
\end{equation}
\end{corollary}
EC.3 proves the embedding, product identity, sharpness, and the failure of unrestricted endpoint matching under history-dependent tie selection.

\subsection{Quantitative multistage control}

For vintages $s,s'\le t$, let
\begin{equation}\label{eq:kernel-defect}
\delta_{s,s';t}
:=\sup_{h_t,a}\Dop(\K_{s,t}(h_t,a),\K_{s',t}(h_t,a)).
\end{equation}
For a policy $\pi$, define the vintage-$0$ Bellman residual
\begin{equation}\label{eq:bellman-residual}
e_t^{0,\pi}(h_t)
:=\max_aQ_t^{0,\pi}(h_t,a)-Q_t^{0,\pi}(h_t,\pi_t(h_t)).
\end{equation}

\begin{proposition}[Recursive stability and residual estimates]\label{prop:recursive-por-estimates}
For bounded rewards and compact one-step ambiguity sets:
\begin{enumerate}[label=(\roman*)]
\item For every policy $\pi$ and vintages $s,s'\le t$,
\begin{equation}\label{eq:fixed-policy-vintage-bound}
\sup_{h_t}|V_{s,t}^{\pi}(h_t)-V_{s',t}^{\pi}(h_t)|
\le
\sum_{j=t}^{T-1}\beta^{j-t+1}B_{j+1}\delta_{s,s';j}.
\end{equation}
\item For every policy $\pi$,
\begin{equation}\label{eq:performance-residual-bound}
\sup_{h_0}[V_0^{0,*}(h_0)-V_0^{0,\pi}(h_0)]
\le\sum_{t=0}^{T-1}\beta^t\|e_t^{0,\pi}\|_\infty.
\end{equation}
\item If $\pi^{\spe}$ is an SPE and
\begin{equation}\label{eq:epsilon-t}
\varepsilon_t
:=\sup_{h_t,a}|Q_t^{0,\pi^{\spe}}(h_t,a)-Q_t^{t,\pi^{\spe}}(h_t,a)|,
\end{equation}
then $\|e_t^{0,\pi^{\spe}}\|_\infty\le2\varepsilon_t$.
\item The action-value discrepancy satisfies
\begin{equation}\label{eq:epsilon-bound}
\varepsilon_t
\le
\sum_{j=t}^{T-1}\beta^{j-t+1}B_{j+1}\delta_{0,t;j}.
\end{equation}
\end{enumerate}
\end{proposition}

\begin{proposition}[Multistage price of relearning]\label{thm:multistage-por}
For an SPE $\pi^{\spe}$, let
\[
\PoR_0:=\sup_{h_0}[V_0^{0,*}(h_0)-V_0^{0,\pi^{\spe}}(h_0)].
\]
Then
\begin{equation}\label{eq:multistage-por}
\PoR_0
\le
2\sum_{t=0}^{T-1}\beta^t\varepsilon_t
\le
2\sum_{t=0}^{T-1}\sum_{j=t}^{T-1}
\beta^{j+1}B_{j+1}\delta_{0,t;j}.
\end{equation}
The universal factor $2$ is unimprovable.  A sharp sequence places the sharp one-step menu at date $1$ after a deterministic zero-reward date-$0$ prefix.
\end{proposition}

\noindent\emph{Proof idea.} The first inequality combines the performance-residual recursion with approximate inherited greediness; the second substitutes the action-value vintage bound.  EC.3 proves all four recursive estimates and the sharp two-decision embedding.

Nested ambiguity vintages permit a one-sided refinement.  The inclusion condition below is satisfied by the Gaussian inherited-versus-fresh family in \cref{sec:pricing}; it is not imposed on the general model.

\begin{proposition}[Factor-one residual under nested vintages]\label{thm:nested-vintage-factor-one}
Suppose that for every relevant descendant node and $s<u\leq j$,
\[
 K_{s,j}(h,a)\subseteq K_{u,j}(h,a).
\]
Then for every fixed policy $\pi$,
$V_{s,j}^{\pi}(h)\geq V_{u,j}^{\pi}(h)$ and
$Q_{s,j}^{\pi}(h,a)\geq Q_{u,j}^{\pi}(h,a)$.
If self $u$ chooses
$a_u\in\arg\max_aQ_{u,u}^{\pi}(h,a)$, then
\[
 \max_aQ_{s,u}^{\pi}(h,a)-Q_{s,u}^{\pi}(h,a_u)
 \leq
 \max_a\{Q_{s,u}^{\pi}(h,a)-Q_{u,u}^{\pi}(h,a)\}.
\]
Thus the generic factor two in the approximate-greediness argument improves to one.
\end{proposition}

The nested order also yields a selection-aware local linear-program envelope and an exact shared-state Markov formulation.  These refinements are stated and proved in EC.3 because they sharpen computation rather than define the welfare object.  General state abstraction is established; the contribution here is the exact reduction for the vintage-indexed correspondence and its certificates \citep{GivanDeanGreig2003,LiWalshLittman2006}.

\subsection{When reconstruction does not create a first-order welfare loss}

\begin{proposition}[Policy robustness under gaps and smooth actions]\label{prop:policy-robustness-regimes}
\begin{enumerate}[label=(\roman*)]
\item Suppose that, at every subgame, the unique maximizer of $Q_t^{0,\pi^{\spe}}$ has gap $\kappa_t(h_t)>2\varepsilon_t$.  Then $\pi^{\spe}$ is also optimal for the vintage-$0$ rectangular problem and
\[
\PoR_0=0.
\]
\item Let $A\subseteq\R^d$ be convex and let $Q^{\inh},Q^{\fresh}$ be continuously differentiable with interior maximizers $a_I,a_F$.  If $Q^{\inh}$ is $\mu$-strongly concave with $L$-Lipschitz gradient and
$\sup_{a\in A}\|\nabla Q^{\fresh}(a)-\nabla Q^{\inh}(a)\|_2\le\varepsilon$, then
\begin{align}
\|a_F-a_I\|_2&\le\varepsilon/\mu,
\label{eq:action-displacement}\\
0\le Q^{\inh}(a_I)-Q^{\inh}(a_F)
&\le L\varepsilon^2/(2\mu^2).
\label{eq:second-order-PoR}
\end{align}
\end{enumerate}
\end{proposition}

\noindent\emph{Proof idea.}  In the discrete case, a perturbation smaller than half the action gap cannot change the maximizer, so every vintage-$0$ Bellman residual vanishes.  In the smooth case, strong concavity converts the gradient discrepancy into an $O(\varepsilon)$ action displacement, and smoothness converts that displacement into an $O(\varepsilon^2)$ inherited loss.  EC.3 gives the complete arguments.  The two parts distinguish locally flat discrete choice from second-order welfare response under smooth interior control.

\section{Repairing Relearning: Compatible Protocol Design}\label{sec:design}

Compatibility is the repair, but it defines a feasible protocol class rather than a preferred member. Protocol choice is endogenous because it changes the sophisticated policy and hence where statistical fidelity matters. This differs from reward design and policy teaching, which alter incentives to induce target behavior \citep{BenPoratEtAl2024,WuMaFuHan2025}.

\subsection{Repair coupled to equilibrium}

On one likelihood orbit, let $\mathscr X$ be a nonempty finite calibration set, let $w_x>0$, and let $\nu$ be a finite Borel measure on the unit direction sphere.  Each target $C_x$ is nonempty, compact, and convex.  For a set $C$ of the same type, write $h_C(u)=\sup_{c\in C}\langle c,u\rangle$ and define
\begin{equation}\label{eq:static-compatible-loss-v7}
\mathcal L(C)=\sum_{x\in\mathscr X}w_x\int[h_C(u)-h_{C_x}(u)]^2\,d\nu(u).
\end{equation}
All integrals are finite because support functions are continuous and bounded on the unit sphere.
\begin{proposition}[Conditional compatible repair]\label{prop:conditional-compatible-repair-v7}
Let $W=\sum_xw_x$ and let
\[
C^{\rm av}=W^{-1}\sum_xw_xC_x
\]
be the weighted Minkowski average.  Then, for every nonempty compact convex $C$,
\begin{equation}\label{eq:conditional-compatible-variance-identity}
\mathcal L(C)=\mathcal L(C^{\rm av})
+W\int[h_C(u)-h_{C^{\rm av}}(u)]^2\,d\nu(u).
\end{equation}
Hence $C^{\rm av}$ is a minimizer.  Every minimizer agrees with it in support function $\nu$-almost everywhere, and it is unique when $\supp(\nu)$ determines compact convex sets.
\end{proposition}
The identity is weighted least squares applied direction by direction, using Minkowski additivity of support functions \citep{RockafellarWets1998}; EC.4 supplies the complete argument.

Let $\mathfrak C_{\rm stat}$ be a nonempty compact set of statistically admissible compatible protocols, $\boldsymbol\Pi$ the finite pure-policy set, and $\mathrm{SPE}(c)$ the equilibrium correspondence induced by $c$.  Under a nominal law, write $d_t^\pi(h_t)$ for history occupancy and $\ell_t(c;h_t,a)$ for continuous decision-visible fidelity loss.  Define
\begin{equation}\label{eq:endogenous-design-objective}
\mathcal J(c,\pi)=\sum_{t<T}\beta^t\sum_{h_t}d_t^\pi(h_t)\ell_t(c;h_t,\pi_t(h_t)).
\end{equation}
\begin{proposition}[Equilibrium-coupled compatible design]\label{thm:endogenous-design-existence}
If pure-policy action values are continuous in $c$ and every $c\in\mathfrak C_{\rm stat}$ admits a pure SPE, then
\begin{equation}\label{eq:joint-design-problem}
\min\{\mathcal J(c,\pi):c\in\mathfrak C_{\rm stat},\ \pi\in\mathrm{SPE}(c)\}
\end{equation}
has a solution.  Moreover, with $\mathfrak C_\pi:=\{c\in\mathfrak C_{\rm stat}:\pi\in\mathrm{SPE}(c)\}$,
\begin{equation}\label{eq:policy-decomposition-design}
\min_{c,\pi:\,\pi\in\mathrm{SPE}(c)}\mathcal J(c,\pi)
=\min_{\pi:\,\mathfrak C_\pi\ne\varnothing}\ \min_{c\in\mathfrak C_\pi}\mathcal J(c,\pi).
\end{equation}
Every nonempty policy cell is compact.
\end{proposition}
\noindent\emph{Proof idea.} Continuity makes each cell closed.  The feasible graph is the nonempty finite union $\bigcup_\pi\mathfrak C_\pi\times\{\pi\}$ and is therefore compact; Weierstrass gives existence and the finite-union identity gives \eqref{eq:policy-decomposition-design}.  EC.4 provides the complete graph argument.
The pair formulation preserves equilibrium multiplicity; absent equilibrium selection, a design objective must specify a selector or a worst-selection criterion.

\subsection{A self-consistent computational selection}

For a one-parameter compatible protocol $r\in[\underline r,\bar r]$, entropy regularization with parameter $\tau>0$ produces a unique soft sophisticated policy $\pi_\tau^r$ \citep{GeistScherrerPietquin2019}.  All protocols are evaluated from a common nominal initial distribution and under a common nominal transition kernel.  Let $\omega_i(r)$ be discounted nominal occupancies, $s_i>0$ control-sensitivity weights, and $\widehat r_i\in[0,R]$ statistical targets.  The conditional repair map is
\begin{equation}\label{eq:endogenous-radius-map}
\mathcal T_\tau(r)=\Pi_{[\underline r,\bar r]}
\frac{\sum_i\omega_i(r)s_i\widehat r_i}{\sum_i\omega_i(r)s_i}.
\end{equation}
\begin{proposition}[Endogenous compatible fixed point]\label{thm:endogenous-design-fixed-point}
Suppose $\sum_i\omega_i(r)s_i\ge\underline d>0$ and
$\sum_i s_i|\omega_i(r)-\omega_i(r')|\le L_a|r-r'|$.  With $\Delta_{\widehat r}=\max_i\widehat r_i-\min_i\widehat r_i$, $\mathcal T_\tau$ has a fixed point and
\begin{equation}\label{eq:design-contraction-modulus}
|\mathcal T_\tau(r)-\mathcal T_\tau(r')|\le q_{\rm tar}|r-r'|,
\qquad q_{\rm tar}=\Delta_{\widehat r}L_a/\underline d.
\end{equation}
If $q_{\rm tar}<1$, the fixed point is unique and $r^{k+1}=\mathcal T_\tau(r^k)$ converges globally at rate $q_{\rm tar}$.
\end{proposition}
EC.4 derives the occupancy Lipschitz constant and proves the proposition.  Regularization is a computational selection device, not a redefinition of pure-SPE welfare.

\subsection{Fidelity to a fixed target}

Let $\mathfrak K^{\rm tar}$ be one preassigned rectangular target evaluator, not the changing-vintage fresh-self protocol.  Let $\pi^c$ select a compatible Bellman maximizer at every subgame under the rectangular family $\mathfrak K^c$, so $\pi_t^c(h_t)\in\argmax_a Q_t^{c,\pi^c}(h_t,a)$ for every $(t,h_t)$.  Define
$\varepsilon_t^{\rm tar}(c)=\sup_{h_t,a}|Q_t^{{\rm tar},\pi^c}-Q_t^{c,\pi^c}|$ and
$\delta_t^{\rm tar,c}=\sup_{h_t,a}\Dop(K_t^{\rm tar},K_t^c)$.
\begin{proposition}[Fixed-target welfare certificate]\label{thm:compatible-design-performance}
If $V_0^{{\rm tar},*}$ is the robust-DP optimum under $\mathfrak K^{\rm tar}$, then
\begin{align}
0\le V_0^{{\rm tar},*}-V_0^{{\rm tar},\pi^c}
&\le2\sum_{t<T}\beta^t\varepsilon_t^{\rm tar}(c),\label{eq:compatible-design-performance}\\
\varepsilon_t^{\rm tar}(c)
&\le\sum_{j=t}^{T-1}\beta^{j-t+1}B_{j+1}\delta_j^{\rm tar,c}.\label{eq:compatible-target-epsilon-bound}
\end{align}
Consequently,
\begin{equation}\label{eq:compatible-design-operational-certificate}
V_0^{{\rm tar},*}-V_0^{{\rm tar},\pi^c}
\le 2\sum_{j=0}^{T-1}(j+1)\beta^{j+1}B_{j+1}\delta_j^{{\rm tar},c}.
\end{equation}
The same statements hold with control-reachable defects whenever the continuation values lie in the normalized reachable class.
\end{proposition}
The proof applies the robust performance-residual recursion to the fixed target and then fixed-policy kernel stability; EC.4 gives the control-reachable refinement.  Changing-vintage welfare remains the distinct Price-of-Relearning object.

\section{Operational Feedback in Dynamic Pricing}\label{sec:pricing}

Pricing provides a canonical feedback setting because actions affect both revenue and future information \citep{AramanCaldentey2009,FariasVanRoy2010,CheungSimchiLeviWang2017,NambiarSimchiLeviWang2019}. Unlike adaptive risk learning or decision-dependent DRO \citep{PakimanChenNadarajahJasin2025,LuoMehrotra2020,QuJiaYou2025}, the Bayesian experiment is fixed here; the later manager changes which ambiguity vintage evaluates it.

\subsection{Gaussian inheritance and information feedback}

At a nonnegative price $p_t\in[\underline p,\bar p]$, with $0\le\underline p<\bar p<\infty$, let
\begin{equation}\label{eq:linear-gaussian-demand}
Y_{t+1}=a-\theta p_t+\varepsilon_{t+1},\qquad \varepsilon_{t+1}\sim N(0,\sigma^2),
\end{equation}
where $\theta\in\R$ is an unknown local demand-sensitivity coefficient.  Positive values correspond to downward-sloping demand; the economic interpretation and empirical reporting below are restricted to the positive-sensitivity region.  If $\theta\mid\F_t\sim N(m_t,v_t)$ with $v_t>0$, then
\begin{equation}\label{eq:linear-gaussian-update}
v_{t+1}^{-1}=v_t^{-1}+p_t^2/\sigma^2,
\qquad
m_{t+1}=v_{t+1}\{m_t/v_t+p_t(a-Y_{t+1})/\sigma^2\}.
\end{equation}
Represent a candidate posterior by $N(m_t+v_tq,v_t)$. For $z>0$, define the date-$s$ fixed-$z$ posterior-mean family
\[
\mathcal G_s(z):=\left\{N(m_s+v_sq,v_s): |q|\le z/\sqrt{v_s}\right\}.
\]
If $z$ is chosen as a Gaussian credible-radius multiplier, this is the corresponding fixed-level credible family; no coverage claim is used below.

\begin{theorem}[Gaussian inheritance, pricing, and feedback]\label{thm:linear-gaussian-inheritance-pricing}
Assume $a>0$, $\sigma^2>0$, $z>0$, and $0<v_t\le v_s$.  Fix $s<t$ and apply the same realized experiment to every candidate posterior.  The natural displacement is invariant,
\begin{equation}\label{eq:q-invariance-pricing}m_t^q=m_t+v_tq,
\end{equation}
and the date-$s$ family $\mathcal G_s(z)$ is inherited at date $t$ with raw radius
\begin{align}
\rho_{s\to t}^{I}&=zv_t/\sqrt{v_s},\label{eq:inherited-credible-radius-pricing}\\
\rho_t^F&=z\sqrt{v_t}.\label{eq:fresh-credible-radius-pricing}
\end{align}
Thus $v_t<v_s$ implies $\rho_t^F>\rho_{s\to t}^I$.  For a $q$-budget $k\ge0$, put $d(k):=m_t+kv_t$.  Worst-case expected one-period revenue on the nonnegative price interval is
\begin{equation}\label{eq:robust-linear-revenue}
\underline r(p;m_t,v_t,k)=ap-d(k)p^2,
\end{equation}
and its unique constrained maximizer is
\begin{equation}\label{eq:robust-linear-price}
p^\star(d)=
\begin{cases}
\bar p, & d\le0,\\[2pt]
\Pi_{[\underline p,\bar p]}\!\left\{a/(2d)\right\}, & d>0.
\end{cases}
\end{equation}
The map $d\mapsto p^\star(d)$ is nonincreasing.  For the inherited and fresh budgets $k_I=z/\sqrt{v_s}$ and $k_F=z/\sqrt{v_t}$, $v_t<v_s$ gives $d(k_F)>d(k_I)$ and therefore $p_t^F\le p_{s,t}^I$.  If $d(k_I)>0$, both objectives are strictly concave; equality then occurs exactly when the two projected optima coincide, and both interior optima imply $p_t^F<p_{s,t}^I$.  If $d(k_I)\le0$, then $p_{s,t}^I=\bar p$, with strict inequality precisely when $p_t^F<\bar p$.  If every candidate posterior mean is to remain positive, the additional economic restriction is $m_t-k_Fv_t>0$; this restriction concerns the means, not the Gaussian supports, and is not needed for the constrained quadratic maximization.
\end{theorem}
The natural-parameter calculation and the quadratic maximization are given in EC.5.  The same derivation shows that the raw fresh-minus-inherited radius is largest at an intermediate information level rather than at either zero or infinite information.

\begin{proposition}[Conditional information-feedback amplification]\label{prop:pricing-feedback}
Assume $a>0$, $z>0$, $\sigma^2>0$, a common positive posterior-mean sequence, and positive price denominators along both conditional paths.  Let the fresh and inherited myopic price rules use the radii in \eqref{eq:inherited-credible-radius-pricing}--\eqref{eq:fresh-credible-radius-pricing}.  Starting from a common variance, their conditional recursions satisfy
\[
v_t^F\ge v_t^I,
\qquad
p_t^F(v_t^F)\le p_t^I(v_t^I)
\quad\text{for all later dates}.
\]
A strict feasible price gap gives a strict next-period variance gap.  Hence fresh reconstruction can lower price, reduce Fisher information, and preserve the variance that sustains future ambiguity.
\end{proposition}
At a common variance, the fresh penalty is weakly larger, and it is strictly larger when $v<v_s$; both price rules are nonincreasing in variance; on the nonnegative price interval the variance update is increasing in variance and decreasing in price.  Induction proves the proposition.  The result is an analytically transparent myopic benchmark conditional on a common posterior-mean path.  It is neither stochastic dominance of the endogenous posterior processes nor a comparative static for fully nonmyopic sophisticated pricing policies.

\subsection{Controlled benchmark and scanner-data scale}

EC.5 reports the scanner-calibrated held-out-parameter experiment, and EC.3 reports the separate certification benchmarks. Under common simulated shocks, fresh reconstruction lowers price, information, and discounted revenue relative to inheritance. We do not report a compatible-design pricing policy here: \cref{sec:design} is the prescriptive theory, whereas this experiment isolates the inherited-versus-fresh mechanism.

The three-brand Dominick's orange-juice scanner panel \citep{Montgomery1997}, available for academic research, contains 28,947 store--brand--week observations at 83 stores.  A descriptive store--brand and week fixed-effects regression gives own log-price sensitivities between $2.84$ and $3.22$ in absolute value.  After a 24-week initial vintage, the median final posterior standard deviation is $45\%$--$61\%$ of its initial value, so the fresh radius is $1.63$--$2.22$ times the inherited radius.  In paths that remain in the positive-sensitivity region, the median local price-suppression index is $9.8\%$--$15.9\%$.

\begin{figure}[t]
\centering
\includegraphics[width=0.88\textwidth]{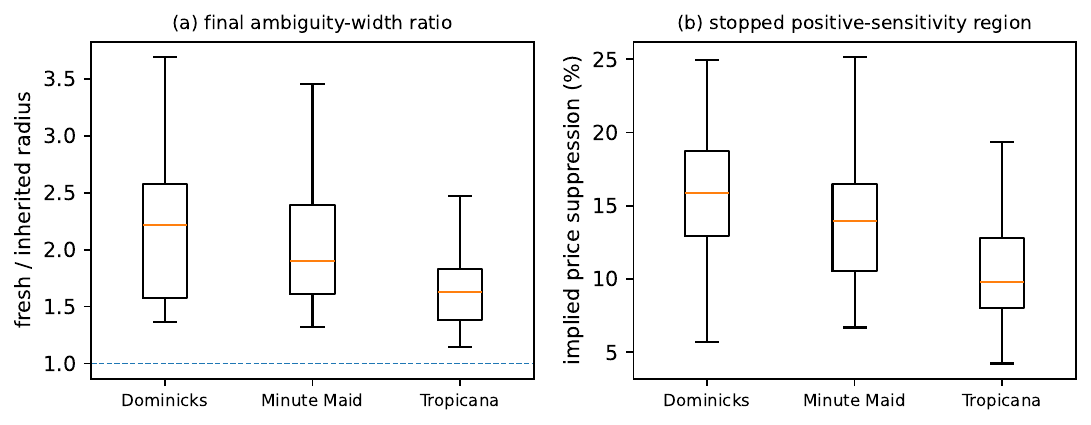}
\caption{Retrospective scanner-data replay.  Panel (a) reports the final fresh-to-inherited radius ratio.  Panel (b) reports the local suppression index only in the positive-sensitivity region.  Historical prices are observational; the figure measures scale, not a causal protocol effect.}
\label{fig:oj-real-replay-main}
\end{figure}

A held-out-parameter experiment initializes policies from early-period posteriors, uses each held-out store's later posterior mean only as an evaluation sensitivity, uses the other four store folds to form a reference sensitivity for the intercept normalization $a=2\theta_{\rm ref}$, and couples policies with common Gaussian shocks.  Over 20 periods, fresh reconstruction lowers normalized discounted revenue relative to inheritance by $0.894$, $0.486$, and $0.330$ for Dominicks, Minute Maid, and Tropicana, while reducing information.  For the fresh policy, the experiment also reports discounted inherited local regret, a one-period diagnostic rather than the exact dynamic Price of Relearning because price changes future information.  EC.5 bounds the omitted continuation term and reports uncertainty and sensitivity analyses.  All contrasts are model generated, not historical treatment effects.

\subsection{What the observational panel can identify}

Let $\mathcal O_H$ be the state--action support of the historical law, $\Delta^{F-I}$ the discounted outcome contrast between the two algorithms, and $d_t^\pi(h,a)=\Pp^\pi(H_t=h,A_t=a)$ their model-induced occupancies.
\begin{proposition}[Identification boundary for a dynamic protocol effect]\label{prop:protocol-identification-boundary}
\begin{enumerate}[label=(\roman*)]
\item Fix the two algorithmic decision rules and the transition law.  If their induced occupancies differ at a reachable state--action pair outside $\mathcal O_H$, then $\Delta^{F-I}$ is not point identified over bounded conditional-mean reward models unrestricted there.
\item Fix the two protocol decision rules and the transition law at their fitted values, so their date-$t$ occupancies $d_t^F,d_t^I$ are fixed.  If the conditional mean reward differs from its fitted value by $u_t(h,a)$ with $|u_t(h,a)|\le\Gamma_t(h,a)$, then the sharp sensitivity interval over this disturbance class is
\begin{equation}\label{eq:sharp-protocol-sensitivity-set}
[\widehat\Delta-B_\Gamma,\widehat\Delta+B_\Gamma],
\qquad
B_\Gamma=\sum_t\beta^t\sum_{h,a}\Gamma_t(h,a)|d_t^F(h,a)-d_t^I(h,a)|.
\end{equation}
\item Suppose whole-horizon potential outcomes are well defined, observed outcomes satisfy consistency, stores do not interfere with one another, and matched-pair assignments are independent with probability one half.  Whole-horizon random assignment within each pair then makes the signed within-pair contrast unbiased for the finite-population average dynamic protocol effect.
\end{enumerate}
\end{proposition}
The result applies established potential-outcome and partial-identification logic to an adaptive protocol \citep{Robins1986,Murphy2003,ImbensRubin2015,Manski2003}.  The break-even disturbances in the replay are $1.09$, $1.33$, and $1.10$ residual standard deviations.  Current heterogeneity implies planning counts of 24, 20, and 12 matched pairs under 90\% target power, a two-sided 5\% level, and a fourfold design-effect allowance.  These are prospective design calculations, not a completed intervention; EC.5 gives the proofs and power formula.

\subsection{Computation of the benchmark}

The application uses the two computational objects only in their proper roles. Fixed-model vintage backward evaluation is polynomial by \cref{thm:vintage-backward-evaluation}; universal worst-reward certification is the separate hard problem in \cref{thm:por-cert-complexity}. On the enumerable instance, four formulations agree to $1.6\times10^{-13}$; the scanner-calibrated four-date benchmark has a feasible-incumbent/analytic-certificate bracket $[0.275,0.599]$, and a 64-date compression experiment attains error below $5\times10^{-5}$ using 14 simultaneous vintage coordinates. EC.3 gives formulations, scaling studies, and archived outputs.

\FloatBarrier

\section{Discussion and Conclusion}\label{sec:conclusion}

Relearning can replace the evaluator that supported an earlier plan. Under weak-evidence richness, the likelihood quotient characterizes provenance-preserving reconstruction. When the discrepancy is decision-visible, evaluator vintage enters the dynamic state; the triangular recursion is exact and its directed welfare gap is the Price of Relearning.

Vintage quotients remove operationally redundant coordinates, while continuous and finite-information lower bounds rule out the stated low-complexity representations uniformly over the unrestricted class. Specified finite models with explicitly listed one-step ambiguity laws are polynomially evaluable; universal worst-reward certification is NP-hard, with fixed-parameter tractability on the stated sparse-interaction subclass. Compatible design preserves provenance while choosing for statistical fidelity and welfare.

Operationally, there is no need to retain every vintage when future lower-envelope functionals coincide or admit a certified cover; absent such structure, the unrestricted results provide no uniform justification for discarding provenance, and compatible reconstruction is the coherence-preserving alternative.

The pricing benchmark exhibits the reconstruction--action--information loop; the identification analysis separates this mechanism from causal evidence. Thus, when rebuilding uncertainty changes the evaluator, provenance itself becomes economically relevant state information.

\section*{Data and Code Accessibility}
We acknowledge the Kilts Center for Marketing at the University of Chicago Booth School of Business for making the Dominick's data available for academic research.  A separate replication package has been prepared containing the code, archived computational outputs, dependency lock, source metadata, checksums, and README needed to reproduce the reported computational results.  The scanner data are retrievable for academic research.  The replication materials record a pinned immutable digital source, provide a download-and-verification script, and record the expected SHA-256 without redistributing the raw CSV.  Additional reproduction details are provided in EC.3 and EC.5.

\bibliographystyle{informs2014}
\bibliography{references}

\clearpage
\pagenumbering{arabic}
\setcounter{section}{0}
\setcounter{subsection}{0}
\setcounter{equation}{0}
\setcounter{figure}{0}
\setcounter{table}{0}
\setcounter{footnote}{0}
\setlength{\abovedisplayskip}{6pt plus 2pt minus 2pt}
\setlength{\belowdisplayskip}{6pt plus 2pt minus 2pt}
\setlength{\abovedisplayshortskip}{3pt plus 1pt}
\setlength{\belowdisplayshortskip}{4pt plus 1pt minus 1pt}

\begin{center}
{\Large\bfseries Electronic Companion to \emph{The Price of Relearning}\par}
\end{center}
\vspace{-0.4em}
\noindent Complete proofs and auxiliary evidence follow; references to main-paper results link directly to the corresponding labels.

\section{EC.1 Proofs for Bayes-Compatible Geometry}

For a centered log-likelihood increment $\ell\in\ClrH$, the associated map $\mathsf B_\ell:\Delta_m\to\Delta_m$ is
\[
\mathsf B_\ell(p)_i:=\frac{p_i e^{\ell_i}}{\sum_k p_k e^{\ell_k}},
\qquad
\mathsf B_\ell^{\#}[A]:=\{\mathsf B_\ell(q):q\in A\}.
\]
This is the Bayes map associated with the positive likelihood ratio $\exp(\ell)$ and its set image.

\subsection{Restricted-likelihood orbit representation}

Let $L\subseteq\ClrH$ be the attainable centered log-likelihood increments and define the generated additive subgroup by
\[
\Gamma(L):=\{0\}\cup\left\{\sum_{k=1}^n n_k\ell_k:n\ge1,\ n_k\in\mathbb Z,\ \ell_k\in L\right\}.
\]
A nonempty-valued map $\Psi:\ClrH\rightrightarrows\ClrH$ is $L$-equivariant if
$\Psi(x+\ell)=\Psi(x)+\ell$ for every $x$ and $\ell\in L$.

\begin{theorem}[Restricted-likelihood orbit representation]\label{thm:EC-orbit-representation}
The map $\Psi$ is $L$-equivariant if and only if there exists a unique set-valued map
$C:\ClrH/\Gamma(L)\rightrightarrows\ClrH$ such that
\[
\Psi(x)=x+C([x])\qquad\forall x\in\ClrH.
\]
Thus exact evidence compatibility permits the ambiguity shape to vary only across likelihood orbits.
\end{theorem}

\begin{proof}
Assume first that $\Psi$ is $L$-equivariant.  We show that equivariance extends from $L$ to the additive subgroup $\Gamma(L)$.  Repeated application gives
\[
\Psi\!\left(x+\sum_{j=1}^n\ell_j\right)
=\Psi(x)+\sum_{j=1}^n\ell_j
\]
for positive finite sums.  If $\ell\in L$, apply equivariance at $x-\ell$:
\[
\Psi(x)=\Psi(x-\ell)+\ell,
\qquad
\Psi(x-\ell)=\Psi(x)-\ell.
\]
Combining positive and negative increments yields
\[
\Psi(x+g)=\Psi(x)+g
\qquad\forall g\in\Gamma(L).
\]
Define the translated shape $S(x):=\Psi(x)-x:=\{y-x:y\in\Psi(x)\}$.  For every $g\in\Gamma(L)$,
\[
S(x+g)=\Psi(x+g)-(x+g)=\Psi(x)-x=S(x).
\]
Hence $S$ is constant on every coset of $\Gamma(L)$.  The map
\[
C([x]):=S(x),\qquad [x]\in\ClrH/\Gamma(L),
\]
is therefore well defined and satisfies $\Psi(x)=x+C([x])$.  If another map $\widetilde C$ had the same representation, subtracting $x$ would give $\widetilde C([x])=C([x])$, proving uniqueness.

Conversely, suppose $\Psi(x)=x+C([x])$.  For $\ell\in L\subseteq\Gamma(L)$, $[x+\ell]=[x]$, and therefore
\[
\Psi(x+\ell)=x+\ell+C([x+\ell])
=x+\ell+C([x])=\Psi(x)+\ell.
\]
Thus $\Psi$ is $L$-equivariant.
\end{proof}

\begin{corollary}[Dense likelihood directions restore rigidity]\label{cor:EC-dense-rigidity}
Suppose $\Psi$ is compact-valued and Hausdorff-continuous.  If
$\overline{\Gamma(L)}=\ClrH$, then there is a fixed compact set $C\subseteq\ClrH$ such that
$\Psi(x)=x+C$ for every $x$.  In particular, the conclusion holds if the attainable increments contain a nonempty open set.
\end{corollary}
\begin{proof}
The translated shape $S(x)=\Psi(x)-x$ is invariant under $\Gamma(L)$.  For arbitrary $x,y$, choose $g_n\in\Gamma(L)$ with $g_n\to y-x$.  Invariance gives $S(x+g_n)=S(x)$, while Hausdorff continuity gives $S(x+g_n)\to S(y)$; hence $S(y)=S(x)$.  If $L$ contains an open set, $L-L$ contains a neighborhood of zero, and an additive subgroup containing such a neighborhood is all of $\ClrH$.
\end{proof}

\subsection{Restricted cross-time compatibility and concentration}

\begin{proof}[Proof of \cref{thm:likelihood-quotient-characterization}]
Write
\[
S_t(x):=\Psi_t(x)-x=\{y-x:y\in\Psi_t(x)\}.
\]
Translation invariance of Hausdorff distance makes each $S_t$ compact-valued and continuous.  For $\ell_n\in L$ with $\ell_n\to0$, compatibility gives
\[
S_{t+1}(x+\ell_n)
=\Psi_{t+1}(x+\ell_n)-(x+\ell_n)
=\Psi_t(x)-x
=S_t(x).
\]
Continuity and $n\to\infty$ give $S_{t+1}(x)=S_t(x)$; write the common residual as $S$.

For $\ell\in L$, compatibility gives $S(x+\ell)=S(x)$, hence invariance under $\Gamma(L)$.  If $g_n\in\Gamma(L)$ and $g_n\to g\in G_L$, then
\[
S(x+g)=\lim_n S(x+g_n)=S(x),
\]
so $S$ is constant on every coset of the closed subgroup $G_L$.  Define
\[
\mathcal C(\pi_Lx):=S(x).
\]
This is representative independent; since $\pi_L$ is a quotient map and $\mathcal C\circ\pi_L=S$ is continuous, $\mathcal C$ is continuous and compact-valued, with
\[
\Psi_t(x)=x+\mathcal C(\pi_Lx).
\]
Uniqueness follows by subtracting $x$.  Conversely, the representation is time independent and $\pi_L(x+\ell)=\pi_Lx$ for $\ell\in L$, so
\[
\Psi_{t+1}(x+\ell)=x+\ell+\mathcal C(\pi_Lx)=\Psi_t(x)+\ell.
\]
Let $x^\star$ be a consistency anchor and choose $x_t\to x^\star$ as in the main paper.  The representation and translation invariance give
\[
d_{H,*}\!\left(
\mathcal C(\pi_Lx_t),\{x^\star-x_t\}
\right)\longrightarrow0.
\]
Continuity and $x^\star-x_t\to0$ give $\mathcal C(\pi_Lx^\star)=\{0\}$.  Density of $\pi_L(\mathcal A)$ then forces $\mathcal C(q)=\{0\}$ throughout the quotient.

For necessity, the quotient is metrizable; fix a compatible metric $d_Q$ and write
$F:=\overline{\pi_L(\mathcal A)}$.

If $F=\varnothing$, choose any nonzero $c_0\in\ClrH$ and let
\[
\mathcal C(q):=\{\alpha c_0:|\alpha|\le1\}
\qquad\text{for every quotient point }q.
\]
This continuous compact-valued compatible map is nondegenerate, and there are no anchor restrictions.

Suppose instead that $\varnothing\ne F$ is a proper closed subset.  Choose $q_0\notin F$, a nonzero $c_0\in\ClrH$, and set
\[
f(q):=\min\{1,d_Q(q,F)\},
\qquad
\mathcal C(q):=\{\alpha c_0:|\alpha|\le f(q)\}.
\]
Then $\mathcal C$ is compact-valued and continuous, vanishes on $F$, is nondegenerate at $q_0$, and induces an $L$-compatible constructor.  It equals $\{0\}$ at every anchor and therefore concentrates along every $x_t\to x^\star$.  Thus nondensity cannot force global degeneracy, including when the anchor set is empty.
\end{proof}

\subsection{Exact geometric and operational reconstruction defects}
\begin{proof}[Proof of \cref{thm:shape-defect}]
Write $x=\clr(p)$ and let the realized centered log-likelihood increment be $\ell$.  By Bayes translation, the inherited set in log-ratio coordinates is
\[
x+\ell+C_t,
\]
whereas the freshly reconstructed set is
\[
x+\ell+C_{t+1}.
\]
Hausdorff distance induced by any norm is translation invariant, so
\[
d_{H,*}(x+\ell+C_t,x+\ell+C_{t+1})=d_{H,*}(C_t,C_{t+1}),
\]
which proves the exact geometric identity.

For the operational estimate, let $\sigma=\softmax$.  Its Jacobian at $x$ is
\[
D\sigma(x)=\operatorname{diag}(\sigma(x))-\sigma(x)\sigma(x)^\top.
\]
For $h\in\ClrH$ and $\bar h=\sum_i\sigma_i(x)h_i$,
\[
\|D\sigma(x)h\|_1
=\sum_i\sigma_i(x)|h_i-\bar h|
\le 2\|h\|_\infty
\le 2c_*\|h\|_*.
\]
Integrating the derivative along the line segment from $x$ to $y$ gives
\[
\|\softmax(x)-\softmax(y)\|_1\le 2c_*\|x-y\|_*.
\]
Thus the softmax map sends the two log-ratio sets into probability sets whose $\ell_1$-Hausdorff distance is at most $2c_*d_{H,*}(C_t,C_{t+1})$.  Taking closed convex hulls cannot increase this bound.  The support-function duality in \cref{thm:operational-dual-v2} identifies the resulting $\ell_1$-Hausdorff distance with the operational defect, proving
\[
\Dop\bigl(\mathsf B_\ell^{\#}[\Phi_t(p)],\Phi_{t+1}(\mathsf B_\ell p)\bigr)
\le2c_*d_{H,*}(C_t,C_{t+1}).
\]
Finally, if $C_t=r_t\mathbb B_*$ and $C_{t+1}=r_{t+1}\mathbb B_*$, the Hausdorff distance of concentric norm balls is $|r_{t+1}-r_t|$.
\end{proof}

\section{EC.2 Decision-Visibility Proofs}

\begin{proof}[Proof of \cref{thm:operational-dual-v2}]
Let $C=\clco A$ and $D=\clco B$, and let $h_C(z)=\sup_{p\in C}p^\top z$.  Linear lower expectations satisfy
\[
\rho_A(z)=-h_C(-z),\qquad \rho_B(z)=-h_D(-z),
\]
so
\[
\Dop(A,B)=\sup_{\|z\|_\infty\le1}|h_C(z)-h_D(z)|.
\]
For compact convex $D$, norm duality gives
\[
\dist_1(x,D)=\sup_{\|z\|_\infty\le1}\{z^\top x-h_D(z)\}.
\]
Taking the supremum over $x\in C$ and interchanging the two suprema yields
\[
\sup_{x\in C}\dist_1(x,D)
=\sup_{\|z\|_\infty\le1}\{h_C(z)-h_D(z)\}.
\]
The reverse directed excess is obtained by exchanging $C$ and $D$.  Their maximum is the Hausdorff distance in $\ell_1$, proving the result.  Since $d_{\rm TV}(p,q)=\tfrac12\|p-q\|_1$, the total-variation statement follows.
\end{proof}

\begin{proof}[Proof of \cref{thm:parameter-predictive-bridge}]
For $v\in\mathbb R^n$,
\[
\inf_{p\in A}pPv=\inf_{\pi\in T_P(A)}\pi^\top v.
\]
Thus $\Dpred^P(A,B)$ is exactly the operational defect between the predictive images.  Applying \cref{thm:operational-dual-v2} and using linearity,
\[
\clco T_P(A)=T_P(\clco A),
\]
gives the exact Hausdorff representation.

For any signed row vector $x$, stochasticity of $P$ implies
\[
\|xP\|_1
\le \sum_i|x_i|\sum_jP_{ij}
=\|x\|_1.
\]
Hence $T_P$ is an $\ell_1$ contraction, which proves the upper bound after applying it to both directed Hausdorff excesses.

Let $C=\clco A$ and $D=\clco B$.  For $p\in C$, every difference $p-q$ with $q\in D$ lies in the zero-sum subspace $\ClrH$.  Therefore
\[
\dist_1(pP,T_P(D))
=\inf_{q\in D}\|(p-q)P\|_1
\ge \alpha(P)\inf_{q\in D}\|p-q\|_1.
\]
Taking the supremum over $p\in C$ controls one directed excess; exchanging $C,D$ controls the other.  The lower bound follows.

The unit $\ell_1$ sphere in $\ClrH$ is compact, so $\alpha(P)>0$ if and only if the restriction of $x\mapsto xP$ to $\ClrH$ has trivial kernel.  A nonzero zero-sum vector in that kernel is exactly a nontrivial affine dependence among the rows of $P$.
\end{proof}

\begin{proof}[Proof of \cref{thm:control-observability}]
Set $F(v):=|\inf_{p\in A}pPv-\inf_{q\in B}qPv|$.  Each lower expectation is one-Lipschitz in the sup norm, so $F$ is continuous and $\sup_{v\in\V}F(v)=\sup_{v\in\overline{\V}}F(v)$.  The upper bound follows because the reachable payoff class is contained in the unit sup-norm ball.  If $\chi(\V)=0$, the richness lower bound is trivial.  Otherwise fix $0<c<\chi(\V)$ and $\eta>0$, and choose $v$ in the unit ball whose predictive separation is at least $\Dpred^P(A,B)-\eta$.  Since $cv\in\overline{\V}$, the displayed closure identity and positive homogeneity give
\[
\Dctrl^{P,\V}(A,B)
\ge c\bigl(\Dpred^P(A,B)-\eta\bigr).
\]
Letting $\eta\downarrow0$ and $c\uparrow\chi(\V)$ gives the richness bound.

For the covering bound, choose such a near-separating $v$ and $w\in\overline{\V}$ with
\[
\|v-w\|_\infty\le\epsilon(\V)+\eta.
\]
For every ambiguity set $C$, the map
\[
v\mapsto\inf_{p\in C}pPv
\]
is one-Lipschitz in sup norm because $pP$ is a probability vector.  Replacing $v$ by $w$ therefore changes the difference of the two lower expectations by at most $2\|v-w\|_\infty$.  It follows that
\[
\Dctrl^{P,\V}(A,B)
\ge \Dpred^P(A,B)-2\epsilon(\V)-3\eta.
\]
Let $\eta\downarrow0$ and combine with nonnegativity.  The equality statements follow by setting $\chi(\V)=1$ or $\epsilon(\V)=0$.
\end{proof}

\subsection{Dynamic protocol and coherence proofs}

\begin{proof}[Proof of \cref{prop:static-parameter-rect}]
At a descendant history $h_j$, any $p\in M(h_j)$ assigns positive probability to $h_j$, and Bayes changes the mixing weights to $B_{t:j}(p)$.  The conditional next-observation law is the corresponding mixture, so varying $p$ gives exactly the displayed Bayes image.  Priors outside $M(h_j)$ put zero mass on the conditioning history.

Under full support, $M(h_j)=M$ at every feasible node.  The conditional family is version free, and its rectangular hull is the smallest closed path-law family permitting independent history-wise pasting; equality holds exactly when the committed family is already stable under those pastings.  Without full support the Bayes calculation remains valid on positive-probability domains, while null histories require specified conditional versions, taken as primitives in the equilibrium section.

Strict inclusion can occur under full support.  Let the first observation have law $(1/2,1/2)$ on $\{L,R\}$ under both parameters.  Conditional on either history, let the probability of the next outcome $1$ be $\varepsilon$ under parameter $1$ and $1-\varepsilon$ under parameter $2$, where $0<\varepsilon<1/2$, and take $M=\Delta_2$.  Any committed mixture has the same posterior weight after $L$ and $R$, whereas the rectangular hull can paste weights $0$ and $1$ at the two histories.  That pasted path law is outside the committed family.
\end{proof}

\begin{proof}[Proof of \cref{thm:universal-two-defect}]
If both operational defects vanish, the committed, rectangular-inherited, and fresh lower-expectation functionals agree on every bounded continuation payoff.  This proves statement (ii) of the theorem and immediately implies identical safe-versus-risky rankings and identical finite-menu choice correspondences.

Conversely, suppose universal safe-versus-risky coherence holds and, for example, the within-vintage defect is positive.  By the definition of operational distance there is a bounded payoff $Z$ with different committed and rectangular robust values.  Choose a constant $c$ strictly between those two values.  One evaluator strictly prefers $Z$ to $c$ and the other strictly prefers $c$ to $Z$, contradicting universal coherence.  The same argument applies to a positive cross-vintage defect.  Finally, absence of finite-menu separation implies safe-versus-risky coherence because two-action menus are a subclass.  The equivalences and the strict two-action separation follow.
\end{proof}

\begin{proof}[Proof of \cref{thm:global-dynamic-iff}]
Apply \cref{thm:universal-two-defect} at each continuation problem.  If all local within- and cross-vintage defects vanish, the three local evaluators agree on every bounded continuation payoff.  At the last decision date their optimal-action correspondences therefore coincide.  Suppose inductively that all continuation preference and policy correspondences coincide after date $t$.  Every current action then induces the same continuation payoff under the three protocols, and local functional equality gives the same current ranking.  Backward induction yields identical continuation preferences and policy correspondences at every subgame.

For the converse, if any local defect is positive, the local theorem supplies a bounded safe-versus-risky augmentation that produces strict opposite rankings at that subgame.  The global criterion quantifies over every subgame, including off-path histories, so this local separation already contradicts global agreement.  If one replaces the subgame-wise criterion by a root-only criterion, full-support reachability of the augmented node is sufficient to transmit the separation to the root.  Thus global agreement is equivalent to vanishing of both defect families.
\end{proof}

\begin{proof}[Proof of \cref{thm:bellman-recovery-v2}]
Fix $\pi$.  At $T$, every vintage assigns $g$.  If all active vintages share $V_{t+1}^{\pi}$ at the successors of $h_t$, Assumption~\ref{ass:kernel-coherence-v2} makes $\K_{s,t}(h_t,a)$ and $\K_t(h_t,a)$ induce the same lower expectation for $a=\pi_t(h_t)$.  Hence
\[
V_{s,t}^{\pi}(h_t)
=
r_t(h_t,a)
+\beta\inf_{q\in\K_t(h_t,a)}
\sum_yq(y)V_{t+1}^{\pi}(h_t,a,y),
\]
This is independent of $s$ and equals \cref{eq:coherent-policy-evaluation-v2}; the same argument before substituting the policy action makes every $Q_t^{s,\pi}(h_t,a)$ vintage independent.  Backward induction proves (i).

Let $V_t^*$ satisfy \cref{eq:bellman-v2} and let $\pi^*$ be an SPE.  At $T-1$, vintage-independent action values make the SPE action a Bellman maximizer.  If all successor subgames after $t$ have common value $V_{t+1}^*$, then for every current action
\[
Q_t^{t,\pi^*}(h_t,a)
=
r_t(h_t,a)
+\beta\inf_{q\in\K_t(h_t,a)}
\sum_yq(y)V_{t+1}^*(h_t,a,y).
\]
The one-period-deviation condition in \cref{thm:triangular-v2} makes $\pi_t^*(h_t)$ a maximizer, hence $V_{s,t}^{\pi^*}(h_t)=V_t^*(h_t)$ for all $s\le t$.  Thus every SPE is subgame-wise Bellman optimal with value $V_t^*$.

Conversely, a policy selecting a Bellman maximizer at every history has continuation value $V_{t+1}^*$ by backward induction and therefore satisfies the triangular one-period-deviation conditions, so it is an SPE.  The rectangularized initial-vintage precommitment problem has the same Bellman recursion and value.  Root-optimal policies may be arbitrary off path, so policy-set equality is asserted only for subgame-wise Bellman-optimal policies.  If the original vintage is a committed static-parameter family decision equivalent to its rectangular hull on every relevant subtree, their Bellman values also coincide.
\end{proof}

\section{EC.3 Exact Dynamic SPE Price-of-Relearning Representation}

\subsection{Sharp one-step bound}
\begin{proof}[Proof of \cref{thm:one-step-por-v2}]
Let $D=\Dop(A,B)$.  For inherited- and fresh-optimal actions $a_A,a_B$,
\begin{align*}
L(A,B)
&=\rho_A(z_{a_A})-\rho_B(z_{a_A})
 +\rho_B(z_{a_A})-\rho_B(z_{a_B})
 +\rho_B(z_{a_B})-\rho_A(z_{a_B})\\
&\le D+0+D.
\end{align*}
For sharpness, take
\[
A=\{(1/2+\varepsilon,1/2-\varepsilon)\},
\qquad
B=\{(1/2,1/2)\},
\]
so $\Dop(A,B)=2\varepsilon$.  Let $s=(1,-1)$ and use two actions
\[
z_1=\alpha s,
\qquad
z_2=-\alpha s+\delta\mathbf 1,
\]
with $\alpha+\delta\le1$.  Evaluator $B$ strictly selects action 2; evaluator $A$ selects action 1 whenever $\delta<4\alpha\varepsilon$, in which case the inherited loss from the fresh action is $4\alpha\varepsilon-\delta$.  Let $\alpha\uparrow1$ and $\delta\downarrow0$ to obtain a loss-to-distance ratio approaching $2$.
\end{proof}

\subsection{Why a triangular vintage state is required}
Self $t$ evaluates a deviation with $V_{t,t+1}^{\pi}$, not the successor self's generally different $V_{t+1,t+1}^{\pi}$.  Exact SPE representation therefore retains every active evaluator vintage.

\begin{proof}[Proof of \cref{thm:vintage-quotient}]
The relation in \cref{def:vintage-equivalence} is an equivalence relation because $\Dop=0$ means equality of the closed-convex-hull lower-expectation functionals, and equality of functionals is reflexive, symmetric, and transitive.

Fix equivalent vintages $s\sim_{t,h_t}s'$ and a policy $\pi$.  We prove equality of their values on every descendant subgame by backward induction.  At the terminal date,
$V_{s,T}^{\pi}=V_{s',T}^{\pi}=g$.  Suppose equality holds at every date-$(j+1)$ descendant of $h_t$.  At a date-$j$ descendant $h_j$, write $a=\pi_j(h_j)$ and let $v(y)$ be the common continuation value at $(h_j,a,y)$.  Operational equivalence at this node gives
\[
\inf_{q\in\K_{s,j}(h_j,a)}\sum_yq(y)v(y)
=
\inf_{q\in\K_{s',j}(h_j,a)}\sum_yq(y)v(y).
\]
The current reward is common to the two evaluators, so the recursive equations imply
$V_{s,j}^{\pi}(h_j)=V_{s',j}^{\pi}(h_j)$.  This proves part (i).

If $s\not\sim_{t,h_t}s'$, there is a descendant state--action triple $(j,h_j,a)$ at which the two lower-expectation functionals differ.  By \cref{cor:decision-equivalence-v2}, there is a bounded vector $z$ over the next observations with
\[
\inf_{q\in\K_{s,j}(h_j,a)}q^\top z
\ne
\inf_{q\in\K_{s',j}(h_j,a)}q^\top z.
\]
Local continuation richness supplies a bounded reward system on the successor subtrees and a common continuation policy whose next-date value is $c z(y)$ for every active vintage, for some $c>0$ small enough to satisfy the reward bounds.  Choose action $a$ at $h_j$ and use zero current reward.  The two date-$j$ values then differ by $\beta c$ times the displayed separation.  Pasting arbitrary admissible rewards and actions outside this branch yields one reward system and one policy on the full subtree.  This proves part (ii).

Part (i) shows that all valuation vectors are constant on the equivalence classes, so one coordinate per class is sufficient.  Part (ii) shows that identifying two distinct classes would force equality in a continuation problem where their values are unequal.  Hence no universal coordinate-identification rule can merge distinct classes under local continuation richness, proving part (iii).
\end{proof}

\begin{proof}[Proof of \cref{thm:exact-dynamic-por-representation}]
Fix a history $h_t$ and define the actual equilibrium-attainable set
\[
\mathcal A_t(h_t)
:=
\left\{
\left(P_t^r(h_t);
V_{0,t}^{\pi}(h_t),\ldots,V_{t,t}^{\pi}(h_t)\right):
 r\in\mathscr R_t(h_t),\ 
 \pi\in\mathrm{SPE}_t(r;h_t)
\right\}.
\]
We prove $\mathfrak S_t(h_t)=\mathcal A_t(h_t)$ at every history by backward induction on the remaining horizon.

\paragraph{Base date.}
At $t=T$, an admissible subtree reward system consists only of a terminal reward $z=g(h_T)$ with $|z|\le G$.  There is no action choice, so every evaluator and the precommitment problem have value $z$.  Hence
\[
\mathcal A_T(h_T)
=
\{(z;z,\ldots,z):|z|\le G\}
=
\mathfrak S_T(h_T).
\]

\paragraph{Actual reward--equilibrium pair $\Rightarrow$ attainable vector.}
Assume the equality holds at every successor history and fix $r\in\mathscr R_t(h_t)$ together with $\pi\in\mathrm{SPE}_t(r;h_t)$.  A subgame-perfect policy specifies behavior after every current action, including off-path actions.  Therefore, for each action $a$ and successor $y$, the restriction of $(r,\pi)$ to the subtree rooted at $(h_t,a,y)$ is an admissible reward system paired with a pure SPE of that successor subgame.  By the induction hypothesis,
\[
\left(
P_{t+1}^r;
V_{0,t+1}^{\pi},\ldots,V_{t+1,t+1}^{\pi}
\right)(h_t,a,y)
\in
\mathfrak S_{t+1}(h_t,a,y).
\]
Choose these successor vectors and $c_a=r_t(h_t,a)$ in the recursive construction.  Then
\[
P_a
=r_t(h_t,a)
+\beta\inf_{q\in\K_{0,t}(h_t,a)}
\E_q[P_{t+1}^r(H_{t+1})]
\]
is the vintage-$0$ Bellman action value, so $P_t^r(h_t)=\max_aP_a$.  For every $s\le t$,
\[
U_{s,a}
=r_t(h_t,a)
+\beta\inf_{q\in\K_{s,t}(h_t,a)}
\E_q[V_{s,t+1}^{\pi}(H_{t+1})]
=Q_t^{s,\pi}(h_t,a).
\]
Because $\pi$ is an SPE, \cref{thm:triangular-v2} gives
\[
\pi_t(h_t)\in\argmax_aU_{t,a}.
\]
Taking $a_F=\pi_t(h_t)$ in the attainable recursion therefore yields
\[
\left(P_t^r(h_t);
V_{0,t}^{\pi}(h_t),\ldots,V_{t,t}^{\pi}(h_t)\right)
\in\mathfrak S_t(h_t).
\]
Thus $\mathcal A_t(h_t)\subseteq\mathfrak S_t(h_t)$.

\paragraph{Attainable vector $\Rightarrow$ one reward system and one SPE.}
Take $(p;u_0,\ldots,u_t)\in\mathfrak S_t(h_t)$.  By construction there are current rewards $c_a\in[-R_t,R_t]$, successor vectors
\[
(p_{a,y};u_{0,a,y},\ldots,u_{t+1,a,y})
\in\mathfrak S_{t+1}(h_t,a,y),
\]
and a current action $a_F\in\argmax_aU_{t,a}$ that generate the selected vector.  By the induction hypothesis, for every $(a,y)$ there exist a reward system $r^{a,y}$ and a pure successor SPE $\pi^{a,y}$ realizing the selected successor vector.

The successor subtrees are disjoint.  Branchwise reward richness therefore allows the $r^{a,y}$ to be pasted into a single reward system $\bar r$, with $\bar r_t(h_t,a)=c_a$.  Paste the policies $\pi^{a,y}$ on the corresponding successor subgames and set $\bar\pi_t(h_t)=a_F$.  Every proper descendant subgame is governed by one of the selected branch SPEs.  At the current history, the branch valuations reproduce exactly the arrays $(U_{s,a})_a$, and $a_F\in\argmax_aU_{t,a}$ by construction.  Hence \cref{thm:triangular-v2} implies
\[
\bar\pi\in\mathrm{SPE}_t(\bar r;h_t).
\]
The vintage-$0$ precommitment value is $\max_aP_a=p$, while the vintage evaluations of $\bar\pi$ are $u_0,\ldots,u_t$.  Therefore
\[
(p;u_0,\ldots,u_t)
=
\left(P_t^{\bar r};V_{0,t}^{\bar\pi},\ldots,V_{t,t}^{\bar\pi}\right)(h_t)
\in\mathcal A_t(h_t).
\]
This proves the reverse inclusion.

\paragraph{Worst-equilibrium divergence and homogeneity.}
Equality of the correspondences gives
\[
\mathfrak d_t^{\rm RL,SPE}(h_t)
=
\sup_{r\in\mathscr R_t(h_t)}
\sup_{\pi\in\mathrm{SPE}_t(r;h_t)}
\left[P_t^r(h_t)-V_{0,t}^{\pi}(h_t)\right].
\]
Each bracket is nonnegative because $P_t^r$ maximizes the vintage-$0$ rectangular value over all feasible policies.  For $c>0$, scaling all rewards and reward bounds by $c$ scales every action value by $c$ and leaves every argmax correspondence, hence every pure SPE set, unchanged.  For $c=0$, every attainable valuation vector is zero.  Thus the correspondence and its directed divergence are positively homogeneous.
\end{proof}

\subsection{A diagnostic counterexample to diagonal nesting}
The distinction above is substantive.  Take $T=2$, $\beta\in(0,1]$, two deterministic date-0 actions $A,B$, and one action at each date-1 branch.  All stage rewards are zero.  On branch $A$, let the terminal payoff be $g(1)=1$, $g(2)=-1$; on branch $B$, let $g(1)=g(2)=0$.  At date 1 let vintage 0 use $\K_{0,1}=\{\delta_1\}$ and the fresh date-1 self use $\K_{1,1}=\{\delta_2\}$.  Then
\[
V_{0,1}(A)=\beta,
\qquad
V_{1,1}(A)=-\beta,
\qquad
V_{0,1}(B)=V_{1,1}(B)=0.
\]
Self 0 evaluates the continuation after $A$ with $V_{0,1}$ and therefore selects $A$, as does the vintage-0 precommitment problem; the genuine SPE Price of Relearning is zero.  A diagonal nesting that instead substitutes the successor self's $V_{1,1}$ assigns $A$ value $-\beta^2$ and selects $B$, falsely reporting a positive loss $\beta^2$.  This is why the main theorem carries the triangular vintage vector.

\begin{proof}[Proof of \cref{prop:stagewise-spe-embedding}]
Fix $m$ and choose $a_u\in M_u^B(m_u)$ at every date.  Define $\pi_u(h_u)\equiv a_u$ at every date-$u$ history.  We verify the SPE conditions backward.  At date $T-1$, the only action-dependent continuation term is $\beta\rho_{B_{T-1}}(z_{T-1,a})$, so $a_{T-1}$ is optimal.  Suppose the candidate policy is already fixed and history invariant after date $u$.  By the stagewise assumptions, neither the current action nor $Y_{u+1}$ changes any later primitive; because the future policy is also history invariant, the value of the continuation after the one-step payoff is the same for every current action and successor history.  Translation equivariance therefore reduces the date-$u$ action comparison to
\[
\beta\rho_{B_u}(z_{u,a})+\text{a common continuation constant}.
\]
Hence $a_u\in M_u^B(m_u)$ is optimal.  Backward induction proves $\pi\in\mathrm{SPE}_{\rm full}(m)$.

Now suppose every $M_u^B(m_u)$ is a singleton, say $\{a_u^*\}$.  At date $T-1$, every pure SPE must use $a_{T-1}^*$ at every history.  If all SPE actions after date $u$ have already been shown to be history invariant and equal to their unique datewise maximizers, then the same common-continuation argument applies to any candidate SPE at date $u$ and forces $a_u^*$ at every history.  Backward induction gives $\mathrm{SPE}_{\rm full}(m)=\mathrm{SPE}^{\rm stg}(m)$.
\end{proof}

\begin{proof}[Proof of \cref{cor:stagewise-matching-por}]
Fix a menu tuple $m$ and a history-invariant stagewise equilibrium $\pi\in\mathrm{SPE}^{\rm stg}(m)$, with $\pi_u(h_u)\equiv a_u$.  The class is a literal specialization of the master timing: set $r_0=0$, set
\[
r_t(h_t,a_t)=z_{t-1,a_{t-1}}(Y_t),\qquad 1\le t<T,
\]
and set $g(h_T)=z_{T-1,a_{T-1}}(Y_T)$.  The realized $r_t$ term is common across all date-$t$ actions.  The bounds $C_{t-1}\le R_t$ and $C_{T-1}\le G$ make the embedding admissible.

Because the selected action at every future date is constant across histories and current actions do not alter later primitives, backward induction with translation equivariance gives
\begin{align*}
P_0(m)
&=\sum_{u=0}^{T-1}\beta^{u+1}\max_{a\in\A_u}\rho_{A_u}(z_{u,a}),\\
V_{0,0}^{\pi}(m)
&=\sum_{u=0}^{T-1}\beta^{u+1}\rho_{A_u}(z_{u,a_u}).
\end{align*}
Therefore
\begin{equation}\label{eq:EC-stagewise-policy-decomposition}
\PoR_0(m,\pi)
=\sum_{u=0}^{T-1}\beta^{u+1}L_u(m_u,a_u),
\qquad
L_u(m_u,a_u):=\max_a\rho_{A_u}(z_{u,a})-\rho_{A_u}(z_{u,a_u}).
\end{equation}
At $u=0$, $A_0=B_0$ and $a_0\in M_0^B(m_0)$, so $L_0=0$.

For fixed $m$, the correspondence in \cref{eq:stagewise-spe-correspondence} permits the action $a_u$ to be chosen independently from $M_u^B(m_u)$ at each date.  Hence
\[
\sup_{\pi\in\mathrm{SPE}^{\rm stg}(m)}\PoR_0(m,\pi)
=
\sum_{u=1}^{T-1}\beta^{u+1}
\sup_{a_u\in M_u^B(m_u)}L_u(m_u,a_u).
\]
Define
\[
r_u^*:=\sup_{m_u\in\mathscr M_u}\sup_{a_u\in M_u^B(m_u)}L_u(m_u,a_u)
=\mathfrak r_u(A_u,B_u).
\]
The upper bound
\[
\sup_{m\in\prod_u\mathscr M_u}\sup_{\pi\in\mathrm{SPE}^{\rm stg}(m)}\PoR_0(m,\pi)
\le\sum_{u=1}^{T-1}\beta^{u+1}r_u^*
\]
is immediate.  For the reverse inequality, let
\[
I_+:=\{u\in\{1,\ldots,T-1\}:\beta^{u+1}>0\},\qquad n_+:=|I_+|.
\]
If $n_+=0$, both sides are zero.  Otherwise, for any $\varepsilon>0$ and each $u\in I_+$ choose $m_u^\varepsilon$ and $a_u^\varepsilon\in M_u^B(m_u^\varepsilon)$ with
\[
L_u(m_u^\varepsilon,a_u^\varepsilon)
>r_u^*-\frac{\varepsilon}{n_+\beta^{u+1}}.
\]
Choose arbitrary admissible menus at zero-weight dates.  Cartesian-product feasibility permits all selected menus simultaneously, and \cref{prop:stagewise-spe-embedding} pastes the selected datewise maximizers into one history-invariant pure SPE.  Equation \eqref{eq:EC-stagewise-policy-decomposition} then gives a root loss exceeding
\[
\sum_{u=1}^{T-1}\beta^{u+1}r_u^*-\varepsilon.
\]
Letting $\varepsilon\downarrow0$ proves \cref{eq:stagewise-exact-sum}.  The inclusion $\mathrm{SPE}^{\rm stg}(m)\subseteq\mathrm{SPE}_{\rm full}(m)$ gives the unrestricted lower-bound statement in the corollary; if all datewise current-self maximizers are unique, \cref{prop:stagewise-spe-embedding} gives equality of the two SPE correspondences and hence the full-SPE version of the exact sum.

For every menu class, the one-step upper bound gives
\[
r_u^*\le 2C_u\Dop(A_u,B_u).
\]
If $\mathscr M_u$ contains every finite $C_u$-bounded menu, the lower bound follows from an explicit two-action construction.  Let $D=\Dop(A_u,B_u)$ and choose $z$ with $\|z\|_\infty\le1$ such that $|\rho_{A_u}(z)-\rho_{B_u}(z)|\ge D-\eta$.  Offer the payoff $C_uz$ and the constant payoff $C_u\rho_{B_u}(z)\mathbf 1$.  The two actions tie under $B_u$.  Under $A_u$, one is better by
\[
C_u|\rho_{A_u}(z)-\rho_{B_u}(z)|.
\]
Because the stagewise loss takes the supremum over all fresh maximizers, it may select the $A_u$-inferior member of the tie.  Letting $\eta\downarrow0$ gives
\[
C_u\Dop(A_u,B_u)\le r_u^*\le2C_u\Dop(A_u,B_u).
\]
The constant action is admissible because $|\rho_{B_u}(z)|\le1$.  The lower coefficient is attained exactly.  On two outcomes take $A_u=\Delta_2$, $B_u=\{\delta_1\}$, so $\Dop(A_u,B_u)=2$, and offer
\[
z_1=(C_u,C_u),
\qquad
z_2=(C_u,-C_u).
\]
Both actions have $B_u$-value $C_u$, so the worst fresh tie selects $z_2$.  Their $A_u$-values are $C_u$ and $-C_u$, giving inherited regret
$2C_u=C_u\Dop(A_u,B_u)$.  Under this full-menu condition, the constants $1$ and $2$ in \cref{eq:stagewise-two-sided} are therefore unimprovable as universal constants over the product model class.  Indeed, at any post-initial date with positive discount weight, embed respectively a one-step lower-sharp or upper-sharp pair and set $A_v=B_v$ at all other dates.  Cartesian-product feasibility then reduces the aggregate ratio to the corresponding one-step ratio.  Since the history-invariant benchmark is a subset of the unrestricted SPE correspondence, the lower embedding also proves that the accumulated one-step terms in the general worst-equilibrium theory cannot be replaced by an order-smaller universal quantity.
\end{proof}

\subsection{Endpoint insufficiency without equilibrium multiplicity}

\begin{proof}[Proof of \cref{prop:endpoint-insufficiency}]
Take $T=3$ and $\beta=1$.  Date $0$ has one action and a deterministic transition to the date-$1$ history.  At date $1$, actions $L$ and $R$ lead deterministically to distinct date-$2$ histories.  Date $2$ has one feasible action at each history and two terminal outcomes.  All stage rewards are zero.  Let
\[
z_L=(0,2),
\qquad
z_R=(1,0),
\qquad
\mu=(1/2,1/2).
\]
The date-$1$ endpoint kernels are identical and the date-$2$ endpoint kernels satisfy
\[
\K_{0,2}=\K_{2,2}=\{\mu\}.
\]
The first triangular system $\mathbb K$ differs only through
\[
\K_{1,2}=\{\delta_1\}.
\]
There is no continuation choice at date $2$.  Self $1$ values $L$ at $0$ and $R$ at $1$, so $R$ is its unique action and the pure SPE is unique.  Vintage $0$ evaluates the two branches under $\mu$: it assigns value $1$ to $L$ and $1/2$ to $R$.  Its rectangular precommitment policy therefore uniquely selects $L$, whereas its value of the unique SPE is $1/2$.  Hence
\[
\mathfrak d_0^{\rm RL,SPE}(\mathbb K)\ge\frac12.
\]
Both action comparisons have strict gaps.

Define $\widetilde{\mathbb K}$ by leaving every endpoint kernel unchanged and replacing only the off-diagonal kernel by
\[
\widetilde\K_{1,2}=\{\mu\}.
\]
All vintages are then decision equivalent at every node.  By \cref{thm:bellman-recovery-v2}, every pure SPE is optimal for vintage $0$'s rectangular precommitment problem for every bounded reward system, and
\[
\mathfrak d_0^{\rm RL,SPE}(\widetilde{\mathbb K})=0.
\]
The two systems have identical endpoint collections but different exact divergences.  Since the positive-loss construction has a unique pure SPE and strict action gaps, endpoint insufficiency does not rely on tie-breaking or equilibrium multiplicity.
\end{proof}

\begin{proof}[Proof of \cref{thm:endpoint-robustness}]
Let $\delta=\Dop(A,B)>0$ and fix $\xi\in(0,\delta)$.  By the definition of the operational distance, after interchanging $A$ and $B$ if necessary, there is a payoff vector $z$ with $\|z\|_\infty\le1$ such that
\[
\Delta:=\rho_A(z)-\rho_B(z)>\delta-\xi.
\]
Let
\[
c:=\frac{\rho_A(z)+\rho_B(z)}2,
\qquad
g:=\frac\Delta2.
\]
Because robust evaluation of a constant equals that constant, evaluator $B$ strictly prefers the constant action $C$ with payoff $c$ to the risky action $Z$ with payoff $z$, and evaluator $A$ strictly prefers $Z$ to $C$.  Both strict gaps equal $g>(\delta-\xi)/2$.

Embed this menu after a deterministic date-$0$ prefix in a three-date triangular system.  Keep the initial and current endpoint evaluator equal to $B$, and use $A$ only as the relevant off-diagonal evaluator.  The off-diagonal self uniquely chooses $Z$, whereas the initial-vintage rectangular precommitment uniquely chooses $C$.  The initial vintage values the selected continuation at $\rho_B(z)$, so its relearning loss is
\[
c-\rho_B(z)=g.
\]
In the endpoint-collapsed system replace the off-diagonal evaluator by $B$.  All vintages then choose $C$ and are decision equivalent, so the Price of Relearning is zero for every reward system.

Now perturb the payoff vectors and the two evaluators jointly, and construct the perturbed endpoint surrogate by replacing the perturbed off-diagonal evaluator $A'$ by the perturbed endpoint evaluator $B'$.  Suppose that, for each action and evaluator, the perturbed robust action value differs from its baseline value by at most $\omega$.  A difference of two action values can move by at most $2\omega$.  Hence both strict gaps in the full system remain at least $g-2\omega$, and the initial-vintage loss from the selected risky action remains at least the same amount.  The coupled endpoint surrogate uses $B'$ for every relevant vintage and therefore remains decision equivalent, so its Price of Relearning is zero.  If $2\omega<g$, the full-system choices remain unique and
\[
\PoR_0-\PoR_0^{\rm end}\ge g-2\omega>0.
\]
Finally, if each payoff vector changes by at most $\varepsilon$ in sup norm and the corresponding ambiguity sets change by at most $\kappa$ in $D_{\rm op}$, payoff Lipschitzness and set Lipschitzness give
\[
|\rho_{E'}(z')-\rho_E(z)|
\le \|z'-z\|_\infty
 +\|z\|_\infty D_{\rm op}(E',E)
\le\varepsilon+\kappa
\]
for the risky action; constants satisfy the same bound without the second term.  Thus $\omega\le\varepsilon+\kappa$, proving the explicit perturbation statement.  Since all inequalities are strict, the construction contains a nonempty open neighborhood of systems with the same endpoint failure.
\end{proof}

\subsection{Why unrestricted stagewise SPEs require a selection condition}

The history-invariant restriction in the stagewise product benchmark cannot generally be removed. Take $T=3$, $\beta=1$, a deterministic date-0 transition, and date-1 actions $L,R$ with payoffs
\[
z_{1,L}\equiv1,
\qquad
z_{1,R}\equiv0,
\qquad
\K_{0,1}=\K_{1,1}.
\]
At date $2$, actions $x$ and $y$ have terminal payoff vectors
\[
z_{2,x}=(1,-1),
\qquad
z_{2,y}=(-1,1),
\qquad
\K_{0,2}=\K_{2,2}=\{(1/2,1/2)\}.
\]
Thus $A_u=B_u$ at every post-initial date, and both date-$2$ actions are tied for the current self.  Set $\K_{1,2}=\{\delta_1\}$ and let the date-$2$ rule choose $y$ after $L$ and $x$ after $R$.  Self $1$ then values $L$ at $1+(-1)=0$ and $R$ at $0+1=1$, so the history-dependent rule is a pure SPE.  Vintage $0$ precommits to $L$ and obtains $1$, but values this SPE at $0$.  Therefore the root loss is one even though
\[
\mathfrak r_1(A_1,B_1)=\mathfrak r_2(A_2,B_2)=0.
\]
A payoff-irrelevant history selects among an indifferent future self's actions, and the omitted off-diagonal vintage changes the earlier action comparison.  This equilibrium-selection channel is captured by the triangular correspondence but not by an endpoint product formula.

\subsection{Observable cost of binary concentration}

\begin{corollary}[Observable cost of binary concentration]\label{cor:EC-observable-concentration}
Let $b(x)=(e^x/(1+e^x),1/(1+e^x))$.  Fix $r_0>0$, $0\le r_u\le r_0$, and $|x_u|+r_0\le M$.  Let
\[
\M_u^I=\{b(x_u+c):|c|\le r_0\},\qquad
\M_u^F=\{b(x_u+c):|c|\le r_u\},
\]
and let $A_u=T_{P_u}(\M_u^I)$ and $B_u=T_{P_u}(\M_u^F)$.  If every finite $C_u$-bounded menu is feasible, then
\[
\sup_m\sup_{\pi\in\mathrm{SPE}_{\rm full}(m)}\PoR_0(m,\pi)
\ge 2\lambda_M\sum_{u=1}^{T-1}\beta^{u+1}C_u\alpha(P_u)(r_0-r_u),
\quad \lambda_M=(2+2\cosh M)^{-1}.
\]
\end{corollary}

\begin{proof}[Proof of \cref{cor:EC-observable-concentration}]
Let $\sigma(x)=e^x/(1+e^x)$, so $b(x)=(\sigma(x),1-\sigma(x))$.  On $[-M,M]$,
\[
\sigma'(x)=\sigma(x)(1-\sigma(x))
=\frac{1}{2+2\cosh x}
\ge\lambda_M.
\]
Because $\M_u^F\subseteq\M_u^I$, the distance from the right outer endpoint of $\M_u^I$ to $\M_u^F$ is attained at the right endpoint of $\M_u^F$.  Hence
\begin{align*}
\Dop(\M_u^I,\M_u^F)
&=d_H^{\ell_1}(\M_u^I,\M_u^F)\\
&\ge \|b(x_u+r_0)-b(x_u+r_u)\|_1\\
&=2\{\sigma(x_u+r_0)-\sigma(x_u+r_u)\}\\
&\ge2\lambda_M(r_0-r_u),
\end{align*}
where the first equality uses \cref{thm:operational-dual-v2} and convexity of the binary probability intervals.  Applying the lower half of the predictive bridge gives
\[
\Dop(A_u,B_u)
=\Dpred^{P_u}(\M_u^I,\M_u^F)
\ge\alpha(P_u)\Dop(\M_u^I,\M_u^F)
\ge2\lambda_M\alpha(P_u)(r_0-r_u).
\]
The full-menu lower bound in \cref{eq:stagewise-two-sided} now yields
\[
\sup_m\sup_{\pi\in\mathrm{SPE}^{\rm stg}(m)}\PoR_0(m,\pi)
\ge
2\lambda_M\sum_{u=1}^{T-1}\beta^{u+1}C_u\alpha(P_u)(r_0-r_u).
\]
Finally, $\mathrm{SPE}^{\rm stg}(m)\subseteq\mathrm{SPE}_{\rm full}(m)$, so the same lower bound holds for the unrestricted worst-equilibrium quantity.  A positive summand proves strict positivity.
\end{proof}

\begin{proof}[Proof of \cref{thm:vintage-memory-irreducibility}]
Fix $t\ge1$ and consider a subtree with one decision at date $t$, discount factor $\beta=1$, zero current rewards, and terminal reward bound $G=1$.  There are two current actions, $A$ and $B$, and terminal outcomes $\{0,1,\ldots,t\}$.  For action $A$, set
\[
\K_{s,t}(A)=\{\delta_s\},\qquad s=0,\ldots,t,
\]
and choose $g(A,s)=x_s$ with $x_s\in[-1/4,1/4]$.  For action $B$, set
\[
\K_{0,t}(B)=\{\delta_0\},\qquad
\K_{t,t}(B)=\{\delta_t\},
\]
use arbitrary nonempty kernels for intermediate vintages, and choose
\[
g(B,0)=p\in[1/2,3/4],
\qquad
g(B,y)=-3/4\quad(y=1,\ldots,t).
\]
The date-$t$ self values $A$ at $x_t\ge-1/4>-3/4$ and $B$ at $-3/4$, so $A$ is its unique action.  Every selected-policy vintage value is therefore $V_{s,t}=x_s$.  Vintage $0$ values $B$ at $p\ge1/2>1/4\ge x_0$ and $A$ at $x_0$, so its precommitment value is $P_t=p$.  This realizes the closed rectangle
\[
[1/2,3/4]\times[-1/4,1/4]^{t+1},
\]
whose interior is the open rectangle in the main paper.  The closed cube $Q_t$ in the main paper is contained in this rectangle.

Set $n=t+2$ and translate $Q_t$ to the centered cube $[-a,a]^n$ with $a=1/8$.  Let $E:Q_t\to\mathbb R^d$ be continuous with $d<n$.  The boundary of the cube is antipodally homeomorphic to $S^{n-1}$: one may map a sphere point radially to the unique point on the cubical boundary, and this map commutes with sign reversal.  Compose this homeomorphism with $E$ and, when $d<n-1$, append zero coordinates to view the range as $\mathbb R^{n-1}$.  The Borsuk--Ulam theorem \citep{EC-Matousek2003} gives an antipodal pair $x,-x$ on the cubical boundary with
\[
E(x)=E(-x).
\]
Every boundary point satisfies $\|x\|_\infty=a$.  For any decoder $D$, put $y=D(E(x))=D(E(-x))$.  The triangle inequality yields
\[
2a=\|x-(-x)\|_\infty
\le \|x-y\|_\infty+\|y-(-x)\|_\infty.
\]
At least one of the two reconstruction errors is therefore at least $a=1/8$.  This proves the uniform approximate lower bound; exact recovery is the zero-error special case.  Multiplying all rewards and reward bounds by $c$ multiplies the attainable cube and the lower bound by $c$.

For the decision-query interpretation, choose a coordinate $j$ with $|x_j|=a$ and compare that coordinate with its cube center by a binary safe-versus-risky menu.  The correct strict action is opposite at $x$ and $-x$, but the code is identical.  Any decision rule based only on the code must therefore err on one state with margin $a$.
\end{proof}

\begin{proposition}[Polyhedral attainable correspondence]\label{prop:EC-polyhedral-attainable}
Suppose every one-step kernel set is a polytope and the reward class is a box.  Then $\mathfrak S_t(h_t)$ is a finite union of compact polytopes at every history, and the maximum of $p-u_0$ over that correspondence is attained.  If all primitive data are rational, this maximum admits an exact finite mixed-integer linear formulation.  The formulation can be exponential in the finite model description.
\end{proposition}
\begin{proof}
We proceed by backward induction.  At a terminal history,
\[
\mathfrak S_T(h_T)=\{(z;z,\ldots,z):|z|\le G\}
\]
is a compact polytope.  Suppose every successor correspondence is a finite union of compact polytopes.  For each action--successor pair $(a,y)$, select one successor polytope component.  Because every one-step kernel set is a polytope, its lower expectation of a continuation vector is the minimum of finitely many affine functions, one for each extreme point.  Select one active extreme point for every lower expectation appearing in the recursion, one precommitment-maximizing action, and one current-self maximizing action.  Conditional on this finite active pattern, the following requirements are all linear:
\begin{enumerate}[label=(\roman*)]
\item membership of each successor vector in its selected polytope component;
\item current reward-box constraints;
\item inequalities making the selected kernel vertex attain each lower expectation;
\item inequalities making the selected precommitment action maximize $P_a$;
\item inequalities making the selected current-self action maximize $U_{t,a}$.
\end{enumerate}
The recursively produced vector is an affine image of the resulting bounded polyhedron.  It is therefore a compact polytope.  There are finitely many actions, successor components, and kernel vertices, so taking the union over all active patterns yields a finite union of compact polytopes.  The inequalities are weak, so ties and all equilibrium selections are included.  This proves the structural claim at date $t$ and completes the induction.

The objective $p-u_0$ is continuous.  A finite union of compact sets is compact, so the supremum defining $\mathfrak d_t^{\rm RL,SPE}$ is attained.  If the primitive data are rational, every active-pattern polytope has a rational linear description.  Introduce one binary selector per active pattern (or an equivalent extended disjunctive encoding), impose the associated linear system conditionally, and maximize $p-u_0$.  Standard disjunctive-programming formulations \citep{EC-Balas1979} represent the finite union exactly.  The number of active patterns can grow exponentially with the tree, actions, vertices, and successor components; the proposition therefore makes no polynomial-complexity claim.
\end{proof}

\begin{proposition}[Policy-cell performance certificate]\label{prop:EC-policy-cell-certificate}
Let $\chi$ be a partial pure policy, and suppose
$\epsilon_t(h,a)$ uniformly bounds
$|Q_t^{0,\pi}(h,a)-Q_t^{t,\pi}(h,a)|$ for every continuation policy $\pi$.  Define $b_T^\chi=0$ and, backward in time,
\[
e_t^\chi(h)=
\begin{cases}
\max_a\epsilon_t(h,a)+\epsilon_t(h,\chi(h)),&h\in\operatorname{dom}\chi,\\
2\max_a\epsilon_t(h,a),&h\notin\operatorname{dom}\chi,
\end{cases}
\]
and
\[
b_t^\chi(h)=e_t^\chi(h)+\beta
\max_a\sup_{q\in K_{0,t}(h,a)}
\sum_yq(y)b_{t+1}^\chi(h,a,y).
\]
Every pure SPE $\pi$ extending $\chi$ satisfies
\[
0\le P_t(h)-V_{0,t}^{\pi}(h)\le b_t^\chi(h).
\]
If $\chi'$ extends $\chi$, then $b_t^{\chi'}(h)\le b_t^\chi(h)$ at every ancestor whose continuation is further fixed.  Consequently, if a finite collection of policy cells partitions the pure-SPE graph and cell $c$ has certified lower and upper bounds $L_c,U_c$, then
\[
\max_cL_c\le\mathfrak d_0^{\rm RL,SPE}\le\max_cU_c.
\]
The bounds remain valid under time limits, and become exact when every surviving cell closes.
\end{proposition}

\begin{proof}
Fix an SPE $\pi$ extending $\chi$.  At a history $h$ let $a_0$ maximize the initial-vintage action value and let $a_t=\pi_t(h)$ maximize the current-self action value.  Add and subtract the two current-self values to obtain
\[
Q_t^{0,\pi}(h,a_0)-Q_t^{0,\pi}(h,a_t)
\le\epsilon_t(h,a_0)+\epsilon_t(h,a_t).
\]
If $h\in\operatorname{dom}\chi$, then $a_t=\chi(h)$ and the right side is at most $e_t^\chi(h)$; otherwise it is at most twice the largest action-specific deviation.  Applying the robust performance-residual recursion gives the displayed recursion for $b_t^\chi$.  The continuation maximum remains over all actions because the initial-vintage precommitment action need not equal the fixed SPE action.  Backward induction proves the value bound.

Fixing an additional policy action replaces $2\max_a\epsilon_t(h,a)$ by
$\max_a\epsilon_t(h,a)+\epsilon_t(h,\chi'(h))$, which is no larger; the continuation operator is unchanged but is monotone in the smaller descendant certificates.  Backward induction gives the refinement claim.  Finally, the policy cells form a disjoint cover of the finite pure-policy graph.  Maximizing valid cell lower and upper bounds therefore gives valid global bounds; if every cell is solved or pruned by its upper bound, the two maxima coincide with the exact optimum.
\end{proof}

\subsection{Approximate vintage quotients and age-forgetting compression}
\label{sec:EC-approximate-vintage-quotient}

\begin{proof}[Proof of \cref{thm:approximate-vintage-quotient}]
Fix $s,s'\le t$ and a continuation policy $\pi$.  For a descendant date $j$ define
\[
\Delta_j:=\sup_{h_j\succeq h_t}
\left|V_{s,j}^{\pi}(h_j)-V_{s',j}^{\pi}(h_j)\right|,
\qquad \Delta_T=0,
\]
and
\[
\delta_j:=
\sup_{\substack{h_j\succeq h_t\\a\in\A_j(h_j)}}
\Dop\!\left(\K_{s,j}(h_j,a),\K_{s',j}(h_j,a)\right).
\]
At a date-$j$ history, add and subtract the lower expectation of the $s'$ continuation value under the $s$ kernel.  The lower-expectation functional is one-Lipschitz in its payoff, while the definition of $\Dop$ bounds the change of ambiguity set by $B_{j+1}\delta_j$.  Hence
\[
\Delta_j\le \beta\Delta_{j+1}+\beta B_{j+1}\delta_j.
\]
Backward iteration gives
\[
\Delta_t
\le
\sum_{j=t}^{T-1}\beta^{j-t+1}B_{j+1}\delta_j
=d_{t,h_t}^{+}(s,s'),
\]
which proves \cref{eq:operational-vintage-cover-value}.  If $R$ is an $\varepsilon$-cover, choose $r(s)\in R$ with $d_{t,h_t}^{+}(s,r(s))\le\varepsilon$ and decode the root coordinate $V_{s,t}^{\pi}(h_t)$ by $V_{r(s),t}^{\pi}(h_t)$.  The error is at most $\varepsilon$.  Applying the same argument with each descendant pair $(j,h_j)$ as the root shows that a cover under $\bar d_{t,h_t}^{+}$ recovers, uniformly over descendant roots, the coordinates of the vintages already active at $h_t$, namely $V_{s,j}^{\pi}(h_j)$ for $s\le t$.  Vintages born after date $t$ require the corresponding cover at their own root date.

For the policy statement, let $a^*$ maximize the exact current-self action values and let $\widetilde Q$ be the compressed values.  For every action $a$,
\[
\widetilde Q(a^*)-\widetilde Q(a)
\ge
Q(a^*)-Q(a)-2\varepsilon_u.
\]
Thus $a^*$ is $2\varepsilon_u$-optimal in the compressed recursion; if the exact gap is strictly larger than $2\varepsilon_u$, it remains the unique maximizer.  Applying this argument backward at every subgame preserves the pure SPE.  When vintage $0$ is retained without approximation, its precommitment recursion and its evaluation of the preserved policy use the original kernels at every date.  Both root coordinates, and therefore their difference, are exact.
\end{proof}

\begin{proof}[Proof of \cref{cor:logarithmic-vintage-memory}]
For each date $j$, let $R_j^L$ and $\phi_j^L$ satisfy the conditions in the corollary.  Define the compressed fixed-policy recursion by
\[
\widetilde V_{r,T}^{\pi}(h_T):=g(h_T),\qquad r\in R_T^L,
\]
and, for $j<T$ and $r\in R_j^L$,
\[
\widetilde V_{r,j}^{\pi}(h_j)
:=r_j(h_j,\pi_j(h_j))
+\beta\inf_{q\in\K_{r,j}(h_j,\pi_j(h_j))}
\sum_yq(y)\widetilde V_{\phi_{j+1}^L(r),j+1}^{\pi}(h_j,\pi_j(h_j),y).
\]
An omitted vintage $s$ is decoded at date $j$ by
$\widetilde V_{\phi_j^L(s),j}^{\pi}$.  The retraction property gives the intended value for stored representatives, while dynamic coherence gives
\[
\phi_{j+1}^L(\phi_j^L(s))=\phi_{j+1}^L(s),
\]
so the date-$j$ representative and the directly assigned date-$(j+1)$ representative use the same continuation coordinate.  This identity is the condition missing from independent datewise clustering.

Let
\[
E_j:=\sup_{\substack{s\le j\\h_j}}
\left|V_{s,j}^{\pi}(h_j)-
\widetilde V_{\phi_j^L(s),j}^{\pi}(h_j)\right|,
\qquad E_T=0.
\]
Adding and subtracting the exact continuation payoff under the representative kernel gives
\[
E_j\le\beta E_{j+1}+\beta\bar B C\rho^L.
\]
Indeed, the continuation-value term is bounded by $\beta E_{j+1}$, and the kernel replacement is bounded by $\beta\bar B C\rho^L$.  Backward iteration yields
\[
E_j
\le
\bar B C\rho^L\sum_{k=1}^{T-j}\beta^k
\le
\frac{\beta\bar B C}{1-\beta}\rho^L.
\]
The same comparison before substituting the policy action gives the action-value bound.  Solving the last display for $L$ gives logarithmic memory.  Because $\phi_j^L(0)=0$ for every $j$, vintage $0$ is evaluated exactly once a policy is fixed.  Because $\phi_j^L(j)=j$, the current-self action values use the correct current kernel; the gap condition therefore preserves the exact SPE, after which the root Price of Relearning is recomputed exactly.
\end{proof}

The coherence condition is substantive.  To see why, consider one action at dates $2$ and $3$, let $K_{1,2}=K_{2,2}$ lead deterministically to the unique date-$3$ state, and choose date-$3$ kernels and a bounded terminal reward with
$V_{1,3}=1/4$ and $V_{2,3}=-1/4$.  Set
$\phi_2(1)=2$, $\phi_3(1)=1$, and $\phi_3(2)=2$.  The direct local discrepancies between every original vintage and its datewise representative are zero, but the compressed coordinate stored as vintage $2$ at date $2$ continues through vintage $2$ at date $3$.  It therefore returns $-\beta/4$ instead of $\beta/4$.  Independent datewise closeness is thus insufficient; the identity
$\phi_{j+1}\circ\phi_j=\phi_{j+1}$ is what makes the recursive representative path well defined.

The archived compression experiment uses coherent retractions with fixed rewards and policy.  At horizon 64, 14 coordinates replace 65; the largest observed fixed-policy error is $4.86\times10^{-5}$ against the bound $5.91\times10^{-4}$.

\subsection{Recursive stability and multistage welfare}
\label{sec:EC-recursive-por-estimates}

\begin{proof}[Proof of \cref{prop:recursive-por-estimates}]
For part (i), fix $\pi$ and define
\[
D_u:=\sup_{h_u}|V_{s,u}^{\pi}(h_u)-V_{s',u}^{\pi}(h_u)|,
\qquad D_T=0.
\]
At a date-$u$ history, both vintages receive the same current reward under the fixed policy.  Add and subtract the lower expectation of $V_{s',u+1}^{\pi}$ under $K_{s,u}$.  The payoff Lipschitz bound and the operational set bound give
\[
D_u\le\beta D_{u+1}+\beta B_{u+1}\delta_{s,s';u}.
\]
Backward iteration yields \cref{eq:fixed-policy-vintage-bound}.

For part (ii), let
\[
G_t:=\sup_{h_t}\{V_t^{0,*}(h_t)-V_t^{0,\pi}(h_t)\},
\qquad G_T=0.
\]
Pointwise, $V_{t+1}^{0,*}\le V_{t+1}^{0,\pi}+G_{t+1}$.  Monotonicity and translation equivariance of lower expectation therefore imply
\[
V_t^{0,*}(h_t)
\le
\max_aQ_t^{0,\pi}(h_t,a)+\beta G_{t+1}.
\]
Subtracting $V_t^{0,\pi}(h_t)=Q_t^{0,\pi}(h_t,\pi_t(h_t))$ and taking suprema gives
\[
G_t\le\|e_t^{0,\pi}\|_\infty+\beta G_{t+1}.
\]
Iteration proves \cref{eq:performance-residual-bound}.

For part (iii), fix a history and let $a_0$ maximize $Q_t^{0,\pi^{\spe}}$ while $a_t=\pi_t^{\spe}(h_t)$ maximizes $Q_t^{t,\pi^{\spe}}$.  Then
\begin{align*}
e_t^{0,\pi^{\spe}}(h_t)
&=Q_t^{0,\pi^{\spe}}(h_t,a_0)-Q_t^{0,\pi^{\spe}}(h_t,a_t)\\
&=\{Q_t^{0,\pi^{\spe}}-Q_t^{t,\pi^{\spe}}\}(h_t,a_0)
 +\{Q_t^{t,\pi^{\spe}}(h_t,a_0)-Q_t^{t,\pi^{\spe}}(h_t,a_t)\}\\
&\quad+\{Q_t^{t,\pi^{\spe}}-Q_t^{0,\pi^{\spe}}\}(h_t,a_t)
\le2\varepsilon_t,
\end{align*}
where the middle term is nonpositive by current-self optimality.

For part (iv), fix $(h_t,a)$ and add and subtract the vintage-$0$ lower expectation of $V_{t,t+1}^{\pi^{\spe}}$.  Part (i) from date $t+1$ gives
\begin{align*}
&|Q_t^{0,\pi^{\spe}}(h_t,a)-Q_t^{t,\pi^{\spe}}(h_t,a)|\\
&\qquad\le
\beta\sup_{h_{t+1}}
|V_{0,t+1}^{\pi^{\spe}}(h_{t+1})-V_{t,t+1}^{\pi^{\spe}}(h_{t+1})|
+\beta B_{t+1}\delta_{0,t;t}\\
&\qquad\le
\sum_{j=t}^{T-1}\beta^{j-t+1}B_{j+1}\delta_{0,t;j}.
\end{align*}
Taking the supremum over histories and actions proves \cref{eq:epsilon-bound}.
\end{proof}

\begin{proof}[Proof of \cref{thm:multistage-por}]
Apply part (ii) of \cref{prop:recursive-por-estimates} to $\pi^{\spe}$ and then part (iii):
\[
\PoR_0
\le
\sum_{t=0}^{T-1}\beta^t\|e_t^{0,\pi^{\spe}}\|_\infty
\le
2\sum_{t=0}^{T-1}\beta^t\varepsilon_t.
\]
Substituting part (iv) and using
$\beta^t\beta^{j-t+1}=\beta^{j+1}$ proves the second inequality.

To show sharpness, take $T=2$.  Date $0$ has one zero-reward action leading deterministically to date $1$.  At date $1$, use two terminal-payoff actions from the sharp one-step construction: inherited kernel
$A=\{(1/2+\epsilon,1/2-\epsilon)\}$, fresh kernel
$B=\{(1/2,1/2)\}$, and payoffs
$z_1=\alpha(1,-1)$ and
$z_2=-\alpha(1,-1)+\delta\mathbf 1$, with
$\alpha+\delta\le1$ and $0<\delta<4\alpha\epsilon$.
The fresh self strictly selects $z_2$, whereas vintage $0$ selects $z_1$.  The root loss is
\[
\PoR_0=\beta^2(4\alpha\epsilon-\delta).
\]
Moreover $\varepsilon_1=2\beta\alpha\epsilon$ and all other $\varepsilon_t$ vanish, so
\[
\frac{\PoR_0}{\sum_t\beta^t\varepsilon_t}
=
\frac{4\alpha\epsilon-\delta}{2\alpha\epsilon}
\longrightarrow2
\]
as $\alpha\uparrow1$ and $\delta\downarrow0$.  Since
$\delta_{0,1;1}=\Dop(A,B)=2\epsilon$ and the terminal bound is one, the same sequence asymptotically attains the factor two in the ambiguity-defect inequality.
\end{proof}

\subsection{Finite-information memory and nested-vintage residual}
\label{sec:EC-memory-and-nesting}

\begin{proof}[Proof of \cref{thm:finite-information-memory}]
Write $X=c+S/8$, where $S=(S_1,\ldots,S_n)$ is uniform on $\{-1,+1\}^n$.  The map between $X$ and $S$ is bijective, so $I(S;Z)=I(X;Z)\le B$.  Let $p_i$ be the Bayes error for predicting $S_i$ from $Z$.  The binary Fano inequality gives $H(S_i\mid Z)\le h_2(p_i)$.  Entropy subadditivity, $H(S)=n$, and concavity of binary entropy yield
\[
n-B
\le H(S\mid Z)
\le \sum_{i=1}^n h_2(p_i)
\le n h_2(\bar p),
\qquad
\bar p:=\frac1n\sum_i p_i.
\]
Because $h_2$ is increasing on $[0,1/2]$,
\[
\bar p\ge h_2^{-1}\!\left(\left[1-\frac Bn\right]_+\right).
\]
For any reconstruction $\widehat X$, predict the sign of coordinate $i$ by the side of $c_i$ on which $\widehat X_i$ lies.  An incorrect sign implies $|\widehat X_i-X_i|\ge1/8$, and the induced sign rule cannot outperform the Bayes rule.  Therefore
\[
\frac1n\E\|\widehat X-X\|_1
\ge
\frac1{8n}\sum_i p_i
\ge
\frac18h_2^{-1}\!\left(\left[1-\frac Bn\right]_+\right).
\]
The same argument without $1/8$ gives the threshold-decision bound.  For every estimator, worst conditional risk over the cube is at least its average risk under the uniform distribution, so the result is also minimax on this attainable subclass.  Finally, for $b>0$, $d$ coordinates of at most $b$ bits satisfy $I(X;Z)\le H(Z)\le bd$.  Requiring average threshold error at most $\eta$ gives $bd\ge n[1-h_2(\eta)]$.
\end{proof}

\begin{proof}[Proof of \cref{thm:nested-vintage-factor-one}]
At the terminal date all vintages assign the same reward.  Suppose the value ordering holds at date $j+1$.  Monotonicity of lower expectation and the inclusion $K_{s,j}(h,a)\subseteq K_{u,j}(h,a)$ imply
\[
\inf_{q\in K_{s,j}(h,a)}E_q[V_{s,j+1}^{\pi}]
\ge
\inf_{q\in K_{s,j}(h,a)}E_q[V_{u,j+1}^{\pi}]
\ge
\inf_{q\in K_{u,j}(h,a)}E_q[V_{u,j+1}^{\pi}].
\]
Adding the common current reward proves $Q_{s,j}^{\pi}(h,a)\ge Q_{u,j}^{\pi}(h,a)$, and evaluating the fixed policy action gives the value ordering.  Backward induction proves the first claim.

Let $a_u$ maximize $Q_{u,u}^{\pi}(h,\cdot)$.  Add and subtract current-self values:
\begin{align*}
\max_aQ_{s,u}^{\pi}(h,a)-Q_{s,u}^{\pi}(h,a_u)
&\le
\max_a\{Q_{s,u}^{\pi}(h,a)-Q_{u,u}^{\pi}(h,a)\}\\
&\quad+\max_a\{Q_{u,u}^{\pi}(h,a)-Q_{u,u}^{\pi}(h,a_u)\}\\
&\quad+Q_{u,u}^{\pi}(h,a_u)-Q_{s,u}^{\pi}(h,a_u).
\end{align*}
The middle term is zero by current-self optimality and the last is nonpositive by the vintage ordering.  Only the first term remains.
\end{proof}

\subsection{Policy robustness under gaps and smooth actions}
\label{sec:EC-policy-robustness-regimes}

\begin{proof}[Proof of \cref{prop:policy-robustness-regimes}]
For part (i), let $a^*$ be the unique vintage-$0$ maximizer at a subgame.  For any $a\ne a^*$,
\[
Q_t^{t,\pi^{\spe}}(a^*)-Q_t^{t,\pi^{\spe}}(a)
\ge
Q_t^{0,\pi^{\spe}}(a^*)-Q_t^{0,\pi^{\spe}}(a)-2\varepsilon_t
\ge \kappa_t(h_t)-2\varepsilon_t>0.
\]
Thus every current self chooses the vintage-$0$ maximizer.  The vintage-$0$ Bellman residual is zero at every subgame, and the performance-residual recursion gives $\PoR_0=0$.

For part (ii), interior optimality and differentiability give
\[
\nabla Q^{\inh}(a_I)=0,
\qquad
\nabla Q^{\fresh}(a_F)=0.
\]
Hence
\[
\|\nabla Q^{\inh}(a_F)\|_2
=
\|\nabla Q^{\inh}(a_F)-\nabla Q^{\fresh}(a_F)\|_2
\le\varepsilon.
\]
For a differentiable $\mu$-strongly concave function, $-\nabla Q^{\inh}$ is $\mu$-strongly monotone.  Cauchy--Schwarz therefore gives
\[
\mu\|a_F-a_I\|_2
\le
\|\nabla Q^{\inh}(a_F)-\nabla Q^{\inh}(a_I)\|_2
=
\|\nabla Q^{\inh}(a_F)\|_2,
\]
which proves the displacement bound.  Finally, $L$-smoothness and $\nabla Q^{\inh}(a_I)=0$ imply
\[
Q^{\inh}(a_F)
\ge
Q^{\inh}(a_I)-\frac L2\|a_F-a_I\|_2^2.
\]
Optimality of $a_I$ supplies the nonnegative lower bound, and substitution yields the second inequality.
\end{proof}

\subsection{Exact fixed-model evaluation and universal certification}
\label{sec:EC-certified-computation}

A specified finite model with explicitly listed one-step ambiguity laws is evaluated by vintage backward induction.  Universal certification additionally optimizes rewards and equilibrium; it is NP-complete already on the two-date singleton-kernel subclass of \cref{thm:por-cert-complexity}.

\paragraph{Vintage backward evaluation.}
Let $\X_t$ be the finite state set at date $t$, let $f_t(x,a,y)$ be the successor state, and suppose each ambiguity set is listed by at most $K$ probability laws.  For a fixed reward system, initialize
\[
 V_{s,T}(x)=g(x),\qquad s<T,
\]
and, for $t=T-1,\ldots,0$, compute
\begin{align*}
 Q_{s,t}(x,a)
 &=r_t(x,a)+\beta\min_{q\in K_{s,t}(x,a)}
   \sum_yq(y)V_{s,t+1}(f_t(x,a,y)),\qquad s\le t,\\
 \pi_t(x)&=\tau_{t,x}\!\left(\argmax_a Q_{t,t}(x,a)\right),\\
 V_{s,t}(x)&=Q_{s,t}(x,\pi_t(x)),\qquad s\le t.
\end{align*}
A separate vintage-zero Bellman recursion gives $P_t(x)$.  The root difference is the selected-SPE Price of Relearning.

\begin{proof}[Proof of \cref{thm:vintage-backward-evaluation}]
At $T$ all vintages assign $g$.  If $V_{s,t+1}$ is correct for every active $s$, the displayed $Q_{s,t}$ is exactly vintage $s$'s lower expectation under the selected continuation policy.  The date-$t$ self maximizes the diagonal value, satisfying the one-period-deviation condition in \cref{thm:triangular-v2}, while the same action is recorded under every older vintage.  Backward induction therefore yields the tie-broken pure Markov SPE and all vintage values.  The vintage-zero Bellman recursion gives the rectangular precommitment optimum; subtraction gives the selected-SPE Price of Relearning.

Each $(t,x,a,s)$ evaluation costs at most $KY$ multiply--adds, so summing over actions, states, and active vintages gives
\[
 O\!\left(AKY\sum_{t=0}^{T-1}(t+1)|\X_t|\right).
\]
One value per active vintage--state pair, plus precommitment values and policies, gives the storage bound; compressing the input kernels requires additional structure.
\end{proof}

The two-date validation agrees with $\PoR_0=0.94^2=0.8836$ and passes exhaustive policy and one-period-deviation checks.  Table~\ref{tab:EC-vintage-backward-scaling} reports median runtimes for the eight-state benchmark.

\begin{proposition}[Selection-aware certificate and shared-state Markov formulation]
\label{prop:EC-selection-aware-certificate}
At a decision node with finite action set $\mathcal A$, let $P_a$, $I_a$, and $F_a$ denote precommitment, initial-vintage policy, and current-self action values.  Let $\varnothing\ne\mathcal J\subseteq\mathcal A$, suppose the realized selectors satisfy
\[
a_P\in\arg\max_{a\in\mathcal A}P_a,
\qquad
a_F\in\mathcal J\cap\arg\max_{a\in\mathcal A}F_a,
\]
and assume
\[
0\le P_a-I_a\le c_a,
\qquad
0\le I_a-F_a\le e_a,
\qquad
|P_a|,|I_a|,|F_a|\le B.
\]
Define $\Omega_B(c,e;\mathcal J)$ as the maximum of $P_{a_P}-I_{a_F}$ over all triples and selectors satisfying these inequalities and maximizing-action conditions.  Then
\begin{equation}\label{eq:EC-selection-aware-envelope}
P_{a_P}-I_{a_F}
\le \Omega_B(c,e;\mathcal J)
\le \max_{a\in\mathcal A}(c_a+e_a).
\end{equation}
For finite actions, $\Omega_B$ is obtained by enumerating $(a_P,a_F)$ and solving linear programs.  Recursively substituting robust continuation certificates for $c_a$ and coupled vintage-action certificates for $e_a$ gives a valid state-dependent Price-of-Relearning upper bound.  The linear-program envelope is no larger than the closed-form factor-one bound.

Now let rewards and kernels be state based on a finite recombining graph, and let $\Pi^{\rm M}$ be the pure Markov policies.  For a state-coupled reward family $\mathcal R_\Theta$, define
\[
\mathfrak d_0^{\rm cert,M}(\mathcal R_\Theta)
:=
\sup_{r\in\mathcal R_\Theta}
\sup_{\pi\in\mathrm{SPE}(r)\cap\Pi^{\rm M}}
\{P_0^r-V_{0,0}^{r,\pi}\}.
\]
There is an exact global shared-state polyhedral/disjunctive formulation for this Markov-restricted certificate.  If $p=\dim(\Theta)$, date $t$ has $S_t$ states, and there are at most $A$ actions, $Y$ outcomes, and $K$ listed laws per ambiguity set, its explicit description has size polynomial in
\begin{equation}\label{eq:EC-shared-state-formulation-size}
p+AKY\sum_{t=0}^{T-1}(t+1)S_t.
\end{equation}
Always $\mathfrak d_0^{\rm cert,M}\le\mathfrak d_0^{\rm cert}$.  A fixed statewise tie-breaking rule yields the corresponding selected-equilibrium certificate, not the selection-free unrestricted supremum.  Numerical equality with the unrestricted certificate holds whenever its supremum is attained or approached by state-consistent pure SPEs.  The stronger condition that every relevant pure SPE is state consistent is sufficient for equality of the equilibrium correspondences; uniqueness of the current-self action at every state is one such condition.  Neither stronger condition is necessary for numerical equality.
\end{proposition}

\begin{proof}
The realized triple is feasible for $\Omega_B$, proving the first inequality.  For any feasible triple,
\begin{align*}
P_{a_P}-I_{a_F}
&=(P-I)_{a_P}+(I-F)_{a_P}
 +(F_{a_P}-F_{a_F})+(F-I)_{a_F}\\
&\le c_{a_P}+e_{a_P}
\le\max_{a\in\mathcal A}(c_a+e_a).
\end{align*}
The last two terms are nonpositive by current-self optimality and the one-sided vintage order.  Enumerating $(a_P,a_F)$ leaves a linear feasible set in $(P,I,F)$.  Valid descendant certificates remain valid under robust lower expectation, so backward induction preserves the envelope.

For the recombining formulation, use one reward parameter vector $\theta$, state-action rewards from its affine map, precommitment values/selectors by date--state, vintage values and action values, and finite disjunctions for minimizing listed laws.  One equilibrium selector per date--state imposes current-self optimality for every incoming history, and all predecessor edges share the same successor-state variables.

A reward system and pure Markov SPE supply a feasible point with the correct objective.  Conversely, any feasible point defines a state-based reward system and Markov policy; shared Bellman equalities evaluate all vintages, diagonal maximizing constraints enforce one-period deviations, and precommitment selectors recover the initial-vintage optimum.  This proves exactness and the size bound.  Because Markov SPEs are a subset of all pure SPEs, the certificate inequality follows.  The stated attainment/approximation condition gives equality of the suprema; full state consistency gives equality of the equilibrium correspondences, while tied history-dependent equilibria can still leave the two numerical suprema equal.
\end{proof}

State-based primitives do not force every pure SPE to be Markov: after recombination, tied current-self actions may be selected differently across histories and valued differently by an older vintage.  One statewise selector therefore represents the Markov, not the unrestricted, equilibrium correspondence.

\paragraph{Universal worst-reward certification.}
The universal polyhedral formulation maximizes over bounded rewards and pure SPEs using reward, vintage-value, ambiguity-vertex, and action-selector variables.  It is exact but not asserted polynomial in the unrestricted history tree \citep{EC-Balas1979}.

A partial policy cell $\chi$ fixes only the SPE action at histories in its domain.  It does not fix the initial-vintage precommitment action.  Let $\varepsilon_{t,a}(h)$ be an action-specific initial-versus-current vintage bound and define
\[
\ell_t^\chi(h)=
\begin{cases}
\max_a\varepsilon_{t,a}(h)+\varepsilon_{t,\chi(h)}(h),&h\in\operatorname{dom}\chi,\\
2\max_a\varepsilon_{t,a}(h),&h\notin\operatorname{dom}\chi.
\end{cases}
\]
The valid recursion is
\begin{equation}\label{eq:EC-correct-policy-cell-certificate}
b_t^\chi(h)=\ell_t^\chi(h)+\beta
\max_{a\in\A_t(h)}\sup_{q\in K_{0,t}(h,a)}
E_q[b_{t+1}^\chi(H_{t+1})],
\qquad b_T^\chi=0.
\end{equation}
The continuation maximum remains over all actions whether or not the SPE action is fixed.  If a global pilot bound is available, the reported cell and global bounds are
\[
U_\chi=\min\{U_\chi^{\rm MIP},B_\chi,U^{\rm pilot}\},
\qquad
U=\min\{U^{\rm pilot},\max_\chi U_\chi\}.
\]
A regression witness has old fixed-branch value $0.285911$, exact cell optimum $0.297387$, and corrected certificate $0.400933$; the audit fails if the invalid recursion returns.

Four independent two-date calculations agree within $1.57\times10^{-13}$.  The corrected three-date interval is $[0.5708068461,1.7245881687]$; on the four-date scanner tree the incumbent is $0.2750440124$, with factor-two, factor-one, and selection-aware bounds $1.2124365273$, $0.6062182637$, and $0.5993942250$.  The $3.66\%$ normalization is additive, not a relative MIP gap.  An independent bracket checker verifies primal witnesses by direct robust backward evaluation and SPE inequalities and recomputes the analytic bounds by small continuous LPs, reproducing the three-date bracket exactly and the scanner upper bound within $1.3\times10^{-14}$.

\paragraph{Approximate vintage compression.}
The age-forgetting experiment fixes rewards and a Markov policy, retains the initial vintage and coherent recent-age representatives, and decodes omitted coordinates.  Table~\ref{tab:EC-vintage-compression} compares actual and certified errors; this structured upper bound is distinct from the unrestricted memory lower bounds.

\paragraph{Reproducibility.}
The archive records inputs, bounds, checksums, and independent regression tests.

\subsection{Complexity of universal Price-of-Relearning certification}
\label{sec:EC-por-cert-complexity}

The hardness statement concerns the outer certification problem, not evaluation of a fixed reward model.  We use the following state-graph subclass.  Let $G=(V,E)$ be an undirected graph with $m=|E|\ge1$.  There is one date-zero state with one action.  Its unique kernel is uniform over date-one edge states $x_e$, $e\in E$.  Each edge state $x_e$, for $e=\{u,v\}$, has two actions $a_e^u,a_e^v$ and two successor terminal states $z_u,z_v$.  All stage rewards vanish.  The terminal reward variables are shared by state and satisfy
\[
 g(z_v)=x_v\in[0,1],\qquad v\in V.
\]
At $x_e$ the singleton vintage kernels are
\begin{align*}
 K_{0,1}(x_e,a_e^u)&=\{\delta_{z_u}\},&
 K_{1,1}(x_e,a_e^u)&=\{\delta_{z_v}\},\\
 K_{0,1}(x_e,a_e^v)&=\{\delta_{z_v}\},&
 K_{1,1}(x_e,a_e^v)&=\{\delta_{z_u}\}.
\end{align*}
Thus the date-one evaluator reverses the two terminal coordinates relative to the initial evaluator.

\begin{proof}[Proof of \cref{thm:por-cert-complexity}]
Fix a terminal vector $x\in[0,1]^V$.  At edge state $e=\{u,v\}$, the initial-vintage precommitment problem chooses the larger of $x_u,x_v$ and has value
\[
 P_1(x_e)=\max\{x_u,x_v\}.
\]
The date-one self values $a_e^u$ at $x_v$ and $a_e^v$ at $x_u$.  It therefore selects an action whose initial-vintage value is $\min\{x_u,x_v\}$.  Ties do not change that value.  Since the root has one action and $\beta=1$,
\begin{equation}\label{eq:EC-maxcut-por-identity}
 \PoR_0(x)
 =\frac1m\sum_{\{u,v\}\in E}
 \left(\max\{x_u,x_v\}-\min\{x_u,x_v\}\right)
 =\frac1m\sum_{\{u,v\}\in E}|x_u-x_v|.
\end{equation}

For fixed coordinates other than $x_v$, the right side of \eqref{eq:EC-maxcut-por-identity} is a convex piecewise-linear function of $x_v\in[0,1]$.  Its maximum on the interval is attained at an endpoint.  Rounding one coordinate at a time therefore produces a binary vector $s\in\{0,1\}^V$ without decreasing the objective.  For a binary vector,
\[
 \sum_{\{u,v\}\in E}|s_u-s_v|
\]
is exactly the number of edges crossing the cut $\{v:s_v=1\}$.  Hence
\begin{equation}\label{eq:EC-por-maxcut}
 \mathfrak d_0^{\rm cert}([0,1]^V;G)=\frac{\operatorname{MaxCut}(G)}{|E|}.
\end{equation}

Given an unweighted \textsc{Max-Cut} instance $(G,k)$, construct the preceding Price-of-Relearning instance and set $\gamma=k/|E|$.  The construction is polynomial, uses two decision dates, zero stage rewards, one root action, two actions per edge state, singleton kernels, and the shared terminal reward box $[0,1]^V$.  Equation \eqref{eq:EC-por-maxcut} gives a yes-instance if and only if $G$ has a cut of size at least $k$.  Simple \textsc{Max-Cut} is NP-complete \citep{EC-GareyJohnsonStockmeyer1976}, so the general certification problem is NP-hard.  The restricted problem is in NP because a cut is a polynomial certificate and its cardinality verifies the threshold.  This proves NP-completeness of the graph-generated restricted subclass.  The construction has fixed horizon and only vintages zero and one, proving that neither parameter alone yields tractability.

For the tractability statement, the reward-interaction graph of this subclass is $G$: two terminal reward variables interact precisely when they appear at a common edge state.  Given a nice tree decomposition of width $w$, maintain for every bag $B$ and assignment $\sigma:B\to\{0,1\}$ the largest number of crossing edges introduced in the processed subtree among assignments extending $\sigma$.  Introduce-vertex nodes copy the child table, introduce-edge nodes add $\mathbf 1\{\sigma(u)\ne\sigma(v)\}$, forget nodes maximize over the forgotten bit, and join nodes add child values after each edge is assigned to exactly one introduce-edge node.  There are at most $2^{w+1}$ assignments per bag.  For a supplied nice decomposition with $N_{\rm bag}$ bags, the running time is
\[
 O\!\left(2^{w+1}N_{\rm bag}\right).
\]
When the supplied or constructed nice decomposition has $N_{\rm bag}=O(|V|+|E|)$, the linear-size specialization in the main paper follows.  More general nice-decomposition conversions may introduce a factor polynomial in $w$; standard treewidth algorithms still make the problem fixed-parameter tractable in $w$ \citep{EC-CyganEtAl2015}.  Dividing the resulting maximum cut by $|E|$ gives $\mathfrak d_0^{\rm cert}([0,1]^V;G)$ for this graph-generated subclass.  The statement does not extend the treewidth guarantee to arbitrary pairwise state-coupled certification instances.
\end{proof}

The archived verification script enumerates the cuts of 35 graphs with four through ten vertices and independently evaluates the constructed relearning objective.  The maximum absolute discrepancy between $\operatorname{MaxCut}(G)/|E|$ and the Price-of-Relearning value is $1.12\times10^{-16}$.

The reduction uses singleton rectangular kernels.  Its hardness is therefore not inherited from nonrectangular robust policy evaluation or from evaluating a fixed Markov decision process, which is polynomial for explicitly represented finite-horizon models \citep{EC-PapadimitriouTsitsiklis1987}.  It is created by the outer optimization over shared reward coordinates together with the cross-vintage reversal of continuation choices.  The treewidth result identifies the corresponding structural parameter: sparse interaction among shared reward coordinates, rather than horizon or the number of active vintages alone.

\section{EC.4 Endogenous Compatible Protocol Design}

\begin{proof}[Proof of \cref{prop:conditional-compatible-repair-v7}]
Let $W=\sum_{x\in\mathscr X}w_x$ and
$\bar h(u)=W^{-1}\sum_xw_xh_{C_x}(u)$.  Minkowski additivity and positive homogeneity of support functions \citep{EC-RockafellarWets1998} give
\[
\bar h=h_{C^{\rm av}}.
\]
For every direction $u$ and every scalar $y$, the weighted variance identity is
\[
\sum_xw_x(y-h_{C_x}(u))^2
=
\sum_xw_x(h_{C_x}(u)-\bar h(u))^2
+W(y-\bar h(u))^2.
\]
Set $y=h_C(u)$ and integrate with respect to $\nu$.  The calibration set is finite, $\nu$ is finite, and all support functions are continuous and bounded on the unit sphere, so every term is measurable and integrable.  The resulting identity is exactly \cref{eq:conditional-compatible-variance-identity}, and it proves that $C^{\rm av}$ is a minimizer.

If $C$ is any minimizer, the nonnegative second term in that identity vanishes, hence $h_C=h_{C^{\rm av}}$ $\nu$-almost everywhere.  If the two continuous support functions differed at a point of $\supp(\nu)$, continuity would produce a relatively open neighborhood of positive $\nu$-measure on which they differed, a contradiction.  Thus they agree on $\supp(\nu)$.  When that support determines compact convex sets, equality of support functions there implies $C=C^{\rm av}$.
\end{proof}

\begin{proof}[Proof of \cref{thm:endogenous-design-existence}]
Let
\[
\mathfrak G:=\{(c,\pi):c\in\mathfrak C_{\rm stat},\ \pi\in\mathrm{SPE}(c)\}.
\]
The graph is nonempty by assumption.  To prove closedness, take $(c_n,\pi_n)\to(c,\pi)$ with $(c_n,\pi_n)\in\mathfrak G$.  Since the pure-policy set is finite, a subsequence has $\pi_n=\pi$ for all $n$.  At every subgame and for every unilateral deviation, the equilibrium inequality under $c_n$ involves finitely many robust action values.  Their continuity in $c$ permits passage to the limit, so the same inequalities hold at $c$ and $\pi\in\mathrm{SPE}(c)$.  Thus $\mathfrak G$ is closed.  It is compact because $\mathfrak C_{\rm stat}$ is compact and the policy set is finite.

For a fixed policy, its nominal occupancy probabilities are fixed, and each fidelity loss is continuous in $c$.  Hence the finite discounted sum defining the design objective is continuous on every policy slice and therefore on $\mathfrak G$.  Weierstrass' theorem gives a minimizer.
\end{proof}

\begin{lemma}[Soft-policy occupancy Lipschitz bound]\label{lem:EC-soft-occupancy-lipschitz}
Suppose the entropy-regularized action values satisfy
$\sup_{t,h,a}|Q_t^r(h,a)-Q_t^{r'}(h,a)|\le L_Q|r-r'|$ and the two policies use a common nominal transition kernel.  Let $d_0^r,d_0^{r'}$ be their initial state distributions.  Then
\begin{align}
\|d_t^r-d_t^{r'}\|_1
&\le \|d_0^r-d_0^{r'}\|_1
+t\frac{2L_Q}{\tau}|r-r'|,\label{eq:EC-soft-state-lipschitz}\\
\|\omega_t^r-\omega_t^{r'}\|_1
&\le \|d_0^r-d_0^{r'}\|_1
+(t+1)\frac{2L_Q}{\tau}|r-r'|.\label{eq:EC-soft-occupancy-datewise}
\end{align}
Consequently,
\begin{align}
\sum_{t=0}^{T-1}\beta^t\|\omega_t^r-\omega_t^{r'}\|_1
&\le
\left(\sum_{t=0}^{T-1}\beta^t\right)\|d_0^r-d_0^{r'}\|_1\notag\\
&\quad+
\frac{2L_Q}{\tau}\left(\sum_{t=0}^{T-1}\beta^t(t+1)\right)|r-r'|.\label{eq:EC-soft-occupancy-general}
\end{align}
In particular, the first term vanishes under a common initial distribution.
\end{lemma}
\begin{proof}
At a subgame, entropy regularization gives the softmax policy
\[
\pi_\tau^r(a\mid h)=
\frac{\exp(Q_t^r(h,a)/\tau)}{\sum_b\exp(Q_t^r(h,b)/\tau)}.
\]
The softmax Jacobian is $\tau^{-1}(\operatorname{diag}(p)-pp^\top)$ and its operator norm from $\ell_\infty$ to $\ell_1$ is at most $2/\tau$.  Hence, with $\eta=(2L_Q/\tau)|r-r'|$,
\[
\sup_h\|\pi_t^r(\cdot\mid h)-\pi_t^{r'}(\cdot\mid h)\|_1\le\eta.
\]
Under the common transition kernel, adding and subtracting the state--action law formed from $d_t^r$ and $\pi_t^{r'}$ gives
\[
\|d_{t+1}^r-d_{t+1}^{r'}\|_1
\le
\|d_t^r-d_t^{r'}\|_1+\eta.
\]
Induction proves \eqref{eq:EC-soft-state-lipschitz}.  Since
$\omega_t^r=d_t^r\pi_t^r$, another add-and-subtract step yields
\[
\|\omega_t^r-\omega_t^{r'}\|_1
\le\|d_t^r-d_t^{r'}\|_1+\eta,
\]
which proves \eqref{eq:EC-soft-occupancy-datewise}.  Discounting and summing gives \eqref{eq:EC-soft-occupancy-general}.
\end{proof}

\begin{proof}[Proof of \cref{thm:endogenous-design-fixed-point}]
Write
\[
N(r)=\sum_i\omega_i(r)s_i\widehat r_i,
\qquad
D(r)=\sum_i\omega_i(r)s_i,
\qquad
m(r)=\frac{N(r)}{D(r)}.
\]
The weighted occupancy Lipschitz assumption makes $N$ and $D$ continuous and $D(r)\ge\underline d>0$.  Because interval projection is continuous and maps into $[\underline r,\bar r]$, $\mathcal T_\tau$ is a continuous self-map of that compact interval.  In one dimension, define $g(r)=\mathcal T_\tau(r)-r$.  Then $g(\underline r)\ge0$ and $g(\bar r)\le0$, so the intermediate value theorem gives a fixed point.

For the Lipschitz estimate, let $\Delta_{\widehat r}=\max_i\widehat r_i-\min_i\widehat r_i$.  Because $m(r')$ is a nonnegative weighted average of the targets, $m(r')\in[\min_i\widehat r_i,\max_i\widehat r_i]$.  Moreover,
\begin{align*}
N(r)-m(r')D(r)
&=\sum_i s_i\omega_i(r)[\widehat r_i-m(r')]\\
&=\sum_i s_i[\omega_i(r)-\omega_i(r')][\widehat r_i-m(r')],
\end{align*}
where the second equality uses $N(r')-m(r')D(r')=0$.  Therefore
\begin{align*}
|m(r)-m(r')|
&=\frac{|N(r)-m(r')D(r)|}{D(r)}\\
&\le\frac{\Delta_{\widehat r}}{D(r)}\sum_i s_i|\omega_i(r)-\omega_i(r')|\\
&\le\frac{\Delta_{\widehat r}L_a}{\underline d}|r-r'|.
\end{align*}
Projection is nonexpansive, so the same modulus applies to $\mathcal T_\tau$.  If $q_{\rm tar}=\Delta_{\widehat r}L_a/\underline d<1$, Banach's theorem gives a unique fixed point and, for every initialization,
\[
|r^k-r_\tau^*|\le q_{\rm tar}^k|r^0-r_\tau^*|.
\]
\end{proof}

\begin{proof}[Proof of \cref{thm:compatible-design-performance}]
Let $\pi^c$ be subgame-wise Bellman optimal for the compatible family and abbreviate
\[
\varepsilon_t=\varepsilon_t^{\rm tar}(c).
\]
Because $\pi^c$ is optimal for the compatible Bellman system, at every subgame
\begin{align*}
&\max_a Q_t^{{\rm tar},\pi^c}(h_t,a)
-Q_t^{{\rm tar},\pi^c}(h_t,\pi_t^c(h_t))\\
&\quad\le
\max_a\left[Q_t^{c,\pi^c}(h_t,a)+\varepsilon_t\right]
-
\left[Q_t^{c,\pi^c}(h_t,\pi_t^c(h_t))-\varepsilon_t\right]
=2\varepsilon_t.
\end{align*}
Let
\[
G_t
:=
\sup_{h_t}
\left(V_t^{{\rm tar},*}(h_t)-V_t^{{\rm tar},\pi^c}(h_t)\right),
\qquad G_T=0.
\]
The usual robust Bellman residual argument for the single fixed target family gives
\[
G_t\le 2\varepsilon_t+\beta G_{t+1}.
\]
Backward iteration yields
\[
V_0^{{\rm tar},*}-V_0^{{\rm tar},\pi^c}
\le
2\sum_{t=0}^{T-1}\beta^t\varepsilon_t.
\]
This proves \cref{eq:compatible-design-performance}.

It remains to control $\varepsilon_t$.  Evaluate the same fixed policy $\pi^c$ under the target and compatible kernel families and define
\[
D_u
:=
\sup_{h_u}
\left|V_u^{{\rm tar},\pi^c}(h_u)-V_u^{c,\pi^c}(h_u)\right|,
\qquad D_T=0.
\]
At a date-$u$ history, insert and subtract the target lower expectation of the compatible continuation value.  The elementary value- and set-Lipschitz inequalities for lower expectations give
\[
D_u
\le
\beta D_{u+1}
+
\beta B_{u+1}\delta_u^{\rm tar,c}.
\]
The identical argument after a one-period deviation $a$ gives
\[
\left|
Q_t^{{\rm tar},\pi^c}(h_t,a)
-
Q_t^{c,\pi^c}(h_t,a)
\right|
\le
\beta D_{t+1}
+
\beta B_{t+1}\delta_t^{\rm tar,c}.
\]
Backward iteration of the first recursion and substitution into the second imply
\[
\varepsilon_t^{\rm tar}(c)
\le
\sum_{j=t}^{T-1}
\beta^{j-t+1}B_{j+1}\delta_j^{\rm tar,c},
\]
which is \cref{eq:compatible-target-epsilon-bound}.  Multiplying by $2\beta^t$, summing over $t$, and using
$\beta^t\beta^{j-t+1}=\beta^{j+1}$ gives \cref{eq:compatible-design-operational-certificate}.

For the control-reachable refinement, suppose that for every $(u,h_u,a)$ the normalized class $\V_{u+1}(h_u,a)$ contains the continuation function
\[
B_{u+1}^{-1}V_{u+1}^{c,\pi^c}(h_u,a,\cdot)
\]
whenever $B_{u+1}>0$; zero bounds are trivial.  Define the local target--candidate control defect by restricting the normalized payoff supremum in $\Dop$ to $\V_{u+1}(h_u,a)$.  The set-Lipschitz step above then uses this smaller defect, because the only payoff direction at which the two kernel families are compared is the compatible continuation value just displayed.  The same backward induction therefore proves the two bounds with control-reachable defects in place of $\delta_u^{\rm tar,c}$.  When the target and candidate predictive kernels arise from parameter ambiguity sets through a controlled observation matrix, \cref{thm:control-observability} provides the stated link to predictive and parameter discrepancies.
\end{proof}
\paragraph{Why root optimality alone is insufficient.}
Let $T=2$, $\beta=1$, and stage rewards be zero.  The sole root action reaches $A$ under the compatible kernel and $B$ under the target kernel.  State $A$ has one zero-payoff action; at $B$, where the two date-$1$ evaluators coincide, actions $b_{\rm bad}$ and $b_{\rm good}$ pay zero and one.  A policy choosing $b_{\rm bad}$ at $B$ is compatible-root optimal because $B$ is unreachable, yet its target loss is one.  Both one-step target--compatible discrepancies along that policy are zero.  Thus initial-state optimality cannot replace compatible Bellman optimality at every subgame.

\section{EC.5 Pricing Derivations and Scanner-Data Calibration}

\subsection{Gaussian inheritance and pricing wedge}
\begin{proof}[Proof of \cref{thm:linear-gaussian-inheritance-pricing}]
The Gaussian natural mean is $m/v$.  Replacing $m_s$ by $m_s+v_sq$ adds $q$ to the natural mean.  The realized likelihood contributes the same sufficient-statistic increment and the same precision increment to every candidate, so
\[
\frac{m_t^q}{v_t}=\frac{m_t}{v_t}+q,
\qquad
m_t^q=m_t+v_tq.
\]
The date-$s$ fixed-level credible family has $|q|\le z/\sqrt{v_s}$.  Multiplication by $v_t$ gives inherited raw radius $zv_t/\sqrt{v_s}$, while rebuilding a fixed-level interval at date $t$ gives $z\sqrt{v_t}$.  If $v_t<v_s$, their ratio is $\sqrt{v_s/v_t}>1$.

Under a candidate mean $m_t+v_tq$, expected one-period revenue is
\[
p\{a-p(m_t+v_tq)\}.
\]
Because prices are nonnegative, the worst candidate is $q=k$.  With $d=m_t+kv_t$ the objective is $ap-dp^2$.  If $d\le0$, its derivative $a-2dp$ is strictly positive on $[\underline p,\bar p]$, so the unique maximizer is $\bar p$.  If $d>0$, the objective is strictly concave and its constrained maximizer is the projection of $a/(2d)$.  This proves \eqref{eq:robust-linear-price}.

The resulting optimizer $p^\star(d)$ is nonincreasing in $d$: it equals $\bar p$ for $d\le0$ and for positive $d$ whose unconstrained optimum exceeds $\bar p$, decreases as $a/(2d)$ while the optimum is interior, and equals $\underline p$ once the unconstrained optimum falls below the lower boundary.  Since $v_t<v_s$ implies $k_F>k_I$ and hence $d_F>d_I$, we obtain $p_t^F\le p_{s,t}^I$.  If $d_I>0$, both objectives are strictly concave and equality is equivalent to the two projected optima coinciding; both interior optima give strict inequality.  If $d_I\le0$, the inherited choice is $\bar p$, and the comparison is strict exactly when the fresh choice lies below that boundary.
\end{proof}

\begin{corollary}[Moderate information maximizes the raw reconstruction wedge]\label{cor:EC-moderate-information-wedge}
Let $I_{s,t}=\sigma^{-2}\sum_{u=s}^{t-1}p_u^2$ and $x=v_t/v_s$.  Then
\[
\rho_t^F-\rho_{s\to t}^I=z\sqrt{v_s}(\sqrt{x}-x),
\qquad
x=\frac{1}{1+v_sI_{s,t}}.
\]
The wedge vanishes at zero and infinite information and has the unique maximum $z\sqrt{v_s}/4$ at $v_sI_{s,t}=3$.
\end{corollary}
\begin{proof}
Conjugate updating gives $v_t^{-1}=v_s^{-1}+I_{s,t}$ and hence the displayed formula.  The derivative of $\sqrt{x}-x$ is $(2\sqrt{x})^{-1}-1$, which vanishes only at $x=1/4$; its second derivative is negative.  Since $x=1$ at zero information and tends to zero under unbounded information, the endpoint limits are zero and the stated maximizer is unique.
\end{proof}

\begin{proof}[Proof of \cref{prop:pricing-feedback}]
Write
\[
p^F(v)=\Pi_{[\underline p,\bar p]}
\left(\frac{a}{2(m+z\sqrt v)}\right),
\qquad
p^I(v)=\Pi_{[\underline p,\bar p]}
\left(\frac{a}{2(m+zv/\sqrt{v_s})}\right),
\]
where $a>0$, $z>0$, $\sigma^2>0$, and the common posterior mean $m>0$ may vary by date but is the same across the two conditional paths; all displayed denominators are assumed positive.  For every $v\in(0,v_s]$,
\[
z\sqrt v\ge z\frac{v}{\sqrt{v_s}},
\]
with strict inequality for $v<v_s$.  The unconstrained price is decreasing in the ambiguity penalty and interval projection is monotone, so
\begin{equation}\label{eq:EC-price-order-same-variance}
p^F(v)\le p^I(v).
\end{equation}
Both price rules are nonincreasing in $v$: their denominators are increasing in $v$, and projection preserves the order.

Let
\[
\Psi(v,p):=\frac{v}{1+vp^2/\sigma^2}.
\]
Direct differentiation gives
\[
\frac{\partial\Psi}{\partial v}
=\frac{1}{(1+vp^2/\sigma^2)^2}>0,
\qquad
\frac{\partial\Psi}{\partial p}
=-\frac{2v^2p/\sigma^2}{(1+vp^2/\sigma^2)^2}\le0
\]
on the maintained feasible interval $0\le\underline p\le p\le\bar p$.  Start from $v_s^F=v_s^I=v_s$.  Suppose inductively that $v_t^F\ge v_t^I$.  Because $p^F(\cdot)$ is nonincreasing,
\[
p_t^F(v_t^F)\le p_t^F(v_t^I)\le p_t^I(v_t^I),
\]
where the second inequality is \eqref{eq:EC-price-order-same-variance}.  Monotonicity of $\Psi$ then yields
\[
\begin{aligned}
v_{t+1}^F
&=\Psi(v_t^F,p_t^F(v_t^F))\\
&\ge \Psi(v_t^I,p_t^F(v_t^F))\\
&\ge \Psi(v_t^I,p_t^I(v_t^I))
=v_{t+1}^I.
\end{aligned}
\]
This proves both weak orderings by induction.

At any date at which the two conditional paths have a common variance $v<v_s$, the unconstrained fresh price is strictly below the inherited price.  If projection does not collapse both values to the same boundary, the price inequality is strict.  On the nonnegative line, $0\le p_1<p_2$ implies $p_1^2<p_2^2$ and hence $\Psi(v,p_1)>\Psi(v,p_2)$, even when $p_1=0$.  Thus every strict feasible price gap makes the following inherited variance strictly smaller.  Once $v_t^F>v_t^I$, the same monotonicity argument preserves the weak ordering, with strictness whenever neither price projection nor a zero-information action collapses it.  These are conditional pathwise comparisons at a common posterior-mean sequence; they do not compare the full endogenous distributions of posterior means.
\end{proof}

\subsection{Dominick's scanner-data replay and policy experiment}\label{sec:EC-oj-replay}
The Dominick's three-brand scanner panel contains 28,947 store--brand--week observations from 83 stores over 121 weeks \citep{EC-Montgomery1997}.  The reproducibility archive records the pinned immutable analysis source, SHA-256, schema, and verified downloader; it does not redistribute the raw CSV.  All estimates below are descriptive or model based.

We estimate the pooled interacted fixed-effects model
\[
\log Q_{ibt}=\alpha_{ib}+\gamma_t+\beta_b\log p_{ibt}+\delta_b\mathrm{Feat}_{ibt}+e_{ibt},
\]
with store--brand intercepts, common week effects, brand-specific price and feature coefficients, and store-clustered standard errors.  The own log-price coefficients are $-3.219$ $(0.050)$, $-3.105$ $(0.057)$, and $-2.842$ $(0.053)$ for Dominicks, Minute Maid, and Tropicana.

For the chronological replay, log sales and log price are residualized by brand on store and week fixed effects and the feature indicator, then fitted as $\widetilde y_{it}=-\theta_i\widetilde x_{it}+\epsilon_{it}$.  The pooled slope, cross-store dispersion, and residual set the Gaussian prior and noise scale.  After a 24-observation reference vintage, median final standard-deviation ratios are $0.451$, $0.525$, and $0.615$; fresh-to-inherited radius ratios are $2.219$, $1.904$, and $1.626$; and median positive-region suppression indices are $15.9\%$, $13.9\%$, and $9.8\%$.

\begin{lemma}[Local-regret diagnostic and dynamic residual]\label{lem:EC-local-regret-diagnostic}
Fix a history $h_t$, a continuation policy $\pi$, and the initial-vintage evaluator.  Write the action value as
\[
Q_t^{0,\pi}(h_t,a)=g_t(h_t,a)+c_t^{\pi}(h_t,a),
\]
where $g_t$ is the one-period inherited robust payoff and $c_t^{\pi}$ contains the discounted continuation value.  Define
\[
\ell_t^{I,\pi}(h_t)
:=\max_a g_t(h_t,a)-g_t(h_t,\pi_t(h_t))
\]
and let $e_t^{0,\pi}(h_t)$ be the dynamic Bellman residual in \cref{eq:bellman-residual}.  Then
\begin{equation}\label{eq:EC-local-regret-residual-gap}
\left|e_t^{0,\pi}(h_t)-\ell_t^{I,\pi}(h_t)\right|
\le
\max_a c_t^{\pi}(h_t,a)-\min_a c_t^{\pi}(h_t,a).
\end{equation}
Consequently the two quantities coincide whenever the inherited continuation term is action independent, as in the stagewise specialization.  They need not coincide when actions change future information.
\end{lemma}
\begin{proof}
Let $M=\max_a(g_a+c_a)-\max_a g_a$.  Since $\min_a c_a\le M\le\max_a c_a$,
\[
e_t^{0,\pi}-\ell_t^{I,\pi}=M-c_{\pi_t(h_t)}
\]
lies between $\min_a c_a-c_{\pi_t(h_t)}$ and $\max_a c_a-c_{\pi_t(h_t)}$.  This proves \eqref{eq:EC-local-regret-residual-gap}.
\end{proof}

For a simulated pricing path $\omega$, Panel D reports
\[
\mathcal L_{\rm loc}^{I}(\pi;\omega)
:=\sum_{t=0}^{H-1}\beta^t\ell_t^{I,\pi}(H_t(\omega)),
\]
with $g_t$ equal to inherited one-period robust revenue.  This is a discounted local-regret diagnostic.  It is not the exact dynamic Price of Relearning because the pricing action changes posterior precision and hence the continuation term in \eqref{eq:EC-local-regret-residual-gap}.

Each held-out fold starts from the store's early posterior, estimates the reference sensitivity from the other four folds, sets $a=2\theta_{\rm ref}$, and couples policies with common shocks for 20 periods. Fresh reconstruction lowers mean discounted expected revenue by $2.34\%$, $1.34\%$, and $0.88\%$ across brands. These outcomes are semisynthetic, not causal historical effects.

\subsection{Identification boundary, sharp sensitivity, and prospective assignment}
\label{sec:EC-identification-boundary}

\begin{proof}[Proof of \cref{prop:protocol-identification-boundary}]
For part (i), let $(t,h,a)$ be outside the historical support and suppose the undiscounted date-$t$ occupancy difference between the fresh and inherited algorithms at that pair is $w=d_t^F(h,a)-d_t^I(h,a)\ne0$.  Hold every transition and reward law on the historical support fixed.  Construct two bounded conditional mean reward functions that agree everywhere except at $(t,h,a)$, where they differ by $2c\operatorname{sgn}(w)$ for some admissible $c>0$.  Because the historical policy never visits the pair, the complete distribution of historical observables is identical under the two models.  The counterfactual protocol contrast differs by $2\beta^t c|w|$.  It is therefore not point identified.  If the available reward range is large enough relative to the fitted contrast, the same construction reverses its sign.

For part (ii), the two decision rules and transition law are fixed at their fitted values, so the protocol occupancies are fixed.  The protocol effect is then linear in the conditional mean reward:
\[
\Delta
=
\sum_t\beta^t\sum_{h,a}
\{d_t^F(h,a)-d_t^I(h,a)\}\{\widehat\mu_t(h,a)+u_t(h,a)\}.
\]
Subtracting the fitted-model effect and applying $|u_t(h,a)|\le\Gamma_t(h,a)$ gives the radius $B_\Gamma$ in \cref{eq:sharp-protocol-sensitivity-set}.  The endpoints are attainable by taking
\[
u_t(h,a)=\pm\Gamma_t(h,a)
\operatorname{sgn}\{d_t^F(h,a)-d_t^I(h,a)\},
\]
with arbitrary values when the occupancy difference is zero.  Hence the interval is sharp within the stated fixed-policy, fixed-transition disturbance class.  If the policy or transition law is re-estimated after perturbing the outcome model, this linear sensitivity set need not remain sharp and no such claim is made.

For part (iii), assume well-defined whole-horizon potential outcomes, consistency of observed with assigned outcomes, no interference across stores, and independent matched-pair randomization.  Index matched pairs by $j$ and stores within a pair by $k$.  Let $Y_{jk}(F)$ and $Y_{jk}(I)$ include the complete learning path under the assigned algorithm.  Draw $Z_j$ independently with probability one half.  If $Z_j=1$, assign fresh to store 1 and inheritance to store 2; reverse the assignment if $Z_j=0$.  With $D_j^{\rm obs}=Y_{j1}^{\rm obs}-Y_{j2}^{\rm obs}$,
\[
\widehat\tau=\frac1J\sum_{j=1}^J(2Z_j-1)D_j^{\rm obs}.
\]
Taking expectation over assignment gives
\[
\E_Z[\widehat\tau]
=\frac1{2J}\sum_{j=1}^J
\{Y_{j1}(F)-Y_{j1}(I)+Y_{j2}(F)-Y_{j2}(I)\},
\]
the finite-population average protocol effect.  Assignment must be held for the full horizon because an earlier protocol changes the later learning state.
\end{proof}

Prospective planning uses
\[
J_0=\left\lceil\left\{\frac{z_{1-\alpha/2}+z_{1-\beta}}{|\delta|/s_D}\right\}^2\right\rceil,
\qquad J=\lceil D_{\rm eff}J_0\rceil.
\]
With $\alpha=0.05$, 90\% target power, and $D_{\rm eff}=4$, the model-based requirements are 24, 20, and 12 matched pairs; they are planning inputs, not results from an implemented experiment.

\setlength{\bibsep}{0pt}

\begin{center}
\centering
\captionof{table}{Supplementary computation and scanner-calibrated evidence. Panel D reports a fresh-policy inherited local-regret diagnostic, not exact dynamic PoR.}
\label{tab:EC-vintage-backward-scaling}
\label{tab:EC-vintage-compression}
\label{tab:EC-oj-replay}
\label{tab:EC-policy-value}
\scriptsize
\renewcommand{\arraystretch}{0.82}

\noindent
\begin{minipage}[t]{0.49\textwidth}
\centering
\textbf{Panel A: exact fixed-model evaluation}\par
\setlength{\tabcolsep}{2.0pt}
\begin{tabular}{@{}rrrrr@{}}
\toprule
Horizon & \shortstack{Date--state\\nodes} & \shortstack{Vintage\\values} & \shortstack{Operation\\scale} & \shortstack{Median\\seconds}\\
\midrule
10  & 73   & 520     & 11,691    & 0.0014\\
25  & 193  & 2,800   & 70,011    & 0.0064\\
50  & 393  & 10,600  & 275,211   & 0.0252\\
100 & 793  & 41,200  & 1,090,611 & 0.0985\\
200 & 1,593& 162,400 & 4,341,411 & 0.3796\\
\bottomrule
\end{tabular}
\end{minipage}%
\hfill
\begin{minipage}[t]{0.49\textwidth}
\centering
\textbf{Panel B: coherent vintage compression}\par
\setlength{\tabcolsep}{1.8pt}
\begin{tabular}{@{}rrrrrr@{}}
\toprule
Horizon & \shortstack{Recent\\classes} & \shortstack{Maximum\\stored} & \shortstack{Storage\\fraction} & \shortstack{Maximum\\error} & Bound\\
\midrule
16 & 8  & 10 & 0.815 & $5.53\times10^{-4}$ & $2.28\times10^{-3}$\\
32 & 8  & 10 & 0.507 & $5.83\times10^{-4}$ & $6.32\times10^{-3}$\\
32 & 12 & 14 & 0.660 & $4.05\times10^{-5}$ & $3.91\times10^{-4}$\\
64 & 8  & 10 & 0.281 & $6.99\times10^{-4}$ & $8.63\times10^{-3}$\\
64 & 12 & 14 & 0.381 & $4.86\times10^{-5}$ & $5.91\times10^{-4}$\\
\bottomrule
\end{tabular}
\end{minipage}

\vspace{1pt}
\textbf{Panel C: retrospective scanner-data replay}\par
\setlength{\tabcolsep}{1.8pt}
\begin{tabular}{lrrrrrr}
\toprule
Brand & $\widehat\beta$ (standard error) & \shortstack{Posterior standard\\deviation ratio} & \shortstack{Fresh-to-inherited\\radius ratio} & \shortstack{Positive\\paths} & \shortstack{Median\\suppression} & \shortstack{90th-percentile\\suppression}\\
\midrule
Dominicks & $-3.219$ (.050) & .451 & 2.219 & 61/83 & 15.9\% & 22.6\%\\
Minute Maid & $-3.105$ (.057) & .525 & 1.904 & 46/83 & 13.9\% & 21.7\%\\
Tropicana & $-2.842$ (.053) & .615 & 1.626 & 68/83 & 9.8\% & 16.6\%\\
\bottomrule
\end{tabular}

\vspace{1pt}
\textbf{Panel D: held-out-parameter policy experiment}\par
\setlength{\tabcolsep}{4.0pt}
\begin{tabular}{lrrrr}
\toprule
Brand & Stores & Revenue effect & Information effect & Fresh-policy local regret\\
\midrule
Dominicks & 61 & $-0.894$ $(-1.053,-0.734)$ & $-21.92$ & $0.492$\\
Minute Maid & 46 & $-0.486$ $(-0.575,-0.397)$ & $-30.93$ & $0.255$\\
Tropicana & 68 & $-0.330$ $(-0.371,-0.288)$ & $-43.34$ & $0.198$\\
\bottomrule
\end{tabular}
\end{center}

\end{document}